\documentclass[10pt]{amsart}

\usepackage[latin2]{inputenc}
\usepackage{amsmath}
\usepackage{graphicx}
\usepackage{amssymb}
\usepackage{esint}
\usepackage{color}
\usepackage{amsthm}
\usepackage{epsfig}
\usepackage{enumitem}
\usepackage{mathtools}
\usepackage{esvect}
\usepackage{tikz}
\usepackage{mathrsfs}
\usetikzlibrary{calc,intersections,through,backgrounds,patterns,arrows.meta, math,through}
\usepackage[english]{babel}
\usepackage[linktocpage=true,colorlinks=true,linkcolor=Blue,citecolor=Green]{hyperref}

\newtheorem{theorem}{Theorem}
\newtheorem{proposition}[theorem]{Proposition}
\newtheorem{lemma}[theorem]{Lemma}
\newtheorem{corollary}[theorem]{Corollary}
\newtheorem{construction}{Construction}
\newtheorem*{theorem*}{Theorem}

\theoremstyle{definition}
\newtheorem{definition}{Definition}
\newtheorem{remark}[definition]{Remark}

\def\Xint#1{\mathchoice
{\XXint\displaystyle\textstyle{#1}}%
{\XXint\textstyle\scriptstyle{#1}}%
{\XXint\scriptstyle\scriptscriptstyle{#1}}%
{\XXint\scriptscriptstyle\scriptscriptstyle{#1}}%
\!\int}
\def\XXint#1#2#3{{\setbox0=\hbox{$#1{#2#3}{\int}$ }
\vcenter{\hbox{$#2#3$ }}\kern-.6\wd0}}

\def\dashint{\Xint-}

\definecolor{Yellow}{rgb}{0.95,0.9,0.0} 
\definecolor{Red}{rgb}{0.8,0.1,0.1}
\definecolor{Green}{rgb}{0.1,0.65,0.2}
\definecolor{Blue}{rgb}{0.1,0.1,0.8}
\definecolor{Purple}{rgb}{0.7,0.1,0.7}
\definecolor{Grey}{rgb}{0.6,0.6,0.6}

\definecolor{YELLOW}{rgb}{0.95,0.9,0.0} 
\definecolor{RED}{rgb}{0.8,0.1,0.1}
\definecolor{GREEN}{rgb}{0.25,0.65,0.1}
\definecolor{BLUE}{rgb}{0.1,0.1,0.8}
\definecolor{PURPLE}{rgb}{0.7,0.1,0.7}

\newcommand{\BV}{\operatorname{BV}} 
\newcommand{\supp}{\operatorname{supp}} 
 
\newcommand{\sign}{\operatorname{sign}}

\newcommand{\dist}{\operatorname{dist}} 
\newcommand{\sdist}{\operatorname{sdist}} 
\DeclareMathOperator*{\esssup}{ess\,sup}

\newcommand{\Id}{\operatorname{Id}}

\newcommand{\Rd}[1][d]{{\mathbb{R}^{#1}}}
\allowdisplaybreaks[4]

\newcommand{\dH}{d\mathcal{H}^{d-1}}

\newcommand{\dx}{dx}
\newcommand{\dt}{dt}

\newcommand{\TBV}{T_{\BV}}

\newcommand{\Tan}{\mathrm{Tan}}

\renewcommand{\vec}[1]{{\operatorname{#1}}}
\newcommand{\n}{\vec{n}}

\newcommand{\tchi }{t_{\chi}}

\begin{document}

\title[Stability of multiphase MCF beyond circular topology changes]
{Stability of multiphase mean curvature flow beyond circular topology changes}

\author{Julian Fischer}
\address{Institute of Science and Technology Austria (ISTA), Am~Campus~1, 
3400 Klosterneuburg, Austria}
\email{julian.fischer@ist.ac.at}

\author{Sebastian Hensel}
\address{Hausdorff Center for Mathematics, Universit{\"a}t Bonn, Endenicher Allee 62, 53115 Bonn, Germany}
\email{sebastian.hensel@hcm.uni-bonn.de}

\author{Alice Marveggio}
\address{Institute of Science and Technology Austria (ISTA), Am~Campus~1, 
	3400 Klosterneuburg, Austria
{\bf \it Current address:} Hausdorff Center for Mathematics, Universit{\"a}t Bonn, Endenicher Allee 60, 53115 Bonn, Germany}
\email{alice.marveggio@hcm.uni-bonn.de}

\author{Maximilian Moser}
\address{Institute of Science and Technology Austria (ISTA), Am~Campus~1, 
	3400 Klosterneuburg, Austria}
\email{maximilian.moser@ist.ac.at}

\begin{abstract}
We prove a weak-strong uniqueness principle for varifold-BV solutions to planar multiphase mean curvature flow beyond a circular topology change: Assuming that there exists a classical solution with an interface that becomes increasingly circular and shrinks to a point, any varifold-BV solution with the same initial interface must coincide with it, and any varifold-BV solution with similar initial data must undergo the same type of topology change.
Our result illustrates the robustness of the relative energy method for establishing weak-strong uniqueness principles for interface evolution equations, showing that it may also be applied beyond certain topological changes.
\end{abstract}


\maketitle
\tableofcontents

\section{Introduction}
For two-phase mean curvature flow, a weak solution concept with highly satisfactory properties is available in form of viscosity solutions \cite{EvansSpruck,ChenGigaGoto} to the level set formulation \cite{OsherSethian,OhtaJasnowKawasaki}: Not only can global-in-time existence of weak solutions be shown for general initial data, but failure of uniqueness may be characterized in quite detail \cite{BarlesSonerSouganidis}, and uniqueness of weak solutions is guaranteed as long as a classical solution exists. However, the concept of viscosity solutions crucially relies on the availability of the comparison principle, restricting it to (mostly) mean curvature flow in the context of interface evolution problems.

For interface evolution problems without comparison principle such as multiphase mean curvature flow or higher-order curvature driven flows, the question of uniqueness of weak solutions -- even in the absence of topology changes -- had remained open for a long time. In fact, solution concepts such as Brakke solutions for mean curvature flow \cite{Brakke} have even been known to admit artificial (unphysical) solutions, and attempts to develop solution concepts without these shortcomings have been made \cite{LauxOtto,KimTonegawa,StuvardTonegawa}. Recently, an approach based on relative energies has proven successful in establishing weak-strong uniqueness prior to singularities \cite{FischerHensel,FischerHenselLauxSimon,FischerHenselLauxSimonMS,HenselLaux}, as well as in deriving sharp-interface limits of phase-field models \cite{FischerLauxSimon,AbelsFischerMoser,FischerMarveggio,HenselMoser}.

So far, these weak-strong uniqueness results for interface evolution problems have been limited to situations without geometric singularities and in particular without topology changes. In the present work we show that the approach of relative energies is robust and capable of handling certain controlled topology changes in the (piecewise-in-time) strong solution: We show that in the case of multiphase mean curvature flow, a weak-strong uniqueness and stability principle holds also beyond shrinking circle type singularities. More precisely, for any classical solution to planar curvature flow whose interface consists of a smooth simple curve that shrinks, becomes increasingly circular, and disappears, any weak solution with similar initial data must stay close to it and disappear in the same kind of singularity. As in \cite{FischerHenselLauxSimon}, the weak solutions we consider are varifold-BV solutions in the sense of Stuvard-Tonegawa \cite{StuvardTonegawa}.

Recall that the classical result of Gage--Hamilton \cite{GageHamilton} and Grayson \cite{Grayson1987} asserts that any smooth, closed, and simple curve in the plane evolving by mean curvature flow (MCF) shrinks to a point in finite time, becoming increasingly circular in the process.
Combining this classical result with our main result, we recover a perturbative but genuinely multiphase version of the Gage--Hamilton--Grayson theorem: If the initial interface of a varifold-BV solution to mean curvature flow is sufficiently close to a smooth, closed, simple curve (in the sense of the relative energy distance, that is, in a tilt-excess-type distance), it will over time become increasingly circular and eventually disappear in a shrinking circle type singularity. In particular, this conclusion remains valid even if initially a small amount of other phases are present in the varifold-BV solution.

\section{Main result}
To state our main result, we first recall the notion of relative energy of a varifold-BV solution; note that the latter consist of a time-indexed family of varifolds $\mathcal{V}_t$ and an indicator function $\chi_i(\cdot,t)$ for each phase. In \cite{FischerHenselLauxSimon} the relative energy of a varifold-BV solution with respect to a strong solution $(\bar \chi_i)_{1\leq i\leq P}$ was defined as
\begin{align}
\label{RelativeEnergy}
E_{rel}(t)
:=& \frac{1}{2} \sum_{i=1}^P \sum_{j=1}^P \int_{I_{i,j}(t)} 1-\vec{n}_{i,j}(\cdot,t) \cdot \xi_{i,j}(\cdot,t) \,d\mathcal{H}^1
\\&
\nonumber
+\int_{\mathbb{R}^2 \times \mathbb{S}^1} 1 - \omega(x,t) \,d\mathcal{V}_t(x,s)
\end{align}
where $\xi_{i,j}(\cdot,t)$ denotes a suitable extension of the unit normal vector field of the interface between phases $i$ and $j$ in the strong solution, $I_{i,j}(t):=\partial^*\{\chi_i(\cdot,t) = 1\} \cap \partial^*\{\chi_j(\cdot,t)=1\}$ denotes the interface between phases $i$ and $j$ in the varifold-BV solution, $\vec{n}_{i,j}$ is its corresponding unit normal vector, and $\omega(\cdot,t)\in [0,1]$ is the local ratio between the surface measure $\frac{1}{2}\sum_{i=1}^P |\nabla \chi_i|(\cdot,t)$ and the weight measure $\mu_t(\cdot):=\int_{\mathbb{S}^1} d\mathcal{V}_t(\cdot,s)$. For our present results the relative energy will share the same structure as in \eqref{RelativeEnergy}, except that as in \cite{FischerHenselLauxSimon} we also add a lower-order term for coercivity; it is merely the vector fields $\xi_{i,j}$ that will be modified suitably (in a way that corresponds to simply shifting the strong solution $\bar \chi$ in space and time).

Observe that the relative energy \eqref{RelativeEnergy} measures the mismatch between the classical solution and the varifold-BV solution in a tilt-excess-type way; furthermore, the second term on the right-hand side of \eqref{RelativeEnergy} measures the mismatch in multiplicity between the varifold $\mathcal{V}_t$ and the surface measure $\frac{1}{2}\sum_{i=1}^P |\nabla \chi_i|(\cdot,t)$.

Recall that the general goal of weak-strong uniqueness proofs via relative energy methods is to establish a Gronwall-type estimate $\frac{d}{dt} E_{rel} \leq C E_{rel}$, which enables one to conclude.
However, previous weak-strong stability results of this form
(e.g., \cite{FischerHenselLauxSimon} and~\cite{HenselLaux}) have been limited
to time horizons before the first topology change of the strong solution:
The reason is that for typical topology changes such as the circular topology change considered in the present work, a naive estimation of the terms on the right-hand side of the relative energy inequality would lead to a Gronwall estimate of the form $\frac{d}{dt} E_{rel} \leq C(t) E_{rel}$ with $C(t)\sim \frac{1}{|T-t|}$. Note that the time-dependent constant $C(t)$ is borderline non-integrable, leading to a loss of any assertion on stability past the topology change. This suggests that in order to deal with topology changes, a more refined estimation is needed to control the right-hand side of the relative energy estimate, e.\,g., by combining the relative entropy approach with a linearized stability analysis.

A linearized stability analysis for the relative energy \eqref{RelativeEnergy} beyond a circular topology change however reveals the presence of two unstable modes and one borderline stable mode. It turns out that the unstable modes correspond to translational degrees of freedom, while the borderline stable mode corresponds to a shift in time (see Figure~\ref{Fig:spacetime} and Figure~\ref{Fig:spacetime2} for a more detailed explanation).
We overcome this issue of unstable modes by developing a weak-strong stability theory for circular topology change up to dynamic shift, which amounts to dynamically
adapting the strong solution to the weak solution to a degree which
takes care of the leading-order non-integrable contributions in 
the Gronwall estimate.

The precise statement of our main result reads as follows.

\begin{theorem}[Weak-strong stability up to shift for circular topology change]
	\label{theo:mainResultC} 
	Let $d = 2$ and $P\geq 2$. Consider a global-in-time varifold-BV solution $(\mathcal{V}, \chi)$ with $\chi=(\chi_1,\ldots,\chi_P)$  (or a BV solution $\chi=(\chi_1,\ldots,\chi_P)$ )
	to multiphase MCF in the sense of Definition~\ref{DefinitionVarSolution} (or Definition~\ref{DefinitionWeakSolution}).
	Consider also a smoothly evolving two-phase strong solution
	to MCF $\bar \chi = (\bar \chi_1,\ldots,\bar\chi_P\equiv 1 {-} \bar\chi_1)$
	with extinction time $T_{ext} =: \smash{\frac{1}{2}}r_0^2 > 0$.
	Fix $\alpha \in (1,5)$.
	
	There exists $\delta_{\mathrm{asymp}} \ll_\alpha \frac{1}{2}$ such that if
	for all $t \in (0,T_{ext})$ the interior of the phase $\{\bar\chi_1(\cdot,t) {=} 1\} \subset \mathbb{R}^2$ is 
	$\delta_{\mathrm{asymp}}$-close to a circle with radius $r(t):=\sqrt{2(T_{ext}{-}t)}$
	in the sense of Definition~\ref{DefinitionShrinkingCircle},
	the evolution of~$\bar\chi$ is \emph{unique and stable until the extinction time $T_{ext}$ modulo shift} 
	in the following sense:
	
	There exists $\delta \ll 1$ as well as an error functional 
	$E[\mathcal{V}_0,\chi_0|\bar\chi_0] \in [0,\infty)$ for the initial 
	data $(\mathcal{V}_0,\chi_0)$ and $\bar\chi_0$
	of $(\mathcal{V},\chi)$ and $\bar\chi$, respectively, 
	such that if
	\begin{align} \label{ineq:initialhp}
		E[\mathcal{V}_0,\chi_0|\bar\chi_0] < \delta r_0,
	\end{align}
	one may then choose 
	\begin{align*}
		&\text{a time horizon } \tchi > 0, \\
		&\text{a path of translations } z \in W^{1,\infty}((0,\tchi);\mathbb{R}^2), \text{ and}\\
		&\text{a strictly increasing bijection }
		T \in W^{1,\infty}((0,t_\chi);(0,T_{ext}))
	\end{align*}
	with the properties $(z(0), T(0))=(0,0)$, 
	\begin{align}
		\label{eq:upperBoundTranslation}
		\frac{1}{r_0}\| z \|_{L^\infty_t(0,\tchi)} &\leq 
		\sqrt{\frac{1}{r_0}E[\mathcal{V}_0,\chi_0|\bar\chi_0]}, \\
		\label{eq:upperBoundTimeDilation}
		\frac{1}{T_{ext}}	\| T - \mathrm{id}\|_{L^\infty_t(0,\tchi)} &\leq 
		\sqrt{\frac{1}{r_0}E[\mathcal{V}_0, \chi_0|\bar\chi_0]},
	\end{align}
	such that for a.e.\ $t \in (0,\tchi)$ it holds
	\begin{align} \label{ineq:decayE}
		E[\mathcal{V}, \chi|\bar\chi^{z,T}](t) \leq 
		E[\mathcal{V}_0,\chi_0|\bar\chi_0] \Big(\frac{r_T(t)}{r_0}\Big)^\alpha
	\end{align}
	where $\bar \chi^{z,T}(x,t) := \bar \chi(x {-} z(t), T(t))$, $(x,t) \in \mathbb{R}^2{\times}[0,\tchi)$,
	denotes the \emph{shifted strong solution}, $r_T(t) := r(T(t))$ for $t \in [0,\tchi)$,
	and $E[\mathcal{V}, \chi|\bar\chi^{z,T}](t) $ is an error functional
	satisfying
	\begin{align}
		\label{eq:basicCoercivity}
		E[\mathcal{V},\chi|\bar\chi^{z,T}](t)  = 0
		\quad\Longleftrightarrow\quad
			\begin{cases}
			\chi(\cdot,t) = \bar\chi^{z,T}(\cdot,t) \quad\mathcal{H}^2 \text{-a.e.\ in } \mathbb{R}^2,
			\, \,\\ \mu_t = \frac12 \sum_{i=1}^P |\nabla \bar \chi   (\cdot,t)|
			\quad\mathcal{H}^1 \text{-a.e.\ in } \mathbb{R}^2.
		\end{cases}
	\end{align}
	In particular, 
	under the assumption of~\eqref{ineq:initialhp},
	the varifold-BV~solution $(\mathcal{V},\chi)$ goes extinct 
	and the associated time horizon~$\tchi$ provides an upper bound
	for the extinction time.
\end{theorem}

We phrased our main result in a form emphasizing the main
contribution of this work, i.e., stability of the evolution
for times close to a circular topology change (formalized above
by means of the notion of quantitative closeness to a shrinking circle,
see Definition~\ref{DefinitionShrinkingCircle} for details). 
One may also derive a corresponding stability estimate starting 
from initial data not entailing an approximately self-similarly
evolving solution at early times. 

\begin{remark}
	Consider a smoothly evolving two-phase solution
	to mean curvature flow $\bar \chi = (\bar \chi_1 ,\ldots,\bar\chi_P\equiv 1 {-} \bar\chi_1)$
	with initial data $\bar\chi_{0,1}=\chi_{\mathcal{A}_0}$ for some
	smooth, bounded, open and simply connected initial set $\mathcal{A}_0\subset\mathbb{R}^2$.
	By the Gage--Hamilton--Grayson theorem (\cite{GageHamilton},\cite{Grayson1987}), 
	the solution goes extinct at time $T_{ext}=\smash{\frac{\mathrm{vol}(\mathcal{A}_0)}{\pi}}$,
	and for any given $\delta_{\mathrm{asymp}} \in (0,1)$, there exists a time 
	$t_0=t_0(\mathcal{A}_0,\delta_{\mathrm{asymp}}) < T_{ext}$ 
	such that for all $t \in [t_0,T_{ext})$ it holds that the interior of $\{\bar\chi_1(\cdot,t) = 1\}$
	is $\delta_{\mathrm{asymp}}$-close to a circle with radius $r(t):=\sqrt{2(T_{ext}{-}t)}$
	in the sense of Definition~\ref{DefinitionShrinkingCircle}. 
	
	In particular, from some time onwards one is in the asymptotic regime
	close to the extinction time for which the conclusions of Theorem~\ref{theo:mainResultC}
	apply, at least if at time $t_0=t_0(\mathcal{A}_0,\delta_{\mathrm{asymp}})$ the assumption~\eqref{ineq:initialhp}
	on the smallness of the initial error is satisfied (i.e., with respect to $r(t_0)$).
	Based on the weak-strong stability estimate prior to topology changes from~\cite{FischerHenselLauxSimon},
	this requirement can be translated into a condition at the initial time $t=0$:
	there exists a constant $\mu_0=\mu_0(t_0,\mathcal{A}_0)>0$ such that if $E[\chi_0|\bar\chi_0] < \frac{1}{\mu_0}\delta r_0$
	then $E[\chi|\bar\chi](t_0) < \delta r(t_0)$.
	
	In summary, for general initial data~$\mathcal{A}_0$ as considered in this remark, 
	one first has, thanks to the main result of~\cite{FischerHenselLauxSimon},
	at least stability in the sense of $\frac{d}{dt} E[\chi|\bar\chi](t) \leq C(t) E[\chi|\bar\chi](t)$ for times
	$t \in (0,t_0)$ where $C(t) \sim (2(T_{ext}{-}t))^{-1} = r(t)^{-2}$.
	Then, if $E[\chi_0|\bar\chi_0] < \frac{1}{\mu_0}\delta r_0$, in addition the decay estimate~\eqref{ineq:decayE} from Theorem~\ref{theo:mainResultC}
	holds true for all times in the asymptotic regime $(t_0,T_{ext})$. 
	\hfill$\diamondsuit$
\end{remark}

Before we recall the precise definitions of the two weak solution concepts
to which our main result applies, we provide two comments on the latter.

\begin{itemize}[leftmargin=0.7cm]
\item 
	First, note that the decay exponent $\alpha < 5$ in our stability estimate \eqref{ineq:decayE} 
	is optimal in the sense that it is consistent with the results obtained by Gage--Hamilton~\cite{GageHamilton}.
	More precisely, from \cite[Corollary 5.7.2]{GageHamilton} one can read off that for
	a smooth, closed and simple curve $\partial\mathcal{A}(t)$ shrinking by MCF
	to the origin $x=0$, it holds
	asymptotically as $t \uparrow T_{ext}$ that
	$\sup_{\partial\mathcal{A}(t)} |\nabla^\mathrm{tan} H_{\partial\mathcal{A}(t)}| 
	\lesssim r(t)^{- \tilde\alpha}$, for any $0<\tilde \alpha \ll 1$. Since $H_{\partial\mathcal{A}}:=-\nabla^\mathrm{tan}\cdot \n_{\partial\mathcal{A}}$, by dimensional analysis, one then expects
	from the fact that our error functional behaves like a tilt excess that one gets
	decay for any exponent $\alpha =5- \tilde\alpha$, $0<\tilde \alpha \ll 1$.
\item Second, note that the error bounds \eqref{eq:upperBoundTranslation}--\eqref{eq:upperBoundTimeDilation} 
	on the space-time shift $(z,T)$ are optimal in terms of the scaling 
	$\sqrt{\frac{1}{r_0}E[\mathcal{V}_0,\chi_0|\bar\chi_0]}$. For example, 
	let $0<\delta \ll 1 $ and consider, next to a shrinking circle
	with initial radius~$r_0$, a shrinking circle with initial radius $(1+\sqrt{\delta})r_0$
	(both centered at the origin). Note that the initial error between the two solutions
	indeed satisfies our assumption~\eqref{ineq:initialhp},
	cf.\ \eqref{eq:bulkPerturbativeRegime}--\eqref{eq:relEntropyPerturbativeRegime}.
	Since the relative error between the two extinction times is $\sim \sqrt{\delta}$,
  this shows the claim for~\eqref{eq:upperBoundTimeDilation}. Shifting instead
	a shrinking circle with initial radius~$r_0$ initially by $\sqrt{\delta}r_0 v$, $v \in \mathbb{S}^1$,
	in turn illustrates the claim for~\eqref{eq:upperBoundTranslation}.
\end{itemize}

\begin{figure}
	\includegraphics{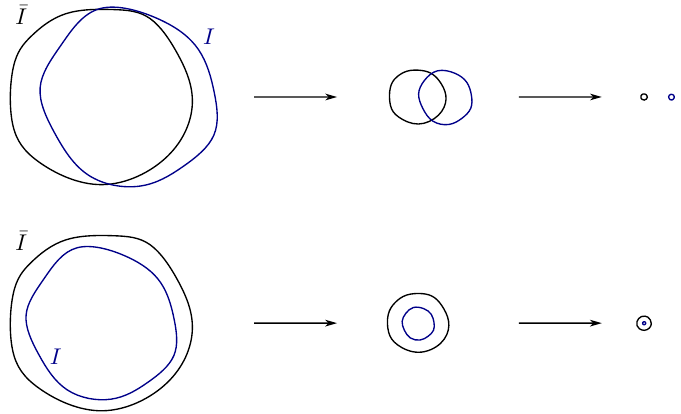}
	\caption{First case - Motivation for space shift $z$: $\bar I = \partial \{\bar \chi_1 = 1\}$ and $I = \partial^\ast \{\chi_1 = 1\}$  simultaneously shrink to two distinct points which are shifted by $z$.
	Second case - Motivation for time shift $T$: $\bar I = \partial \{\bar \chi_1 = 1\}$ and $I = \partial^\ast \{\chi_1 = 1\}$ shrink to the same point but at distinct times.}
	\label{Fig:spacetime}
\end{figure}

\begin{figure}
	\includegraphics{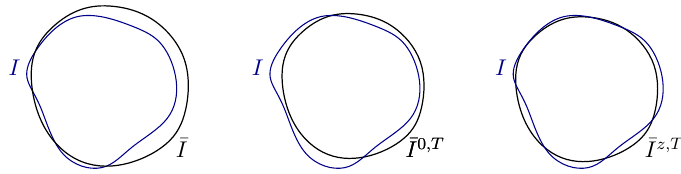}
	\caption{The interface $\bar I = \partial \{\bar \chi_1 = 1\}$ is dynamically adapted to $I = \partial^\ast \{\chi_1 = 1\}$ by means of the space-time shift $(z,T)$, namely $\bar I^{z,T} = \partial \{\bar \chi^{z,T}_1 = 1\}$.}
	\label{Fig:spacetime2}
\end{figure}

Our arguments to prove weak-strong stability up to dynamic shift for circular topology change work for both BV solutions in the sense of Laux-Otto-Simon (\cite{LauxOtto}, \cite{Laux2020}, \cite{LauxSimon}) and for more general varifold-BV solutions, recently introduced by Stuvard and Tonegawa in \cite{StuvardTonegawa}.
Here we recall the definitions of both of these weak 
solution concepts (cf. \cite[Definition 12]{FischerHenselLauxSimon} and \cite[Definition 18]{FischerHenselLauxSimon}).

\begin{definition}[BV solution to multiphase MCF]
	\label{DefinitionWeakSolution}
	Let $d = 2$ and $P \geq 2$. A measurable map
	\begin{align*}
		\chi=(\chi_1,\ldots,\chi_P)\colon\mathbb{R}^2 \times [0,\infty) \to \{0,1\}^P
	\end{align*}
	(or the corresponding tuple of sets $\Omega_i(t):=\{\chi_i(t) =1\}$ for $i=1,\ldots, P$)
	is called a \emph{global-in-time BV~solution to multiphase MCF} with
	initial data $\chi_0=(\chi_{0,1},\ldots,\chi_{0,P})\colon$ $\mathbb{R}^2 \to \{0,1\}^P$
	(of finite interface energy in the sense of~\cite[Definition~12]{FischerHenselLauxSimon})
	if the following conditions are satisfied:
	\begin{itemize}[leftmargin=0.7cm]
		\item[i)] For any $T_{\mathrm{BV}} \in (0,\infty)$, $\chi$ is a BV~solution to multiphase~MCF
		on $[0,T_{\mathrm{BV}})$ with initial data~$\chi_0$ in the sense of~\cite[Definition~13]{FischerHenselLauxSimon}
		(with trivial surface tension matrix $\sigma=\mathrm{diag}(1,\ldots,1)\in\mathbb{R}^{P\times P}$) such that
		\begin{itemize}[leftmargin=0.7cm]
			\item[i.a)] (Partition with finite interface energy) For almost every $T\in [0,\TBV)$, $\chi(T)$
			is a partition of $\Rd[2]$ with interface energy 
			\begin{align}\label{EnergyFunctionalBVPartition}
				E[\chi] := 
				\frac{1}{2}
				\sum_{i,j=1,i\neq j}^P 
				\mathcal{H}^1(I_{i,j}(t))
			\end{align}
			such that
			\begin{align}\label{GlobalEnergyBoundBVSolution}
				\esssup_{T\in [0,\TBV)} E[\chi(\cdot,T)] 
				 < \infty,
			\end{align}
			where $I_{i,j}(t)=\partial^*\{\chi_i(\cdot,t) = 1\} \cap \partial^*\{\chi_j(\cdot,t)=1\}$
			for $i\neq j$ is the interface between 
			the $i$-th and the $j$-th phase. We also define $\n_{i,j}(\cdot,t) 
			:= -\frac{\nabla\chi_i(\cdot,t)}{|\nabla\chi_i(\cdot,t)|}
			= \frac{\nabla\chi_j(\cdot,t)}{|\nabla\chi_j(\cdot,t)|}$
			abeing the unit normal vector field along~$I_{i,j}(t)$
			pointing from the $i$-th to the $j$-th phase.
			\item[i.b)] (Evolution equation) For all $i\in \{1,\ldots,P\}$, there exist normal velocities 
			$V_i\in L^2(\Rd[2]\times [0,\TBV),|\nabla\chi_i|\otimes\mathcal{L}^1)$
			in the sense that each $\chi_i$ satisfies the evolution equation
			\begin{align}\label{EvolutionPhasesBVSolution}
				\nonumber
				&\int_{\Rd[2]} \chi_i(\cdot,T) \varphi(\cdot,T) \, \dx
				-\int_{\Rd[2]} \chi_{0,i} \varphi(\cdot,0)  \, \dx
				\\&
				=\int_0^T \int_{\Rd[2]} V_i \varphi  \, \mathrm{d}|\nabla\chi_i| \dt
				+ \int_0^T \int_{\Rd[2]} \chi_i \partial_t \varphi  \, \dx \dt
			\end{align}
			for almost every $T\in [0,\TBV)$ and all $\varphi\in C^\infty_{\mathrm{cpt}}(\Rd[2] \times [0,\TBV])$.
			Moreover, the (reflection) symmetry condition 
			$V_i \smash{\frac{\nabla \chi_i}{|\nabla \chi_i|}}=V_j  \smash{\frac{\nabla \chi_j}{|\nabla \chi_j|}}$
			shall hold $\mathcal{H}^{1}$-almost everywhere on 
			the interfaces $I_{i,j}$ for $i\neq j$.
			\item[i.c)] (BV formulation of mean curvature) The normal velocities satisfy the equation
			\begin{align}\label{BVFormulationMeanCurvature}
				\nonumber
				&\sum_{i,j=1,i\neq j}^P \int_0^{\TBV}\int_{I_{i,j}(t)}V_i \frac{\nabla \chi_i}{|\nabla \chi_i|} \cdot \vec{B} 
				\,\mathrm{d}\mathcal{H}^{d-1} \dt
				\\&
				= \sum_{i,j=1,i\neq j}^P \int_0^{\TBV}\int_{I_{i,j}(t)} 
				\bigg(\Id-\frac{\nabla \chi_i}{|\nabla \chi_i|}\otimes \frac{\nabla \chi_i}{|\nabla \chi_i|}\bigg) : \nabla \vec{B} 
				\,\mathrm{d}\mathcal{H}^{d-1} \dt
			\end{align}
			for all $B\in C^\infty_{\mathrm{cpt}}(\Rd[2] \times [0,\TBV];\Rd)$.
		\end{itemize}
		
		\item[ii)] {For all} $[s,\tau] \subset [0,\infty)$, 
		the energy dissipation inequality 
		\begin{align}
			\label{eq:energyDissip}
			E[\chi(\cdot,\tau)] + \int_{s}^\tau
			\sum_{i,j=1,i\neq j}^{P} \int_{I_{i,j}(t)}
			\frac{1}{2}|V_i|^2 \,d\mathcal{H}^1dt
			\leq E[\chi(\cdot,s)]
		\end{align}
		holds true, and in addition more generally the corresponding Brakke inequality in
		the BV-setting~\cite[Definition~2.1]{Laux2020}.
		\hfill$\diamondsuit$
	\end{itemize} 
\end{definition}

\begin{definition}[Varifold-BV solution to multiphase MCF,
cf.\ \cite{StuvardTonegawa}] \label{DefinitionVarSolution}
	Let $\mathcal{V} = (\mathcal{V}_t)_{t\in(0,\infty)}$ be a measurable family of integral and rectifiable
	$(d{-}1)$-varifolds; denote by $(\mu_t)_{t\in (0,\infty)}$
	the associated family of weight measures. 
	Let $(\chi_1, . . . , \chi_P )\colon \mathbb{R}^d\times [0, \infty) \rightarrow\{0, 1\}^P$
	denote a family of indicator functions of sets
	with bounded perimeter subject to the properties in item i.a) of Definition \ref{DefinitionWeakSolution}.
	
	The tuple $(\mathcal{V}, \chi)$ is called a \emph{global-in-time varifold-BV~solution to multiphase MCF} with
		initial data $(\mathcal{V}_0,\chi_0)$, where $\mathcal{V}_0$ 
		is an integral and rectifiable $(d{-}1)$-varifold and~$\chi_0$
		is as in Definition \ref{DefinitionWeakSolution}, if the following conditions are satisfied:
	\begin{itemize}
		\item[i)]  For a.e.\ $t \in [0, \infty)$, there exists a generalized mean curvature vector $H_\mu(\cdot, t) \in
		L^2 (\mathbb{R}^d, \mu_t)$ of $\mathcal{V}_t$ in the sense that
		\begin{align}\label{eq:Hvarifold}
			- \int_{\mathbb{R}^d} H_\mu \cdot B \, \mathrm{d} \mu_t = \int_{\mathbb{R}^d \times \mathbf{G}(d, d-1)} \Id_{\mathbf{G}(d, d-1)} : \nabla B \, \mathrm{d} \mathcal{V}_t
		\end{align}
		for all $B \in C^\infty_{cpt} (\Rd; \Rd)$, where $\mathbf{G}(d, d-1)$ denotes the space of all 
		$(d{-}1)$-dimensional linear subspaces of $\Rd$. 
		
		\item[ii)]The family of varifolds $\mathcal{V}$ is a Brakke solution to multiphase mean curvature
		flow (cf.\ \cite[Definition 2.1]{StuvardTonegawa}). 
		Furthermore, {for all} $[s,\tau] \subset [0,\infty)$
		the global energy dissipation estimate
		\begin{align}
			\mu_\tau (\Rd) + \int_s^\tau \int_{\Rd} |H_\mu|^2 \, \mathrm{d} \mu_t \leq \mu_s(\Rd)
		\end{align}
		holds true. 	
		
		\item[iii)] For a.e.~$t \in (0, \infty)$, the varifold $\mathcal{V}_t$  describes the interfaces $\partial^\ast \{\chi_i(\cdot, t) = 1\}$ in the sense that
		\begin{align} \label{eq:varsol3}
			\frac12 \sum_{i=1}^P |\nabla \chi_i (\cdot, t)| \leq \mu_t.
		\end{align}
		\item[iv)] The indicator functions $\chi_i$ evolve according to the mean curvature of $\mathcal{V}$  in
		the sense that
		\begin{align}
			\partial_t \chi_i + H_\mu \cdot \nabla \chi_i = 0
		\end{align}
		holds distributionally for all $i \in \{1,...,P\}$. 	\hfill$\diamondsuit$
	\end{itemize}
\end{definition}

We finally formalize the notion of being quantitatively close to a circle.

\begin{definition}[Quantitative closeness to circle]
	\label{DefinitionShrinkingCircle}
	Let $\mathcal{A} \subset \mathbb{R}^2$ be a bounded, open and simply connected set
	with $C^\infty$ boundary $\partial\mathcal{A}$. Fix two constants $\delta_{\mathrm{asymp}} \in (0,\frac{1}{2})$
	and $r > 0$. We refer to $\mathcal{A}$ as \emph{$\delta_{\mathrm{asymp}}$-close to a circle with radius~$r$}	
	if there exists an arc-length parametrization $\gamma\colon[0,L) \to \mathbb{R}^2$ of~$\partial\mathcal{A}$
	such that $\smash{\frac{1}{2}}r$ is a tubular neighborhood width of $\partial\mathcal{A}$ and
	\begin{align}
		\label{eq:asymptotically_circular1}
		\frac{1}{2\pi r}|L - 2\pi r| &\leq \delta_{\mathrm{asymp}},
		\\
		\label{eq:asymptotically_circular2}
		\sup_{\theta \in [0,L)} 
		\big|\n_{\partial\mathcal{A}}\big(\gamma(\theta)\big) - (-e^{2\pi i \frac{\theta}{L}})\big|
		&\leq \delta_{\mathrm{asymp}},
		\\
		\label{eq:asymptotically_circular3}
		\sup_{\theta \in [0,L)} 
		r \Big|H_{\partial\mathcal{A}}\big(\gamma(\theta)\big) - \frac{1}{r}\Big| &\leq \delta_{\mathrm{asymp}},
		\\
		\label{eq:asymptotically_circular4}
		\sup_{\theta \in [0,L)} 
		r^2 \big|(\nabla^\mathrm{tan} H_{\partial\mathcal{A}})\big(\gamma(\theta)\big)\big| &\leq \delta_{\mathrm{asymp}},
	\end{align}
	where $\n_{\partial\mathcal{A}}$ denotes the unit normal vector field along~$\partial\mathcal{A}$
	pointing inside~$\mathcal{A}$ and $H_{\partial\mathcal{A}}:=-\nabla^\mathrm{tan}\cdot \n_{\partial\mathcal{A}}$
	is the associated scalar mean curvature of~$\partial\mathcal{A}$. 
	\hfill$\diamondsuit$
\end{definition}

\paragraph{\textbf{Notation and some elementary differential geometry}} 
For the smoothly evolving $\bar\chi$, we write $\n_{\bar{I}}(\cdot,t)$ for the unit normal vector field
of~$\bar{I}(\cdot,t):=\partial\{\bar\chi_1(\cdot,t){=}1\}$ 
pointing inside $\{\bar\chi_1(\cdot,t){=}1\}$,
and also define a tangent vector field through 
$\tau_{\bar{I}}(\cdot,t) := J^{-1}\n_{\bar{I}}(\cdot,t)$
with $J \in \mathbb{R}^{2\times 2}$
being counter-clockwise rotation by $90^\circ$, $t \in (0,T)$.
Curvature is defined by $H_{\bar{I}}(\cdot,t) := - \nabla^\mathrm{tan} \cdot \n_{\bar{I}}(\cdot,t)$ for $t \in (0,T_{ext})$.
In particular, it holds
\begin{align}
	\nabla^\mathrm{tan} \n_{\bar{I}} &= -H_{\bar{I}} \tau_{\bar{I}} \otimes\tau_{\bar{I}},
	&&\nabla^\mathrm{tan} \tau_{\bar{I}} = H_{\bar{I}}\n_{\bar{I}} \otimes \tau_{\bar{I}}.
\end{align}

Within the tubular neighborhood $\{x\in\mathbb{R}^2\colon \dist(x,\bar{I}(t)) < r(t)/2\}$,
the nearest-point projection onto~$\partial\{\bar\chi(\cdot,t){=}1\}$ is denoted
by $P_{\bar I}(\cdot,t)$, whereas we write $\sdist_{\bar I}(\cdot,t)$ for the signed
distance function, with orientation fixed through the requirement 
$\nabla \sdist_{\bar I}(\cdot,t)|_{\bar I} = \n_{\bar{I}}(\cdot,t)$, $t \in (0,T_{ext})$.

Given a map $f\colon \mathbb{R}^2 \times [0,T_{ext}) \to \mathbb{R}^m$
(or $f\colon \bigcup_{t\in [0,T_{ext})}\bar{I}(t){\times}\{t\} \to \mathbb{R}^m$), we will use the notation $f^{z,T}$ 
to refer to the space-time shifted function $\mathbb{R}^2 \times (0,\tchi) \ni (x,t) 
\mapsto f(x {-} z(t), T(t)) \in \mathbb{R}^m$  
(or in the other case $\bigcup_{t\in [0,\tchi)}(z(t){+}\bar{I}(T(t)){\times}\{t\} 
\ni (x,t) \mapsto f(x {-} z(t), T(t)) \in \mathbb{R}^m$)
for any $\tchi \in (0,\infty)$, 
$z\colon [0,\tchi) \to \mathbb{R}^2$ 
and $T\colon	 [0, \tchi) \to [0,T_{ext})$.
We also define $\bar{I}:=\bigcup_{t\in [0,T_{ext})}\bar{I}(t){\times}\{t\}$.

The shifted geometry itself will be abbreviated by $\bar{I}^{z,T}(t) := z(t) + \bar{I}(T(t))$, $t \in (0,\tchi)$,
and analogously for an associated arc-length parametrization~$\bar\gamma(\cdot,t)$ of $\bar{I}(\cdot,t)$:
$\bar\gamma^{z,T}(\cdot,t) := z(t) + \bar\gamma(\cdot,T(t))$, $t \in (0,\tchi)$.
We also write $\bar{I}^{z,T}:=\bigcup_{t\in [0,t_\chi)}\bar{I}^{z,T}(t){\times}\{t\}$.
Note then that
\begin{align}
	\label{eq:shiftedSignedDistance}
	\sdist_{\bar I}^{z,T}(\cdot,t) = \sdist_{\bar I^{z,T}}(\cdot,t),
\end{align}
and thus as a direct consequence
\begin{align}
	\label{eq:shiftedNormal}
	\n^{z,T}_{\bar{I}}(\cdot,t) = \n_{\bar{I}^{z,T}}(\cdot,t).
\end{align}
Indeed, the former simply follows from
\begin{align*}
	\sdist_{\bar I}(\cdot,t) &= \dist\big(\cdot,\mathbb{R}^2\setminus\{\bar\chi_1(\cdot,t){=}1\}\big) 
	- \dist\big(\cdot,\{\bar\chi_1(\cdot,t){=}1\}\big).
\end{align*} 
Furthermore, within the tubular neighborhood
$\{x\in\mathbb{R}^2\colon \dist(x{-}z(t),\bar{I}(T(t)))<r(T(t)	)/2\}
= \{x\in\mathbb{R}^2\colon \dist(x,\bar{I}^{z,T}(t))<r(T(t)	)/2\}$ it holds
\begin{align}
	\label{eq:shiftedProjection}
	P_{\bar{I}}^{z,T}(\cdot,t) = -z(t) + P_{\bar{I}^{z,T}}(\cdot,t).
\end{align}

Finally, for simplicity, 
we will denote  $\frac{d}{\mathrm{d} t} f $ by $\dot{f}$.

\section{Overview of the strategy}
For the rest of the paper, we consider the more general framework of varifold-BV solutions. 
In particular, it follows that all the results hold also for BV solutions.

We fix a global-in-time varifold-BV solution $(\mathcal{V},\chi=(\chi_1,\ldots,\chi_P))$ 
to (planar) multiphase MCF in the sense of Definition~\ref{DefinitionVarSolution} as well as
a smoothly evolving two-phase solution
to MCF $\bar \chi = (\bar \chi_1 ,\ldots,\bar\chi_P\equiv 1 {-} \bar\chi_1)$
with extinction time $T_{ext} =: \smash{\frac{1}{2}}r_0^2 > 0$. 
We also assume
that for all $t \in (0,T_{ext})$ the interior of the phase $\{\bar\chi_1(\cdot,t) {=} 1\} \subset \mathbb{R}^2$ is 
$\delta_{\mathrm{asymp}}$-close to a circle with radius $r(t):=\sqrt{2(T_{ext}{-}t)}$
in the sense of Definition~\ref{DefinitionShrinkingCircle}.
Consistent with the claim of Theorem~\ref{theo:mainResultC},
we will choose a suitable value of the constant $\delta_{\mathrm{asymp}}$
in the course of the proof.

\subsection{Heuristics: Leading-order behaviour near extinction time}
\label{subsec:heuristicsLeadingOrder}
The aim of this subsection is to compute
heuristically the time evolution of our linearized error functional in the simplified case of a centered self-similarly shrinking circle. As a result, our analysis reveals the instability of our linearized error functional near the extinction time.

Consider a centered circle self-similarly shrinking by mean curvature flow: $t \mapsto 
\partial B_{r(t)} = \mathrm{im}\,\bar{\gamma}(t) \subset \mathbb{R}^2$,
where $\bar{\gamma}(t)\colon [0,2\pi r(t)) \to \partial B_{r(t)}$,
$\theta \mapsto r(t) e^{i\frac{\theta}{r(t)}}$, is an arc-length
parametrization of $\partial B_{r(t)}$. In particular, $\dot r = -\frac{1}{r}$
in the interval $(0,\frac{1}{2}r_0^2=T_{ext})$ for $r_0:=r(0)>0$,
i.e., $r(t) = \sqrt{2(T_{ext} - t)}$.

Apart from the shrinking circle, let us consider a second solution to mean curvature flow,
for which we in addition assume that it can be written as a smooth graph over the self-similarly
shrinking circle. More precisely, there exists a smooth time-dependent height function
$h(\cdot,t)\colon\partial B_{r(t)} \to \mathbb{R}$ with $|h(\cdot,t)| \ll r(t)$ and $|h'(\cdot,t)| \ll 1$
such that this second solution is represented as the image of the curve
\begin{align}
	\label{eq:heuristics1}
	\gamma_h(\cdot,t) := \big(\mathrm{id} + h(\cdot,t) \n_{\partial B_{r(t)}}\big) \circ \bar\gamma(\cdot,t)
	\quad\text{on } [0,2\pi r(t)),
\end{align}
where $\n_{\partial B_{r(t)}}$ denotes the inward-pointing unit normal along $\partial B_{r(t)}$
and by slight abuse of notation $h'(\cdot,t) := (\tau_{\partial B_{r(t)}}\cdot\nabla^\mathrm{tan}) h(\cdot,t)$
for the choice of tangent vector field $\tau_{\partial B_{r(t)}}\big(\bar\gamma(\theta,t)\big) = i e^{i\frac{\theta}{r(t)}}$.
As we will show in Lemma \ref{lemma:relentfreaze}, our error functional in this perturbative setting corresponds to leading order to
\begin{align}
	\label{eq:heuristics9}
	E_h(t) := \int_{\partial B_{r(t)}} \frac{1}{2} \frac{h^2(\cdot,t)}{r^2(t)}
	+ \frac{1}{2} (h')^2(\cdot,t) \,d\mathcal{H}^1.
\end{align}
For the current purposes, we content ourselves with studying the stability of~$E_h(t)$ near the extinction time.

To this end, we have to derive the PDE satisfied by the height function~$h$
(and its derivative). Dropping from now on for ease of notation the time dependence
of all involved quantities, we first note that by definition in case of self-similarly shrinking circle
\begin{align}
	\label{eq:heuristics2}
	\partial_t \bar \gamma = \Big(\frac{1}{r}\n_{\partial B_{r}}
	+ \lambda\tau_{\partial B_{r}}\Big) \circ \bar\gamma
	\quad\text{on } [0,2\pi r),
\end{align}
where $\lambda$ denotes the tangential velocity.  
Second, we may then, on one side, directly compute based on the definition~\eqref{eq:heuristics1}
\begin{equation}
	\label{eq:heuristics3}
	\begin{aligned}
		\partial_t \gamma_h &= \bigg(\Big(\frac{1}{r} + \partial_t h
		+ \lambda h'\Big) \n_{\partial B_{r}}\bigg) \circ \bar\gamma
		+ \bigg(\lambda \Big(1 - \frac{h}{r}\Big) \tau_{\partial B_{r}}\bigg) \circ \bar\gamma.
	\end{aligned}
\end{equation}
On the other side, since $\gamma_h$ is assumed to evolve by mean curvature flow, it holds
\begin{align}
	\label{eq:heuristics4}
	H_{\gamma_h} = \partial_t \gamma_h \cdot \n_{\gamma_h}
	\quad\text{on } [0,2\pi r),
\end{align}
where the normal $\n_{\gamma_h}$ and mean curvature $H_{\gamma_h}$
of the curve $\gamma_h$ are given by the elementary formulas
(with $J$ denoting the counter-clockwise rotation by~$90^\circ$)
\begin{equation}
	\begin{aligned}
		\label{eq:heuristics5}
		\n_{\gamma_h} &= J \frac{\partial_\theta\gamma_h}{|\partial_\theta\gamma_h|}
		= \Bigg(\frac{\Big(1-\frac{h}{r}\Big)\n_{\partial B_{r}} - h' \tau_{\partial B_{r}}}
		{\sqrt{\big(1-\frac{h}{r}\big)^2 + (h')^2}}\Bigg) \circ \bar\gamma
	\end{aligned}
\end{equation}
and
\begin{equation}
	\begin{aligned}
		\label{eq:heuristics6}
		H_{\gamma_h} &= \frac{\partial_{\theta\theta}\gamma_h}{|\partial_\theta\gamma_h|^2} \cdot \n_{\gamma_h}
		= \Bigg(\frac{\Big(1-\frac{h}{r}\Big)\Big(\frac{1}{r}+h''-\frac{h}{r^2}\Big) + 2 \frac{(h')^2}{r}}
		{\Big(\big(1-\frac{h}{r}\big)^2 + (h')^2\Big)^\frac{3}{2}}\Bigg) \circ \bar\gamma. 
	\end{aligned}
\end{equation}
From~\eqref{eq:heuristics3}--\eqref{eq:heuristics6}, one may now deduce the non-linear PDE
satisfied by the height function~$h$.
However, because in what follows we are only interested in identifying the leading-order behavior, we suppose
from now on that the height function~$h$ instead satisfies the corresponding linearized equation:
\begin{align}
	\label{eq:heuristics7}
	\partial_t h = h'' +\frac{h}{r^2}
	\quad\text{on } \partial B_{r}.
\end{align}
From this, using $(\partial_t h)'= \partial_t h' +  \frac{h}{r^2} $, we in particular deduce
\begin{align}
	\label{eq:heuristics8}
	\partial_t h' =  h''' +  2\frac{h'}{r^2} .
\end{align}
Indeed, this follows easily from~\eqref{eq:heuristics7}
and exploiting the change of variables $\widetilde{h}(\theta) := h(re^{i\theta})$
as a useful computational device.

Recalling~\eqref{eq:heuristics9}, we thus get from the transport theorem
as well as~\eqref{eq:heuristics7}--\eqref{eq:heuristics8}
\begin{align}
	\nonumber
	\frac{d}{dt} E_h &= 
	\int_{\partial B_{r}} \partial_t\bigg(\frac{1}{2} \frac{h^2}{r^2} 
	+ \frac{1}{2} (h')^2\bigg) \,d\mathcal{H}^1
	- \int_{\partial B_{r}} H_{\partial B_{r}}^2\bigg(\frac{1}{2} \frac{h^2}{r^2} 
	+ \frac{1}{2} (h')^2\bigg) \,d\mathcal{H}^1
	\\& 
	\label{eq:heuristics10}
	=\int_{\partial B_{r}} \frac{h}{r} 
	\Big(2\frac{h}{r^3} + \frac{h''}{r}\Big) \,d\mathcal{H}^1
	+ \int_{\partial B_{r}} h'
	\Big(h'''+ 2\frac{h'}{r^2}\Big) \,d\mathcal{H}^1
	\\&~~~
	\nonumber
	- \int_{\partial B_{r}} \frac{1}{r^2}\bigg(\frac{1}{2} \frac{h^2}{r^2} 
	+ \frac{1}{2} (h')^2\bigg) \,d\mathcal{H}^1.
\end{align}
Integrating by parts and collecting similar terms therefore yields
\begin{align}
	\label{eq:heuristics11}
	\frac{d}{dt} E_h + \int_{\partial B_{r}} (h'')^2 \,d\mathcal{H}^1
	= \int_{\partial B_{r}} \frac{3}{2} \frac{h^2}{r^4} 
	+ \frac{1}{2} \frac{(h')^2}{r^2} \,d\mathcal{H}^1.
\end{align}
Fourier decomposing 
\begin{align}
	\label{eq:fourierDecomp}
	[0,2\pi) \ni \theta \mapsto \widetilde{h}(\theta) = h(re^{i\theta})
	= a_0 \frac{1}{\sqrt{2\pi}}\chi_{[0,2\pi]} + \sum_{k=1}^\infty a_k \frac{\cos(k\theta)}{\sqrt{\pi}} 
	+ b_k \frac{\sin(k\theta)}{\sqrt{\pi}},
\end{align}
where we also recall the formulas for the associated Fourier coefficients
\begin{align*}
	a_0 = \int_0^{2\pi} \frac{1}{\sqrt{2\pi}} \widetilde h(\theta) \,d\theta,
	\quad
	a_k = \int_0^{2\pi} \widetilde h(\theta) \frac{\cos(k\theta)}{\sqrt{\pi}} \,d\theta,
	\quad
	b_k = \int_0^{2\pi} \frac{\sin(k\theta)}{\sqrt{\pi}} \,d\theta, 
\end{align*} 
then rearranges \eqref{eq:heuristics11} as
\begin{equation}
	\begin{aligned}
		\label{eq:heuristics12}
		&\frac{d}{dt} E_h + \frac{1}{r^3} \sum_{k=1}^\infty k^4\big(a_k^2+b_k^2\big)
		= \frac{1}{r^3} \frac{3}{2} a_0^2
		+ \frac{1}{r^3} \sum_{k=1}^\infty \Big(\frac{3}{2} + \frac{1}{2}k^2\Big)\big(a_k^2+b_k^2\big).
	\end{aligned}
\end{equation}
Since $k^4 - \frac{3}{2} - \frac{1}{2}k^2 > 0$ for $k \geq 2$, we infer that only
the modes $(a_0,a_1,b_1)$ are unstable near the extinction time (in the sense that
these are precisely those inducing the borderline non-integrable singularity $r^{-2}$
in the Gronwall estimate of $E_h$). 

\subsection{Heuristics: Decay estimate} 
\label{subsec:heuristicsDecay}
Geometrically, the unstable
modes correspond to time dilations and spatial translations.
The basic idea of the present work is to correct these by dynamically 
adapting the smoothly evolving strong solution. In the simplified
context of a self-similarly shrinking circle, this works heuristically as follows.

Consider $t_h > 0$ (to be interpreted as an upper bound for the perturbed solution)
as well as a smooth path $z\colon (0,t_h) \to \mathbb{R}^2$ of translations 
together with a smooth time diffeomorphism $T\colon (0,t_h) \to (0,\smash{\frac{1}{2}}r_0^2)$,
the latter to be thought of as a perturbation of the identity: $T(t) =: t + \mathfrak{T}(t)$
for $t \in (0,t_h)$. Based on this input, 
we then introduce the dynamically adapted solution
\begin{align}
	\label{eq:heuristics13}
	\bar\gamma^{z,T}(\theta,t) := \bar\gamma(\theta,T(t)) + z(t),
	\quad \theta \in [0,2\pi r_T(t)),\,t\in (0,t_h),
\end{align}
where $r_T(t):=r(T(t))$, and assume that the perturbed solution $\gamma_h$ is given by
\begin{align}
	\label{eq:heuristics14}
	\gamma_h(\cdot,t) = \big(\mathrm{id} + h(\cdot,t) \n_{\partial B_{r_T(t)}(z(t))}\big) \circ \bar\gamma^{z,T}(\cdot,t),
	\quad t\in (0,t_h),
\end{align}
where $|h(\cdot,t)| \ll r_T(t)$ and $|h'(\cdot,t)| \ll 1$. We are again interested in the
stability properties of
\begin{align}
	\label{eq:heuristics15}
	E_h^{z,T}(t) := \int_{\partial B_{r_T(t)}(z(t))} \frac{1}{2} \frac{h^2(\cdot,t)}{r^2(t)}
	+ \frac{1}{2} (h')^2(\cdot,t) \,d\mathcal{H}^1,
	\quad t\in (0,t_h).
\end{align}
In fact, we actually aim to identify ODEs for $z$ and $\mathfrak{T}$ such that $E_h^{z,T}$
satisfies a quantitative decay estimate on $(0,t_h)$. One of course already expects the
ODE for~$\mathfrak{T}$ to involve the mode $a_0$, whereas the ODE for $z$ is expected to be
encoded in terms of $(a_1,b_1)$. From now on, we again make use of the notational convention
of suppressing the time dependence of all involved quantities. To this end, it will
be convenient to associate to any map $f(\cdot,t)\colon \partial B_{r(t)} \to \mathbb{R}$
its time-rescaled version $f_T(\cdot,t)\colon \partial B_{r_T(t)} \to \mathbb{R}$ 
defined by $f_T(\cdot,t):=f(\cdot,T(t))$.

We start by computing the normal speed of $\partial B_{r_T}(z)$. 
By definition~\eqref{eq:heuristics13},
\begin{align}
	\label{eq:heuristics16}
	\partial_t \bar\gamma^{z,T} = (\partial_t\bar\gamma)_T \big(1 + \dot{\mathfrak{T}}\big)
	+ \dot{z}.
\end{align}
Hence, the normal speed of $\partial B_{r_T}(z)$
in the direction of $\n_{\partial B_{r_T}(z)}$ is given by
\begin{align}
	\label{eq:heuristics17}
	V_{\partial B_{r_T}(z)} = \frac{1}{r_T} \big(1 + \dot{\mathfrak{T}}\big)
	+ \n_{\partial B_{r_T}(z)} \cdot \dot{z}.
\end{align}
The tangential speed in the direction of $\tau_{\partial B_{r_T}(z)}$ is furthermore given by
\begin{align}
	\label{eq:heuristics18}
	\lambda_{\partial B_{r_T}(z)} = \lambda_T \big(1 + \dot{\mathfrak{T}}\big) 
	+ \tau_{\partial B_{r_T}(z)} \cdot \dot{z},
\end{align}
where $\lambda$ is the tangential velocity from~\eqref{eq:heuristics2}.
In particular, we may now compute
\begin{equation}
	\label{eq:heuristics19}
	\begin{aligned}
		\partial_t \gamma_h
		&= \bigg(\Big(V_{\partial B_{r_T}(z)} {+} \partial_t h
		{+} \lambda_{\partial B_{r_T}(z)} h'\Big) \n_{\partial B_{r_T}(z)}\bigg) \circ \bar\gamma^{z,T}
		\\&~~~
		+ \bigg(\Big(\lambda_{\partial B_{r_T}(z)} - \lambda_T\big(1 + \dot{\mathfrak{T}}\big)\frac{h}{r_T}\Big)
		\tau_{\partial B_{r_T}(z)}\bigg) \circ \bar\gamma^{z,T}.
	\end{aligned}
\end{equation}
Furthermore, the analogous versions of the formulas~\eqref{eq:heuristics5}--\eqref{eq:heuristics6} hold true:
\begin{equation}
	\begin{aligned}
		\label{eq:heuristics20}
		\n_{\gamma_h} 
		= \bigg(\frac{\Big(1-\frac{h}{r_T}\Big)\n_{\partial B_{r_T}(z)} 
			- h' \tau_{\partial B_{r_T}(z)}}
		{\sqrt{\big(1-\frac{h}{r_T}\big)^2 + (h')^2}}\bigg) \circ \bar\gamma^{z,T}
	\end{aligned}
\end{equation}
and
\begin{equation}
	\begin{aligned}
		\label{eq:heuristics21}
		H_{\gamma_h}
		= \bigg(\frac{\Big(1-\frac{h}{r_T}\Big)\Big(\frac{1}{r_T}+h''-\frac{h}{r_T^2}\Big) + 2 \frac{(h')^2}{r_T}}
		{\sqrt{\big(1-\frac{h}{r_T}\big)^2 + (h')^2}^3}\bigg) \circ \bar\gamma^{z,T}.
	\end{aligned}
\end{equation}
Combining the information provided by~\eqref{eq:heuristics17}--\eqref{eq:heuristics20}, we deduce
\begin{align*}
	\partial_t \gamma_h \cdot \n_{\gamma_h}&= 
	\Big(1-\frac{h}{r_T}\Big)
	\Big(V_{\partial B_{r_T}(z)} {+} \partial_t h
	{+} \lambda_{\partial B_{r_T}(z)} h'\Big)
	\\&~~~
	- h'\bigg(\Big(1-\frac{h}{r_T}\Big)\lambda_T\big(1 + \dot{\mathfrak{T}}\big)
	+ \tau_{\partial B_{r_T}(z)} \cdot \dot{z}\bigg)
	\\&
	= \Big(1-\frac{h}{r_T}\Big) \Big(\frac{1}{r_T}\big(1 + \dot{\mathfrak{T}}\big) 
	+ \n_{\partial B_{r_T}(z)} \cdot \dot{z} + \partial_t h\Big)
	- \frac{h}{r_T} h' \tau_{\partial B_{r_T}(z)} \cdot \dot{z}.
\end{align*}
Turning as in Section~\ref{subsec:heuristicsLeadingOrder} to the linearized PDE satisfied by the height function, we therefore obtain
\begin{align}
	\label{eq:heuristics22}
	\partial_t h = h'' + \frac{h}{r^2_T}
	- \frac{\dot{\mathfrak{T}}}{r_T} - \n_{\partial B_{r_T}(z)} \cdot \dot{z}
\end{align}
as well as
\begin{align}
	\label{eq:heuristics23}
	(\partial_t h') = h'''+ (2 + \dot{\mathfrak{T}})\frac{h'}{r^2_T}
	+ \frac{1}{r_T} \tau_{\partial B_{r_T}(z)} \cdot \dot{z}.
\end{align}

We may now finally compute based on the transport theorem, the definition~\eqref{eq:heuristics15},
and the formulas~\eqref{eq:heuristics17} as well as~\eqref{eq:heuristics22}--\eqref{eq:heuristics23}
\begin{align}
	\nonumber
	\frac{d}{dt} E_h^{z,T}
	&=\int_{\partial B_{r_T}(z)} \partial_t\bigg(\frac{1}{2} \frac{h^2}{r_T^2} 
	+ \frac{1}{2} (h')^2\bigg) \,d\mathcal{H}^1
	\\&~~~
	\nonumber
	- \int_{\partial B_{r_T(z)}} H_{\partial B_{r_T}(z)}V_{\partial B_{r_T}(z)} 
	\bigg(\frac{1}{2} \frac{h^2}{r_T^2} + \frac{1}{2} (h')^2\bigg) \,d\mathcal{H}^1
	\\&
	\nonumber
	= \int_{\partial B_{r_T}(z)} \frac{h}{r_T} \bigg(\frac{1{+}\dot{\mathfrak{T}}}{r_T^3}h
	+ \frac{1}{r_T}\Big(h'' + \frac{h}{r^2_T}
	- \frac{\dot{\mathfrak{T}}}{r_T} - \n_{\partial B_{r_T}(z)} \cdot \dot{z}\Big)\bigg) \,d\mathcal{H}^1
	\\&~~~
	\nonumber
	+ \int_{\partial B_{r_T}(z)} h' \Big(h'''+ (2 + \dot{\mathfrak{T}})\frac{h'}{r^2_T}
	+ \frac{1}{r_T} \tau_{\partial B_{r_T}(z)} \cdot \dot{z}\Big)  \,d\mathcal{H}^1
	\\&~~~
	\nonumber
	- \int_{\partial B_{r_T(z)}} \frac{1}{r_T^2} 
	\bigg(\frac{1}{2} \frac{h^2}{r_T^2} + \frac{1}{2} (h')^2\bigg) \,d\mathcal{H}^1
	\\&~~~
	\nonumber
	- \int_{\partial B_{r_T(z)}} H_{\partial B_{r_T}(z)}\big(V_{\partial B_{r_T}(z)} - H_{\partial B_{r_T}(z)}\big)
	\bigg(\frac{1}{2} \frac{h^2}{r_T^2} + \frac{1}{2} (h')^2\bigg) \,d\mathcal{H}^1.
\end{align}
Hence, integrating by parts and collecting again similar terms yields
\begin{align}
	\label{eq:heuristics24}
	\frac{d}{dt} E_h^{z,T}
	&=  \int_{\partial B_{r_T}(z)} \frac{3}{2}\frac{h^2}{r_T^4} \,d\mathcal{H}^1
	- \int_{\partial B_{r_T}(z)} \frac{h}{r_T^3} \dot{\mathfrak{T}} \,d\mathcal{H}^1
	\\&~~~
	\nonumber
	+ \int_{\partial B_{r_T}(z)} \frac{1}{2}\frac{(h')^2}{r_T^2} \,d\mathcal{H}^1
	- \int_{\partial B_{r_T}(z)} 2\frac{h}{r_T^2} \n_{\partial B_{r_T}(z)} \cdot \dot{z} \,d\mathcal{H}^1
	\\&~~~
	\nonumber
	- \int_{\partial B_{r_T}(z)} (h'')^2 \,d\mathcal{H}^1
	\\&~~~
	\nonumber
	+ R_{h.o.t.},
\end{align}
where 
\begin{align}
	\label{eq:heuristics25}
	R_{h.o.t.} := \int_{\partial B_{r_T}(z)} \frac{1}{r_T}\Big(\frac{\dot{\mathfrak{T}}}{r_T} 
	- \n_{\partial B_{r_T}(z)} \cdot \dot{z}\Big)
	\bigg(\frac{1}{2} \frac{h^2}{r_T^2} + \frac{1}{2} (h')^2\bigg) \,d\mathcal{H}^1.
\end{align}

Based on the Fourier decomposition~\eqref{eq:fourierDecomp},
the identity~\eqref{eq:heuristics24} now motivates to define
\begin{align}
	\label{eq:expectedODEs}
	\dot{\mathfrak{T}} = \frac{c_T}{r_T}\dashint_{0}^{2\pi} \widetilde{h} \,d\theta,
	\quad 
	\dot{z} = \frac{c_z}{r_T^2}\dashint_0^{2\pi} \widetilde{h} (-e^{i\theta}) \,d\theta,
\end{align}
where the constants $(c_T,c_z)$ are yet to be chosen. Indeed, with these choices
we get
\begin{equation}
	\begin{aligned}
		\label{eq:heuristics26}
		&\frac{d}{dt} E_h^{z,T} + \frac{(c_T {-} 3/2)}{r_T^2} \frac{a_0^2}{r_T} + \frac{(c_z {-} 1)}{r_T^2} \frac{a_1^2+b_1^2}{r_T}
		\\&~~~~~~~~~~~~~~~~~
		+ \frac{1}{r_T^2} \sum_{k = 2}^\infty \Big(k^4{-}\frac{3}{2}{-}\frac{1}{2}k^2\Big)\frac{a_k^2+b_k^2}{r_T}
		= R_{h.o.t.},
	\end{aligned}
\end{equation}
where, due to $|\dot{\mathfrak{T}}| \leq c_T\frac{1}{r_T}\|h\|_{L^\infty(\partial B_{r_T}(z))}$
and $|\dot{z}| \leq c_z\frac{1}{r_T^2}\|h\|_{L^\infty(\partial B_{r_T}(z))}$, one has
an estimate for the remainder term in the form of
\begin{align}
	\label{eq:heuristics27}
	\Big|R_{h.o.t.}\Big| \leq \Big(c_T {+} c_z\Big) \frac{\|h\|_{L^\infty(\partial B_{r_T}(z))}}{r_T}
	\frac{1}{r_T^3} \Big(\frac{1}{2}a_0^2 + \sum_{k=1}^\infty \frac{1}{2}(1{+}k^2)(a_k^2+b_k^2)\Big).
\end{align}
Hence, for given $\widetilde{\delta} \in (0,1)$, 
if $|h|\ll_{\widetilde{\delta},c_T,c_z} r_T$, one
gets an upgrade of~\eqref{eq:heuristics24} in the form of
\begin{equation}
	\label{eq:heuristics28}
	\begin{aligned}
		&\frac{d}{dt} E_h^{z,T} + \frac{c_T {-} 3/2(1{+}\widetilde{\delta})}{r_T^2} \frac{a_0^2}{r_T} 
		+ \frac{c_z {-} (1 {+} \widetilde{\delta})}{r_T^2} \frac{a_1^2+b_1^2}{r_T}
		\\&~~~~~~~~~~~~~~~~~~~
		+ \frac{1}{r_T^2} \sum_{k = 2}^\infty \bigg(k^4{-}(1{+}\widetilde{\delta})\Big(\frac{3}{2}{+}\frac{1}{2}k^2\Big)\bigg)\frac{a_k^2+b_k^2}{r_T}
		\leq 0.
	\end{aligned}
\end{equation}
Because of
\begin{align}
	\label{eq:heuristics28b}
	E_h^{z,T} = \frac{1}{2}\frac{a_0^2}{r_T} + \sum_{k=1}^\infty \frac{1}{2}(1{+}k^2)\frac{a_k^2+b_k^2}{r_T},
\end{align}
we deduce that for any constant $\alpha > 1$ satisfying
\begin{align}
	\label{eq:heuristics29}
	\alpha &\leq \min\{2c_T - 3(1{+}\widetilde{\delta}), c_z - (1{+}\widetilde{\delta})\},
	\\
	\label{eq:heuristics29b}
	\alpha \frac{1}{2}(1{+}k^2) &\leq k^4{-}(1{+}\widetilde{\delta})\Big(\frac{3}{2}{+}\frac{1}{2}k^2\Big),
	&& k\geq 2,
\end{align}
it holds
\begin{align}
	\label{eq:heuristics30}
	\frac{d}{dt} E_h^{z,T} + \frac{\alpha}{r_T^2} E_h^{z,T} \leq 0.
\end{align}

Choosing $c_T:=4$ and $c_z:= 6$,
optimizing shows that for any desired exponent $\alpha \in (1,5)$ there exists a choice
of the constant $\widetilde{\delta}$ such that~\eqref{eq:heuristics30} holds true
(in the perturbative regime $|h| \ll_{c_T,c_z,\widetilde{\delta}} 1$ with linearized
evolution law~\eqref{eq:heuristics22}). Indeed, the function
$f\colon [2,\infty) \to [0,\infty),\,x\mapsto \smash{\big(\frac{1}{2}(1{+}x^2)\big)^{-1}\big(x^4-(\frac{3}{2}{+}\frac{1}{2}x^2)\big)}$
is monotonically increasing and satisfies $f(2) = 5$. 

Finally, since 
$|\dot{\mathfrak{T}}| \leq \smash{c_T\frac{1}{r_T}\|h\|_{L^\infty(\partial B_{r_T}(z))}}$,
one may choose for any $\alpha \in (1,5)$ the constant $\widetilde{\delta}$
such that 
(in the perturbative regime $|h| \ll_{c_T,c_z,\widetilde{\delta}} 1$ with linearized
evolution law~\eqref{eq:heuristics22})
it even holds $\frac{d}{dt} E_h^{z,T} + (1{+}\dot{\mathfrak{T}})\frac{\alpha}{r_T^2} E_h^{z,T} \leq 0$,
so that $\smash{\frac{d}{dt}r_T^\alpha = -\alpha r_T^{\alpha-1}\frac{1}{r_T}
=-(1{+}\dot{\mathfrak{T}})\frac{\alpha}{r_T^2} r_T^\alpha}$ implies
\begin{align}
	\label{eq:heuristics31}
	E_h^{z,T}(t) \leq E_h^{z,T}(0) \Big(\frac{r_T(t)}{r_0}\Big)^\alpha,
	\quad t \in (0,t_h).
\end{align}
This is precisely the type of decay estimate (or, weak-strong stability estimate up to shift)
claimed in our main result, Theorem~\ref{theo:mainResultC}.

Before we turn in the upcoming subsections to a description of the key ingredients and steps for our proof of
Theorem~\ref{theo:mainResultC} (with the above considerations,
of course, being their main motivation), let us provide some final remarks on the main
assumptions behind the derivation of the decay estimate~\eqref{eq:heuristics31}.

First, one may derive a version of~\eqref{eq:heuristics24} also in the case where
the time-evolving curve $\bar\gamma$ is not parametrizing a perfect circle. The
main difference in this case is that the coefficients are not anymore simply constant
along~$\bar\gamma$ (i.e., not proportional to inverse powers of~$r_T$). It is
precisely at this stage where we exploit our notion of quantitative closeness
of the strong solution to a circular solution, cf.\ Definition~\ref{DefinitionShrinkingCircle},
allowing us to effectively reduce the situation to the constant-coefficient
computation~\eqref{eq:heuristics24} (i.e., in PDE~jargon, we perform nothing else than a global freezing
of coefficients). 

A second simplifying assumption was the usage of the linearized evolution law~\eqref{eq:heuristics22}
for the height function~$h$ as well as that we only considered the stability of the
leading order contribution~$E_h$ to our actual error functional. It will turn out that
these linearization errors are harmless and only impact the final stability estimate
qualitatively in the same manner as the term~$R_{h.o.t.}$ from~\eqref{eq:heuristics24}.

Needless to say, in the general setting of Theorem~\ref{theo:mainResultC} where we
aim for quantitative stability beyond circular topology change even for
the broader class of weak (i.e., varifold-$BV$) solutions, we can not rely
on the above considerations (e.g., transport theorem, derivation of the
(linearized) evolution law~\eqref{eq:heuristics22}) in order to rigorously derive the evolution of the error functional.
In order to still unravel the structure of the right hand side of~\eqref{eq:heuristics24}, we
instead make use of the recently introduced notions of gradient flow
calibrations and relative entropies for multiphase mean curvature flow
from~\cite{FischerHenselLauxSimon},
serving as a robust replacement of the above considerations to the weak setting.

Last but not least, one of course also needs an independent argument ensuring
that one can reduce the whole estimation strategy to a perturbative graph setting as above.
This, however, is precisely one of the key points of the upcoming subsections. 

\subsection{A general stability estimate for multiphase MCF}
\label{subsec:prelimStability}
Starting point of our strategy is a stability estimate, 
see Lemma~\ref{lem:preliminaryStability} below,
which one may essentially directly infer from the combination of~\cite[Proposition~17]{FischerHenselLauxSimon} and~\cite[Lemma~20]{FischerHenselLauxSimon} (or more precisely, their proofs), together with the following compatibility properties of the varifold $\mathcal{V}_t$ and the indicator functions $\chi_i$.
From Definition \ref{DefinitionVarSolution} one may infer that, for each $ i\in \{1,...,P\}$ 
and a.e.\ $t\in (0,\infty)$,  the Radon--Nikod\'{y}m derivatives
\begin{align}\label{eq:multiplicity}
\omega_i(\cdot, t) := \frac{\mathrm{d} |\nabla \chi_i(\cdot, t)|}{\mathrm{d} \mu_t}  \in [0,1], \quad \omega(\cdot, t) := \frac{\mathrm{d} \frac12 \sum_{i=1}^{P} |\nabla \chi_i(\cdot, t)|}{\mathrm{d} \mu_t}  \in [0,1]
\end{align}
exist. Note that $\omega = \frac12 \sum_{i=1}^{P} \omega_i $. 
Since $\mathcal{V}$ a family of integral varifolds, we have that $\omega \in \{1/n : n \in \mathbb{N}\} \cup \{0\}$ and 
\[
\mu_t \llcorner \Big\{\frac12 \sum_{i=1}^P \omega_i (\cdot, t) = 1\Big\} = \mathcal{H}^{d-1}\llcorner \Big(\Big\{\frac12 \sum_{i=1}^P \omega_i (\cdot, t) = 1\Big\} \cap \bigcup_{i \neq j} I_{i,j}(t)\Big)
\]
and, since $\mathcal{V}_t$ is rectifiable for a.e. $t \in (0,\infty)$, it follows that 
\begin{align*}\notag 
&\mathcal{V}_t \llcorner \Big\{\frac12 \sum_{i=1}^P \omega_i (\cdot, t) = 1\Big\} \\
&= \frac12 \sum_{i=1}^P \Big( \supp |\nabla \chi_i(\cdot, t)| \llcorner \Big\{\frac12 \sum_{i=1}^P \omega_i (\cdot, t) = 1\Big\} \otimes \big(\delta_{\operatorname{Tan}_x^{d-1}(\supp |\nabla \chi_i(\cdot, t)|)}\big)_{x \in \supp |\nabla \chi_i|}\Big) 
\end{align*}
for a.e. $t \in (0,\infty)$. In particular, from Brakke's perpendicularity theorem it follows that for a.e. $t \in (0,\infty)$
\[
H_\mu(\cdot, t)= \Big( H_\mu(\cdot, t) \cdot \frac{\nabla \chi_i (\cdot, t)}{|\nabla \chi_i (\cdot, t)|}\Big) \frac{\nabla \chi_i (\cdot, t)}{|\nabla \chi_i (\cdot, t)|}
\]
$\mathcal{H}^{d-1}$-a.e. on $\{\frac12 \sum_{i=1}^P \omega_i (\cdot, t) = 1\} \cap \supp |\nabla \chi_i(\cdot, t)|$ for all $i \in \{1,...,P\}$.
We define 
\[
V_i := - H_\mu(\cdot, t) \cdot \frac{\nabla \chi_i (\cdot, t)}{|\nabla \chi_i (\cdot, t)|}
\]
 and $V_{i,j}:=V_i=-V_j$ for all $i,j \in \{1,...,P\}$, $i \neq j$.

\begin{lemma}[Preliminary stability estimate]
	\label{lem:preliminaryStability}
	Let 
	$((\xi_i)_{i=1,\ldots,P},(\vartheta_i)_{i=1,\ldots,P{-}1},B)$---to be thought of
	as being constructed from~$\bar\chi$---such that
	\begin{align*}
		\xi_i &\in W^{1,\infty}_{loc}\big([0,T_{ext});W^{1,\infty}(\mathbb{R}^2;\mathbb{R}^2)\big) \cap 
		L^\infty_{loc}\big([0,T_{ext});W^{2,\infty}(\mathbb{R}^2;\mathbb{R}^2)\big),
		\\
		\vartheta_i &\in W^{1,1}_{loc}\big([0,T_{ext});L^1(\mathbb{R}^2)\big) \cap 
		L^1_{loc}\big([0,T_{ext});(W^{1,1} {\cap} W^{1,\infty})(\mathbb{R}^2)\big),
		\\
		B &\in L^\infty_{loc}\big([0,T_{ext});W^{2,\infty}(\mathbb{R}^2;\mathbb{R}^2)\big),
	\end{align*}
	where $(\vartheta_i)_{i=1,\ldots,P{-}1}$ is supposed to satisfy,
	for all $t \in (0,T_{ext})$,
	\begin{align*}
		\vartheta_1(\cdot,t) &< 0 &&\text{in the interior of } \{\bar\chi_1(\cdot,t) = 1\},
		\\
		\vartheta_1(\cdot,t) &> 0 &&\text{in the exterior of } \overline{\{\bar\chi_1(\cdot,t) = 1\}},
		\\
		\vartheta_i(\cdot,t) &= 1 &&\text{throughout } \mathbb{R}^2 \text{ for } 
		i \in \{1,\ldots,P{-}1\}.
	\end{align*}
	Define $\xi_{i,j} := \xi_{i} - \xi_{j}$ for all distinct $i,j\in\{1,\ldots,P\}$.
	Consider in addition a triple $(\tchi,z,T)$ so that $\tchi \in (0,\infty)$,
	$z \in W^{1,\infty}_{loc}((0,\tchi);\mathbb{R}^2)$ and $T\in W^{1,\infty}((0,\tchi);(0,T_{ext}))$,
	and define for all $t \in (0,\tchi)$
	\begin{align}
		\label{def:interfaceerror}
		E_{\mathrm{int}}[\mathcal{V},\chi|\bar\chi^{z,T}](t) &:= 
		\mu_t(\Rd) - \sum_{i,j=1,i\neq j}^{P} \frac{1}{2}
		\int_{I_{i,j}(t)} \n_{i,j}(\cdot,t) \cdot \xi_{i,j}^{z,T}(\cdot,t) \,d\mathcal{H}^1,
		\\
		\label{def:bulkError}
		E_{\mathrm{bulk}}[\chi|\bar\chi^{z,T}](t) &:= \sum_{i=1}^{P-1} \int_{\mathbb{R}^2} 
		|\chi_i(\cdot,t) {-} \bar\chi_i^{z,T}(\cdot,t)| |\vartheta_i^{z,T}(\cdot,t)| \,dx,
		\\
		\label{def:relenergy}
		E[\mathcal{V},\chi|\bar\chi^{z,T}](t) &:= E_{\mathrm{int}}[\mathcal{V},\chi|\bar\chi^{z,T}](t) + E_{\mathrm{bulk}}[\chi|\bar\chi^{z,T}](t).
	\end{align}
	
	Then, for all $[s,\tau] \subset [0,\tchi)$
	and all $\psi \in C^1_{cpt}([0,\tchi);[0,\infty))$, 
	it holds
	\begin{equation}
		\label{eq:prelimStabilityRelEntropy}
		\begin{aligned}
			&\psi(\tau)
			E_{\mathrm{int}}[\mathcal{V}, \chi|\bar\chi^{z,T}](\tau) 
			+ \int_{s}^{\tau} \psi(t) \sum_{i,j=1,i\neq j}^{P} \frac{1}{2} 
			\mathcal{D}_{i,j}[\chi|\bar\chi^{z,T}](t) \,dt
			\\&~~~
			\leq \psi(s)
			E_{\mathrm{int}}[\mathcal{V}, \chi|\bar\chi^{z,T}](s) 
				+ \int_{s}^{\tau} \psi(t) 
				RHS^{\operatorname{var-BV}}[\mathcal{V},\chi|\bar\chi^{z,T}](t ) \,dt 
				\\&~~~~~~
			+ \int_{s}^{\tau} \psi(t) \sum_{i,j=1,i\neq j}^{P} \frac{1}{2} 
			RHS_{i,j}^{\mathrm{int}}[\chi|\bar\chi^{z,T}](t) \,dt
			\\&~~~~~~
			+ \int_{s}^{\tau} \Big(\frac{d}{dt}\psi(t)\Big) 
		E_{\mathrm{int}}[\mathcal{V}, \chi| \bar \chi^{z,T}](t) \,dt,
		\end{aligned}
	\end{equation}
	as well as
	\begin{equation}
	\begin{aligned}
		\label{eq:prelimStabilityBulk}
		&\psi(\tau)
		E_{\mathrm{bulk}} [\chi|\bar\chi^{z,T}](\tau) 
		\\&~~~
		= \psi(s)
		E_{\mathrm{bulk}}[\chi|\bar\chi^{z,T}](s) + \int_{s}^{\tau} \psi(t)
		\sum_{i=1}^{P-1} RHS_{i}^{\mathrm{bulk}}[\chi|\bar\chi^{z,T}](t) \,dt
		\\&~~~
		+ \int_{s}^{\tau} \Big(\frac{d}{dt}\psi(t)\Big) 
		E_{\mathrm{bulk}}[\mathcal{V}, \chi| \bar \chi^{z,T}](t) \,dt,
	\end{aligned}
	\end{equation}
	where the individual terms are given by
	\begin{align*}
		\mathcal{D}_{i,j}[\chi|\bar\chi^{z,T}](t)
		&:=  \int_{I_{i,j}(t)} \frac{1}{2} 
		\big|V_{i,j}{+}\nabla\cdot\xi_{i,j}^{z,T}\big|^2(\cdot,t) \,d\mathcal{H}^1
		\\&~~~~
		+ \int_{I_{i,j}(t)} \frac{1}{2} 
		\big|V_{i,j}\n_{i,j}{-}(B^{z,T}\cdot\xi_{i,j}^{z,T})\xi_{i,j}^{z,T}\big|^2(\cdot,t) \,d\mathcal{H}^1,
	\end{align*}
\begin{align*}
	&  RHS^{\operatorname{var-BV}}[\mathcal{V},\chi|\bar\chi^{z,T}](t) \\
	& := - \int_{\Rd} |H_\mu|^2 \Big(1- \frac12 \sum_{i=1}^P \rho_i \Big) \, {d}\mu_t 	+ \int_{\Rd} H_\mu \cdot B \Big(1- \frac12 \sum_{i=1}^P \rho_i \Big) \, {d}\mu_t\\
&~~~~	-     \sum_{i,j=1,i\neq j}^{P} \frac{1}{2}  \int_{I_{i,j}} \big( \Id - \n_{i,j} \otimes \n_{i,j}\big)  : \nabla B \,  d\mathcal{H}^1
	+ \int_{\Rd \times \mathbf{G}(d,d-1)} \Id_{\mathbf{G}(d,d-1)} : \nabla B \, d \mathcal{V}_t
,
\end{align*}
	and
	\begin{align*}
		&RHS_{i,j}^{\mathrm{int}}[\chi|\bar\chi^{z,T}](t) 
		\\&~
		:= 
		- \int_{I_{i,j}(t)} 
		\big(\partial_t\xi_{i,j}^{z,T} {+} (B^{z,T}\cdot\nabla)\xi_{i,j}^{z,T} {+} (\nabla B^{z,T})^\mathsf{T}\xi_{i,j}^{z,T}\big)(\cdot,t)
		\cdot (\n_{i,j}{-}\xi_{i,j}^{z,T})(\cdot,t)\,d\mathcal{H}^1
		\\&~~~~~
		- \int_{I_{i,j}(t)} \big(\partial_t\xi_{i,j}^{z,T} {+} (B^{z,T}\cdot\nabla)\xi_{i,j}^{z,T}\big)(\cdot,t)
		\cdot \xi_{i,j}^{z,T}(\cdot,t) \,d\mathcal{H}^1
		\\&~~~~~
		+ \int_{I_{i,j}(t)} \frac{1}{2} \big|\nabla\cdot\xi_{i,j}^{z,T} {+} B^{z,T}\cdot\xi_{i,j}^{z,T}\big|^2(\cdot,t) \,d\mathcal{H}^1
		\\&~~~~~
		- \int_{I_{i,j}(t)} \frac{1}{2} \big|B^{z,T}\cdot\xi_{i,j}^{z,T}\big|(\cdot,t) 
		\big(1 - |\xi_{i,j}^{z,T}|^2\big)(\cdot,t)\,d\mathcal{H}^1
		\\&~~~~~
		- \int_{I_{i,j}(t)} (1 - \n_{i,j}\cdot\xi_{i,j}^{z,T})(\cdot,t) \nabla\cdot\xi_{i,j}^{z,T}(\cdot,t)
		(B^{z,T}\cdot\xi_{i,j}^{z,T})(\cdot,t) \,d\mathcal{H}^1
		\\&~~~~~
		+ \int_{I_{i,j}(t)} \big((\mathrm{Id} {-} \xi_{i,j}^{z,T}\otimes \xi_{i,j}^{z,T})B^{z,T}\big)(\cdot,t) \cdot 
		\big((V_{i,j} {+} \nabla\cdot\xi_{i,j}^{z,T})\n_{i,j}\big)(\cdot,t) \,d\mathcal{H}^1
		\\&~~~~~
		+ \int_{I_{i,j}(t)} (1 - \n_{i,j}\cdot\xi_{i,j}^{z,T})(\cdot,t) \nabla\cdot B^{z,T}(\cdot,t) \,d\mathcal{H}^1
		\\&~~~~~
		- \int_{I_{i,j}(t)} (\n_{i,j} {-} \xi_{i,j}^{z,T})(\cdot,t) \otimes (\n_{i,j} {-} \xi_{i,j}^{z,T})(\cdot,t)
		: \nabla B^{z,T}(\cdot,t) \,d\mathcal{H}^1,
	\end{align*}
	as well as
	\begin{align*}
		&RHS_{i}^{\mathrm{bulk}}[\chi|\bar\chi^{z,T}](t)
		\\&~
		:= - \sum_{j = 1, j\neq i}^{P} \int_{I_{i,j}(t)} \vartheta_i^{z,T}(\cdot,t) 
		(B^{z,T}\cdot\xi_{i,j}^{z,T} {-} V_{i,j})(\cdot,t) \,d\mathcal{H}^1
		\\&~~~~~
		- \sum_{j = 1, j\neq i}^{P} \int_{I_{i,j}(t)} \vartheta_i^{z,T}(\cdot,t) 
		B^{z,T}(\cdot,t)\cdot (\n_{i,j} - \xi_{i,j}^{z,T})(\cdot,t) \,d\mathcal{H}^1
		\\&~~~~~
		+ \int_{\mathbb{R}^2} (\chi_i {-} \bar\chi_i^{z,T})(\cdot,t) \vartheta_i^{z,T}(\cdot,t) \nabla\cdot B^{z,T}(\cdot,t) \,dx
		\\&~~~~~
		+ \int_{\mathbb{R}^2} (\chi_i {-} \bar\chi_i^{z,T})(\cdot,t) 
		\big(\partial_t\vartheta_i^{z,T} {+} (B^{z,T}\cdot\nabla)\vartheta_i^{z,T}\big)(\cdot,t) \,dx.
	\end{align*}
\end{lemma}

One may rewrite \eqref{def:interfaceerror} as (cf. \cite[Section 4.4]{FischerHenselLauxSimon})
\begin{align}\label{varinterfaceerror}
E_{\mathrm{int}}[\mathcal{V},\chi|\bar\chi^{z,T}](t) = \int_{\Rd} 1 - \frac12 \sum_{i=1}^{P} \omega_i (\cdot, t) \,  \mathrm{d} \mu_t  + E_{\mathrm{int}}[\chi|\bar\chi^{z,T}](t),
\end{align}
where $E_{\mathrm{int}}[\chi|\bar\chi^{z,T}]$ is the interface error for BV solutions in the 
sense of Definition~\ref{DefinitionWeakSolution}, namely 
\begin{align*}E_{\mathrm{int}}[\chi|\bar\chi^{z,T}](t)&= \sum_{i,j=1,i\neq j}^{P} \frac{1}{2} \int_{I_{i,j}(t)} 1 - \n_{i,j}(\cdot,t) \cdot \xi_{i,j}^{z,T}(\cdot,t) \,d\mathcal{H}^1.\notag 
\end{align*}
Observe that the first term on the right hand side of \eqref{varinterfaceerror} is nonnegative by \eqref{eq:varsol3} and provides
control of the multiplicity of the varifold whenever it exceeds the multiplicity of the BV interface $\frac12 \sum_{i=1}^{P} |\nabla \chi_i (\cdot, t)|$.
In particular, the varifold-BV interface error  \eqref{def:interfaceerror} controls the interface error for BV solutions.

The next two steps of our strategy are concerned with the construction
of the input data for Lemma~\ref{lem:preliminaryStability}: first $(\tchi,z,T)$
and second $((\xi_i)_{i=1,\ldots,P},(\vartheta_i)_{i=1,\ldots,P{-}1},B)$.

\subsection{Construction of dynamic shifts}
In case of  a single closed curve, a characteristic length scale associated with the evolution of~$\bar\chi$
is given by $r(t) := \sqrt{\frac{\mathrm{vol}(\{\bar{\chi}_1(\cdot,t){=}1\})}{\pi}}$,
$t \in [0,T_{ext})$. Since
\begin{align*}
	\frac{d}{dt} \mathrm{vol}(\{\bar{\chi}(\cdot,t){=}1\}) = -2\pi
\end{align*}
we infer that 
\begin{equation}\label{eq:timeevr}
	\begin{cases}
		\dot r(t) = -\frac{1}{r(t)} , & t \in (0, T_{ext}),\\
		r(0) = r_0 := \sqrt{\frac{\mathrm{vol}(\{\bar{\chi}_{1}(\cdot,0){=}1\})}{\pi}},
	\end{cases}
\end{equation}
Hence, $r(t) = \sqrt{2(T_{ext} - t)}$ and $T_{ext}= \frac12 r_0^2$.

In Subsection~\ref{subsec:heuristicsDecay}, we already derived the
defining ODEs for $(z,T = \mathrm{id}{+}\mathfrak{T})$, at least in
a regime where the weak solution is represented as a sufficiently regular
graph over the smooth solution, cf.\ \eqref{eq:expectedODEs}. Of course,
there is no guarantee to be in that regime for all times, so that the
general construction needs a robust version of~\eqref{eq:expectedODEs}.
To this end, it is convenient to work with the notion of interface
error heights (see Figure \ref{Figure:errorheights}).

\begin{figure}\small 
\includegraphics{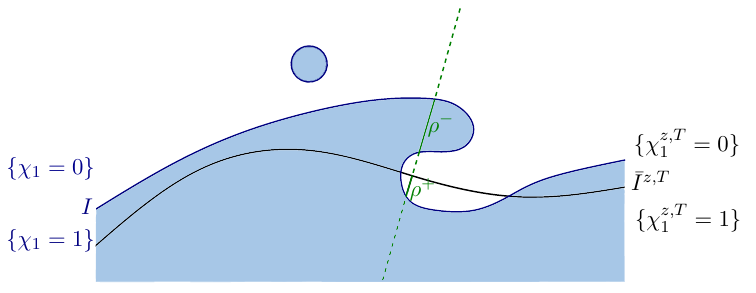}
 \caption{Interface error heights}
 \label{Figure:errorheights}
 \end{figure}

\begin{construction}[Interface error heights] 
	\label{DefinitionErrorHeights}	
 Consider a triple $(\tchi,z,T)$ so that $\tchi \in (0,\infty)$,
	$z \in W^{1,\infty}_{loc}((0,\tchi);\mathbb{R}^2)$ and $T\in W^{1,\infty}((0,\tchi);(0,T_{ext}))$.
	Let $\zeta\colon\mathbb{R} \rightarrow  [0,1]$ be a smooth cutoff function
	such that $\zeta(s)=1$ for $|s|\leq 1/(16C_\zeta)$
	and $\zeta(s)=0$ for $|s| > 1/(8C_\zeta)$,
	where $C_\zeta \in [1,\infty)$ is a given constant.
	We then define \emph{interface error heights}
	\begin{align*}
		\rho(\cdot,\cdot;z,T),\rho_\pm(\cdot,\cdot;z,T) \colon 
		\bar{I}^{z,T}
		\rightarrow \mathbb{R}
	\end{align*}
	through a slicing construction (recall that $r_T(t) := r(T(t)),\,t\in (0,t_{\chi})$):
	\begin{align} 
		\label{eq:defrho_plus} 
		\rho_+(x, t;z,T) &:=  \int_0^{\frac12 r_T(t)} (\bar\chi_1^{z,T} - \chi_1)
		\big(x {+} \ell \n_{\bar{I}^{z,T}}(\cdot,t),t\big)  \zeta\Big(\frac{\ell}{r_T(t)}\Big) \,d\ell,
		\\  
		\label{eq:defrho_minus}
		\rho_-(x, t;z,T) &:= \int_{-\frac12 r_T(t)}^{0} (\chi_1 - \bar\chi_1^{z,T}) 
		\big(x {+} \ell \n_{\bar{I}^{z,T}}(\cdot,t),t\big)  \zeta\Big(\frac{\ell}{r_T(t)}\Big) \,d\ell,
		\\
		\label{eq:defrho}
		\rho(x,t;z,T) &:= \rho_+(x,t;z,T) - \rho_-(x,t;z,T).
	\end{align}
\end{construction}

We have everything in place to construct the dynamic shifts.

\begin{lemma}[Existence of space-time shifts] \label{lemma:existzT}
	There exists a unique choice of
	\begin{itemize}[leftmargin=0.7cm]
		\item a time horizon $\tchi > 0$,
		\item a path of translations $z \in W^{1,\infty}_{loc}((0,\tchi);\mathbb{R}^2)$, and
		\item a strictly increasing bijection 
		$T\in W^{1,\infty}((0,\tchi);(0,T_{ext}))$,
	\end{itemize}
	which in addition satisfy $(z(0),T(0)) = (0,0)$
	as well as, by defining $\mathfrak{T} := T - \mathrm{id}$,
	\begin{align}\label{eq:evzT} 
		\begin{bmatrix}
			\dot{z}(t) \\ \dot{\mathfrak{T}}(t)
		\end{bmatrix} 
		= 		\begin{bmatrix} \frac{6}{r_T^2(t)} \dashint_{\bar I^{z,T}(t)} 
			\rho(\cdot,t;z,T) \n_{\bar{I}^{z,T}}(\cdot,t) \,d\mathcal{H}^1 \\ 
			\frac{4}{r_T(t)} \dashint_{\bar I^{z,T}(t)} \rho(\cdot,t;z,T) \,d\mathcal{H}^1
		\end{bmatrix}, \quad t \in (0,\tchi).
	\end{align}
	Moreover, for given $\delta_{\mathrm{err}} \in (0,\frac{1}{2})$ one may choose
	the constant $C_\zeta \gg_{\delta_{\mathrm{err}}} 1$ 
	from Construction~\ref{DefinitionErrorHeights} such that 
	\begin{align} \label{ineq:zTprime}
		|\dot{z}(t)| \leq \delta_{\mathrm{err}}\frac{1}{r_T(t)},
		\quad |\dot{\mathfrak{T}}(t)| \leq \delta_{\mathrm{err}},
		\quad t \in (0,\tchi).
	\end{align}
\end{lemma}
The proof of Lemma \ref{lemma:existzT} is given in Section~\ref{subsec:existencespacetimeshifts}.

\subsection{Construction of gradient flow calibrations}
In contrast to~\cite{FischerHenselLauxSimon}, in the present work the smoothly evolving solution~$\bar\chi$
stems from a simple two-phase geometry instead of a more complicated multiphase geometry
with branching interfaces. As a consequence, the construction of a gradient flow 
calibration (cf.\ \cite[Definition~2 and Definition~4]{FischerHenselLauxSimon})
is particularly simple and can be given directly as follows.

\begin{construction}[Gradient flow calibration up to extinction time]
	\label{DefinitionCalibration}	
	Consider a smooth cutoff function
	$\eta\colon\mathbb{R} \rightarrow [0,1]$ such that 
	$\eta(s)=1$ for $|s| \leq 1/8$, $\eta(s)=0$ for $|s|\geq 1/4$
	and $\|\eta'\|_{L^\infty(\mathbb{R})} \leq 16$. We then 
	define an extension~$\xi\colon\mathbb{R}^2{\times}[0,T_{ext})\to\mathbb{R}^2$ of the unit vector field~$\n_{\bar{I}}$ by means of	
	\begin{equation}\label{eq:defxi}
		\xi(x,t) =  \eta\Big(\frac{\sdist_{\bar I}(x,t)}{r(t)}\Big) \n_{\bar{I}}\big(P_{\bar I}(x,t),t\big),
		\quad (x,t) \in \mathbb{R}^2{\times}[0,T_{ext}).
	\end{equation}
	
	Based on this auxiliary construction, we may now introduce families of vector fields
	$(\xi_{i})_{i=1,\ldots,P}$ and $(\xi_{i,j})_{i,j\in\{1,\ldots,P\},i\neq j}$
	(defined as maps $\mathbb{R}^2{\times}[0,T_{ext}) \to \mathbb{R}^2$) by the following simple procedure:
	\begin{itemize}[leftmargin=1.4cm]
		\item $\xi_{i,j} := \xi_{i} - \xi_{j}$ for any $i,j \in \{1,...,P\}$, $i \neq j$.
		\item  $\xi_i \equiv 0$ for all $i \notin \{1,P\}$.
		\item $\xi_1 := -\frac{1}{2}\xi$ and $\xi_P:=\frac{1}{2}\xi$.
	\end{itemize}	
	
	Furthermore, we define an extension $B\colon \mathbb{R}^2{\times}[0,T_{ext})\to\mathbb{R}^2$
	of the normal velocity field $H_{\bar{I}}\n_{\bar{I}}$ of $\bar\chi$ through
	\begin{align}\label{eq:defB}
		B(x,t) :=	\eta\Big(\frac{\sdist_{\bar I}(x,t)}{r(t)}\Big) 
		(H_{\bar{I}}\n_{\bar{I}})\big(P_{\bar I}(x,t),t\big),
		\quad (x,t) \in \mathbb{R}^2{\times}[0,T_{ext}).
	\end{align}
	
	Finally, for the construction of the family $(\vartheta_i)_{i=1,\ldots,P{-}1}$
	(defined as functions mapping $\mathbb{R}^2{\times}[0,T_{ext})\to [-1,1]$), we proceed as follows.	
	Let $\bar \vartheta\colon \mathbb{R} \to [-1,1]$ be a smooth function such that 
	$\bar \vartheta (s)= -s$ for $|s|\leq 1/4$,  $\bar \vartheta (s)=-1$  for $s\geq 1/2$, 
	$\bar \vartheta (s)= 1$  for $s\leq -1/2$, and $\|\bar\vartheta'\|_{L^\infty(\mathbb{R})}\leq 4$.
	We then define 
	\begin{align}
		\label{eq:defWeight}
		\vartheta(x,t) := \frac{1}{r(t)}\bar\vartheta\Big(\frac{\sdist_{\bar I}(x,t)}{r(t)}\Big),
		\quad (x,t) \in \mathbb{R}^2{\times}[0,T_{ext}),
	\end{align}
	and, at last,
	\begin{itemize}[leftmargin=1.4cm]
		\item $\vartheta_1 := \vartheta$,
		\item  $\vartheta_i := 1$ for all $i \in \{2,\ldots,P{-}1\}$.
	\end{itemize}	
\end{construction}

Note that the gradient flow calibration $((\xi_i)_{i=1,\ldots,P},(\vartheta_i)_{i=1,\ldots,P{-}1},B)$
from Construction~\ref{DefinitionCalibration} is an admissible input for
Lemma~\ref{lem:preliminaryStability}. From now on, whenever we refer to
an admissible element from either $(\tchi,z,T)$ or 
$((\xi_i)_{i=1,\ldots,P},(\vartheta_i)_{i=1,\ldots,P{-}1},B)$, we
always mean their specific realizations provided by Lemma~\ref{lemma:existzT}
or Construction~\ref{DefinitionCalibration}, respectively.

\subsection{Time splitting: Regular vs.\ non-regular times}
\label{subsec:timeSplitting}
With the input for Lemma~\ref{lem:preliminaryStability}
being constructed, the main remaining major task is to upgrade
the preliminary stability estimates~\eqref{eq:prelimStabilityRelEntropy} 
and~\eqref{eq:prelimStabilityBulk} to the decay estimate~\eqref{ineq:decayE}
for the overall error. The main idea here is to reduce the whole estimation
strategy to a regime where the weak solution~$\chi$ is only a small perturbation
of~$\bar\chi$, for which we in turn already formally identified the leading-order contributions
to the stability estimates in 
Subsections~\ref{subsec:heuristicsLeadingOrder}--\ref{subsec:heuristicsDecay}.

\begin{definition}[Regular and non-regular times]
	Fix $\Lambda > 0$.	We then define a disjoint decomposition 
	\[
	(0, \tchi) = \mathcal{T}_{\mathrm{non\text{-}reg}}(\Lambda) \cup \mathcal{T}_{\mathrm{reg}}(\Lambda)
	\]
	where
	\begin{align} \label{def:badTimes}
		\mathcal{T}_{\mathrm{non\text{-}reg}}(\Lambda) 
		:= \Bigg\{t \in \left(0 , \tchi\right) \colon
		 \int_{\Rd} |H_\mu(\cdot,t)|^2 \,d\mu_t  \geq \Lambda\frac{2\pi}{r_T(t)}
		\Bigg\}
	\end{align}
and~$H_{\mu}(\cdot,t)$ is the 
generalized mean curvature vector of~$\mathcal{V}_t$, see~\eqref{eq:Hvarifold}.
\end{definition}
Observe that in the framework of BV solutions in the sense of Definition \ref{DefinitionWeakSolution}, the defining inquality in \eqref{def:badTimes} reduces to
\[
\sum_{i,j=1, i \neq j}^P \int_{{I}_{i,j}(t)}
\frac{1}{2}|V_{i,j}(\cdot,t)|^2 \,d\mathcal{H}^1 \geq \Lambda\frac{2\pi}{r_T(t)}.
\]

The motivation behind the previous definition is as follows. On one side, for non-regular
times, the right hand sides of the preliminary stability estimates~\eqref{eq:prelimStabilityRelEntropy} 
and~\eqref{eq:prelimStabilityBulk} turn out to be easily estimated thanks to the 
defining condition of disproportionally large dissipation of the weak solution,
cf.\ \eqref{def:badTimes}. On the other side,  
the opposite of \eqref{def:badTimes} together with a smallness assumption 
on the overall error (consistent with the decay~\eqref{ineq:decayE}) imply
for regular times the desired perturbative setting. The latter is formalized
in the following result.

\begin{proposition}[Perturbative regime at regular times] \label{theo:graph}
	Fix $\Lambda > 0$ and let $t \in \mathcal{T}_{\mathrm{reg}}(\Lambda)$, i.e., 
	$t \in (0,\tchi)$ such that
	\begin{align} \label{boundgoodtimes}
		\int_{\Rd} |H_\mu(\cdot,t)|^2 \,d\mu_t 
		< \Lambda\frac{2\pi}{r_T(t)}.
	\end{align}
	Given $C_\zeta \geq 1$ from Construction~\ref{DefinitionErrorHeights} and given any
	$C, C' \geq 1$, there exists a constant $\delta \ll_{\Lambda,C,C',C_\zeta} \frac{1}{2}$  such that 
	\begin{align}\label{hp:regularitygoodtimes}
		E[\mathcal{V}, \chi | \bar \chi^{z,T}](t)  \leq \delta {r_T(t)} 
	\end{align}
	implies: 
	\begin{itemize}[leftmargin=0.7cm]
		\item $\chi_i(\cdot,t) \equiv 0$ for all $i \notin \{1,P\}$ 
		and
		\begin{align*}
		\mathcal{V}_t = \big(\mathcal{H}^1 \llcorner \supp|\nabla\chi_1(\cdot,t)|\big)
		\otimes (\delta_{\mathrm{Tan}_x(\supp|\nabla\chi_1(\cdot,t)|)})_{x \in \supp|\nabla\chi_1(\cdot,t)|}.
		\end{align*}
		\item There exists a \emph{height function} 
		\begin{align}
			\label{eq:graphReg}
			h(\cdot, t)\in H^2(\bar{I}^{z,T}(t))
		\end{align}
		such that the only remaining interface is given by
		\begin{align}
			\label{eq:graphRep}
			I_{1,P}(t) = \big\{x\in\bar{I}^{z,T}(t)\colon x + h(x,t)\n_{\bar{I}^{z,T}}(x,t)\big\}.
		\end{align}
		\item Finally, it holds
		\begin{align}
			\label{eq:smallC0norm}
			\|h(\cdot,t)\|_{L^\infty(\bar{I}^{z,T}(t))} &\leq \frac{r_T(t)}{16 \max\{C,C_\zeta\}},
			\\
			\label{eq:smallC1norm}
			\|h'(\cdot,t)\|_{L^\infty(\bar{I}^{z,T}(t))} &\leq \frac{1}{C'}.
		\end{align}
	\end{itemize}
	In particular, the height function~$h(\cdot,t)$ coincides with the
	interface error height $\rho(\cdot,t;z,T)$ from Construction~\ref{DefinitionErrorHeights}
	and~\eqref{eq:evzT} simply reads
	\begin{align}\label{eq:evzTgraph} 
		\begin{bmatrix}
			\dot{z}(t) \\ \dot{\mathfrak{T}}(t)
		\end{bmatrix} 
		= 		\begin{bmatrix} \frac{6}{r_T^2(t)} \dashint_{\bar I^{z,T}(t)} 
			h(\cdot,t) \n_{\bar{I}^{z,T}}(\cdot,t) \,d\mathcal{H}^1 \\ 
			\frac{4}{r_T(t)} \dashint_{\bar I^{z,T}(t)} h(\cdot,t) \,d\mathcal{H}^1
		\end{bmatrix}.
	\end{align}
\end{proposition}

In the perturbative regime of Proposition~\ref{theo:graph},
our error functionals take the following form.

\begin{lemma}[Error functionals in perturbative regime]
	\label{lemma:relentfreaze}
	Fix $t \in (0,\tchi)$ and assume that the conclusions of Proposition~\ref{theo:graph}
	hold true. Given $\delta_{\mathrm{err}} \in (0,1)$, one may select $C,C' \gg_{\delta_{\mathrm{err}}} 1$
	from~\emph{\eqref{eq:smallC0norm}--\eqref{eq:smallC1norm}} such that
	\begin{equation}
		\label{eq:bulkPerturbativeRegime}
		\begin{aligned}
			&(1{-}\delta_{\mathrm{err}}) \int_{\bar{I}^{z,T}(t)} 
			\frac{1}{2} \Big(\frac{h(\cdot,t)}{r_T(t)}\Big)^2 \,d\mathcal{H}^1
			\\&~~~
			\leq E_{\mathrm{bulk}}[\chi|\bar\chi^{z,T}](t)
			\leq (1{+}\delta_{\mathrm{err}}) \int_{\bar{I}^{z,T}(t)} 
			\frac{1}{2} \Big(\frac{h(\cdot,t)}{r_T(t)}\Big)^2 \,d\mathcal{H}^1
		\end{aligned}
	\end{equation}
	as well as
	\begin{equation}
		\label{eq:relEntropyPerturbativeRegime}
		\begin{aligned}
			&(1{-}\delta_{\mathrm{err}}) \int_{\bar{I}^{z,T}(t)} 
			\frac{1}{2} |h'(\cdot,t)|^2 \,d\mathcal{H}^1
			\\&~~~
			\leq
			E_{\mathrm{int}}[ \chi|\bar\chi^{z,T}](t)
			\leq (1{+}\delta_{\mathrm{err}}) \int_{\bar{I}^{z,T}(t)} 
			\frac{1}{2} |h'(\cdot,t)|^2 \,d\mathcal{H}^1.
		\end{aligned}
	\end{equation}
\end{lemma}

The proofs of Proposition~\ref{theo:graph} and of Lemma~\ref{lemma:relentfreaze} are given in Section~\ref{proof:reductiontograph} and in Section~\ref{proof:errorfuncperturbative}, respectively.

\subsection{Stability estimates at non-regular times}
\label{subsec:stabilityEstimateNonRegular}
In a next step, we take care of the estimation of
the right hand sides of~\eqref{eq:prelimStabilityRelEntropy} 
and~\eqref{eq:prelimStabilityBulk} 
in the case of disproportionally large dissipation.

\begin{lemma} \label{lemma:stabilitybad}
	There exist $\delta,\delta_{\mathrm{asymp}}  \ll \frac{1}{2}$ as well as $\Lambda \gg_{\delta,\delta_{\mathrm{asymp}} } 1$
	such that for every $t \in \mathcal{T}_{\mathrm{non\text{-}reg}}(\Lambda)$
	satisfying~\eqref{eq:asymptotically_circular1} and $E[\mathcal{V}, \chi|\bar\chi^{z,T}](t) \leq \delta r_T(t)$
	it holds
	\begin{align} 
		\label{ineq:stabilitybad}
		&\sum_{i,j=1,i\neq j}^{P} \frac{1}{2} 
		\Big( -\mathcal{D}_{i,j}[\chi|\bar\chi^{z,T}](t)
		+  RHS_{i,j}^{\mathrm{int}}[\chi|\bar\chi^{z,T}](t) \Big)\\\notag
		&~~~~+  RHS^{\operatorname{var-BV}}[\mathcal{V},\chi|\bar\chi^{z,T}](t)
			+ \sum_{i=1}^{P-1} RHS_{i}^{\mathrm{bulk}}[\chi|\bar\chi^{z,T}](t)
		\\&~~~~
		\nonumber
		\leq
		- \frac{1}{2} \int_{\Rd}|H_\mu(\cdot,t)|^2 \,d\mu_t.
	\end{align}
\end{lemma}

We may easily post-process the estimate~\eqref{ineq:stabilitybad}
to an estimate in terms of our error functional consistent with the final
decay estimate~\eqref{ineq:decayE}.

\begin{corollary}
	\label{cor:nonRegular}
	There exist $\delta,\delta_{\mathrm{asymp}}  \ll \frac{1}{2}$ as well as $\Lambda \gg_{\delta,\delta_{\mathrm{asymp}} } 1$
	such that for every $t \in \mathcal{T}_{\mathrm{non\text{-}reg}}(\Lambda)$
	satisfying~\eqref{eq:asymptotically_circular1} and
	$E[\mathcal{V},\chi|\bar\chi^{z,T}](t) \leq \delta r_T(t)$ it holds
	\begin{align} 
		\label{ineq:stabilitybad2}
			&\sum_{i,j=1,i\neq j}^{P} \frac{1}{2} 
		\Big( -\mathcal{D}_{i,j}[\chi|\bar\chi^{z,T}](t)
		+  RHS_{i,j}^{\mathrm{int}}[\chi|\bar\chi^{z,T}](t) \Big)\\\notag
		&~~~~+  RHS^{\operatorname{var-BV}}[\mathcal{V},\chi|\bar\chi^{z,T}](t)
		+ \sum_{i=1}^{P-1} RHS_{i}^{\mathrm{bulk}}[\chi|\bar\chi^{z,T}](t)
		\\&~~~~
		\nonumber
		\leq - \frac{5}{r_T^2(t)} E[\mathcal{V},\chi|\bar\chi^{z,T}](t).
	\end{align}
\end{corollary}

In fact, there is nothing particular about having~$5$ as a factor on the right hand side
of~\eqref{ineq:stabilitybad2}, and it can indeed be replaced by any constant~$C$.
The proofs of Lemma~\ref{lemma:stabilitybad} and Corollary~\ref{cor:nonRegular} 
can be found in Section~\ref{proof:stabilitybad}. 

\subsection{Stability estimates for perturbative regime}
\label{subsec:stabilityEstimateGraph}
We proceed with the estimation of
the right hand sides of~\eqref{eq:prelimStabilityRelEntropy} 
and~\eqref{eq:prelimStabilityBulk} in the 
perturbative regime described by Proposition~\ref{theo:graph}.
We first derive the version of the stability estimate~\eqref{eq:heuristics24}
without making use of the assumption on~$\bar\chi$
being quantitatively close to a shrinking circle. 
The derivation of these estimates, namely the proofs of the following lemmas, are contained in Sections \ref{proof:perreg1}-\ref{proof:perreg2}-\ref{proof:perreg3}.

\begin{lemma}
	[Stability estimate in perturbative setting: variable coefficients]
	\label{lemma:relenineqgraph}
	Fix $t \in (0,\tchi)$ and assume that the conclusions of Proposition~\ref{theo:graph}
	hold true. Given $\delta_{\mathrm{err}} \in (0,1)$, one may choose the constants 
	$C,C' \gg_{\delta_{\mathrm{err}}} 1$
	from~\emph{\eqref{eq:smallC0norm}--\eqref{eq:smallC1norm}} such that
	\begin{equation}
		\begin{aligned}\label{ineq:graphset}
			&\sum_{i,j=1,i\neq j}^{P} \frac{1}{2} 
		\Big( -\mathcal{D}_{i,j}[\chi|\bar\chi^{z,T}](t)
		+  RHS_{i,j}^{\mathrm{int}}[\chi|\bar\chi^{z,T}](t) \Big)\\
		&~~~~+  RHS^{\operatorname{var-BV}}[\mathcal{V},\chi|\bar\chi^{z,T}](t)
			+ \sum_{i=1}^{P-1} RHS_{i}^{\mathrm{bulk}}[\chi|\bar\chi^{z,T}](t)
			\\&~~~~~~~~
			\leq R_{l.o.t.} + R_{h.o.t.},
		\end{aligned}
	\end{equation}
	where the leading order terms are given by
	\begin{align*}
		R_{l.o.t.} &:= - \int_{\bar{I}^{z,T}(t)} (h'')^2(\cdot,t) \,d\mathcal{H}^1
		\\&~~~~
		+ \int_{\bar{I}^{z,T}(t)} \bigg(\frac{3}{2} H_{\bar{I}^{z,T}}^2(\cdot,t) - \frac{1}{r_T^2(t)}\bigg)
		(h')^2(\cdot,t) \,d\mathcal{H}^1
		\\&~~~~
		+ \int_{\bar{I}^{z,T}(t)} \frac{1}{r_T^2(t)}\bigg(\frac{1}{2}
		H_{\bar{I}^{z,T}}^2(\cdot,t) + \frac{1}{r_T^2(t)}\bigg)	h^2(\cdot,t) \,d\mathcal{H}^1
		\\&~~~~
		- \int_{\bar{I}^{z,T}(t)} \Big(\frac{1}{r_T^2(t)} + H_{\bar{I}^{z,T}}^2(\cdot,t) \Big)h(\cdot,t)
		\n_{\bar{I}^{z,T}}(\cdot,t) \cdot \dot{z}(t) \,d\mathcal{H}^1 
		\\&~~~~
		- \int_{\bar{I}^{z,T}(t)} \frac{1}{r^2_T(t)} H_{\bar{I}^{z,T}}(\cdot,t)
		h(\cdot,t) \dot{\mathfrak{T}}(t)\,\,d\mathcal{H}^1
		\\&~~~~
		- \int_{\bar{I}^{z,T}(t)} H_{\bar{I}^{z,T}}'(\cdot,t) 
		\big(\tau_{\bar{I}^{z,T}}(\cdot,t)\cdot\dot{z}(t)\big) h(\cdot,t) \,d\mathcal{H}^1 
		\\&~~~~
		- \int_{\bar{I}^{z,T}(t)} H_{\bar{I}^{z,T}}'(\cdot,t) \dot{\mathfrak{T}}(t) h'(\cdot,t) \,d\mathcal{H}^1 
		\\&~~~~
		+ \int_{\bar{I}^{z,T}(t)} 2 H_{\bar{I}^{z,T}}(\cdot,t) H_{\bar{I}^{z,T}}'(\cdot,t) 
		h(\cdot,t) h'(\cdot,t) \,d\mathcal{H}^1
	\end{align*}
	and the higher order terms are given by
	\begin{align*}
		R_{h.o.t.} &:=
		\delta_{\mathrm{err}}
		\int_{\bar{I}^{z,T}(t)} (h'')^2(\cdot,t) \,d\mathcal{H}^1
		\\&~~~~
		+ \delta_{\mathrm{err}}
		\int_{\bar{I}^{z,T}(t)} \Big(\frac{1}{r_T^2(t)} {+} \big|H_{\bar{I}^{z,T}}'\big|(\cdot,t)
		\Big) (h')^2(\cdot,t) \,d\mathcal{H}^1
		\\&~~~~
		+ \delta_{\mathrm{err}}
		\int_{\bar{I}^{z,T}(t)} \Big(\frac{1}{r_T^4(t)} {+} \big(H_{\bar{I}^{z,T}}'\big)^2(\cdot,t)
		\Big) h^2(\cdot,t) \,d\mathcal{H}^1
		\\&~~~~
		+ \delta_{\mathrm{err}} 
		\int_{\bar{I}^{z,T}(t)} \frac{1}{r_T(t)}\big|h'(\cdot,t) \tau_{\bar{I}^{z,T}} \cdot \dot{z}\big| 
		+ \frac{1}{r_T^2(t)}\big|h(\cdot,t) \n_{\bar{I}^{z,T}} \cdot \dot{z}\big|
		\,d\mathcal{H}^1
		\\&~~~~
		+ \delta_{\mathrm{err}} 
		\int_{\bar{I}^{z,T}(t)} \frac{1}{r_T^3(t)} \big|h(\cdot,t)\dot{\mathfrak{T}}\big|
		+ \big|H_{\bar{I}^{z,T}}'(\cdot,t) h'(\cdot,t)\dot{\mathfrak{T}}\big| 
		\,d\mathcal{H}^1.
	\end{align*}
\end{lemma}

Note that in case $\bar\chi_1$ describes a circle shrinking by mean curvature flow,
then the leading order terms on the right hand side of~\eqref{ineq:graphset}
are indeed precisely those captured by~\eqref{eq:heuristics24}.

In a second step, we post-process the previous estimate~\eqref{ineq:graphset} 
to the constant-coefficient estimate~\eqref{eq:heuristics24}. In PDE jargon,
this amounts to nothing else than a freezing of coefficients, only exploiting
the estimates from Definition~\ref{DefinitionShrinkingCircle}.

\begin{lemma} 
	[Stability estimate in perturbative setting: frozen coefficients]
	\label{lemma:relenineqpolar}
	Fix $t \in (0,\tchi)$, assume that the conclusions of Proposition~\ref{theo:graph}
	hold true, and define $\widetilde h(\cdot,t)\colon [0,2\pi) \to \mathbb{R}$,
	$\theta \mapsto h\big(\bar{\gamma}^{z,T}\big(\frac{L_{\bar{\gamma}^{z,T}}}{2\pi}\theta,t\big),t\big)$,
	where, for each $\widetilde{t} \in (0,T_{ext})$,
	$\bar{\gamma}(\cdot,\widetilde{t}\,)$ denotes an arc-length parametrization of~$\bar{I}(\widetilde{t}\,)$.
	Given $\delta_{\mathrm{err}} \in (0,1)$, one may choose the constants 
	$C,C' \gg_{\delta_{\mathrm{err}}} 1$ from~\emph{\eqref{eq:smallC0norm}--\eqref{eq:smallC1norm}} 
	as well as the constant $\delta_{\mathrm{asymp}} \ll_{\delta_{\mathrm{err}}} \frac{1}{2}$
	from Definition~\ref{DefinitionShrinkingCircle} such that
	\begin{equation}
		\begin{aligned} \label{ineq:freaze}
			&\sum_{i,j=1,i\neq j}^{P} \frac{1}{2} 
		\Big( -\mathcal{D}_{i,j}[\chi|\bar\chi^{z,T}](t)
		+  RHS_{i,j}^{\mathrm{int}}[\chi|\bar\chi^{z,T}](t) \Big)\\
		&~~~~+  RHS^{\operatorname{var-BV}}[\mathcal{V},\chi|\bar\chi^{z,T}](t)
			+ \sum_{i=1}^{P-1} RHS_{i}^{\mathrm{bulk}}[\chi|\bar\chi^{z,T}](t)
			\\&~~~~~~~~
			\leq \widetilde{R}_{l.o.t.} + \widetilde{R}_{h.o.t.},
		\end{aligned}
	\end{equation}
	where the leading order terms are given by
	\begin{align*}
		\widetilde{R}_{l.o.t.} &:= 
		- \frac{1}{r^3_T(t)}\int_{0}^{2\pi } (\partial_\theta ^2 \widetilde h)^{2}(\cdot,t)
		- \frac{1}{2} (\partial_\theta  \widetilde h)^{2}(\cdot,t)
		- \frac{3}{2} \widetilde h^2(\cdot,t) \,d \theta
		\\&~~~~
		- 4 \frac{1}{r_T^3(t)} \bigg(\frac{1}{\sqrt{2\pi}}\int_0^{2\pi} \widetilde h(\cdot,t) \,d \theta \bigg) ^2
		- 6 \frac{1}{r_T^3(t)} \bigg|\frac{1}{\sqrt{\pi}}\int_0^{2\pi} \widetilde h(\cdot,t) e^{i\theta} \,d \theta \bigg| ^2
	\end{align*}
	and the higher order term is simply given by
	\begin{align*}
		\widetilde{R}_{h.o.t.} &:=
		\delta_{\mathrm{err}} \frac{1}{r_T^3(t)}
		\int_{0}^{2\pi } (\partial_\theta ^2 \widetilde h)^{2}(\cdot,t)
		+ \frac{1}{2} (\partial_\theta  \widetilde h)^{2}(\cdot,t)
		+ \frac{3}{2} \widetilde h^2(\cdot,t) \,d \theta.
	\end{align*}
\end{lemma}

Since we reduced matters to the constant coefficient case, 
we may now in a third step employ Fourier methods to obtain
in the perturbative regime a stability estimate consistent 
with the decay estimate~\eqref{ineq:decayE}.

\begin{lemma}[Final stability estimate in perturbative setting] \label{lemma:stabilitygood}
	Fix $t \in (0,\tchi)$ and assume that the conclusions of Proposition~\ref{theo:graph}
	hold true. Given $\alpha \in (1,5)$, one may choose the constants 
	$C,C' \gg_{\alpha} 1$ from~\emph{\eqref{eq:smallC0norm}--\eqref{eq:smallC1norm}},
	the constant $\delta_{\mathrm{asymp}} \ll_{\alpha} \frac{1}{2}$
	from Definition~\ref{DefinitionShrinkingCircle}, and
	the constant $C_\zeta \gg_\alpha 1$ from Construction~\ref{DefinitionErrorHeights}
	such that
	\begin{align}\label{ineq:stabilityregular} 
		&	\sum_{i,j=1,i\neq j}^{P} \frac{1}{2} 
		\Big( -\mathcal{D}_{i,j}[\chi|\bar\chi^{z,T}](t)
		+  \frac{1}{2} RHS_{i,j}^{\mathrm{int}}[\chi|\bar\chi^{z,T}](t) \Big)
		+ \sum_{i=1}^{P-1} RHS_{i}^{\mathrm{bulk}}[\chi|\bar\chi^{z,T}](t) 
		\\&~~~ \nonumber
		\leq - \frac{\alpha}{r_T^2(t)}(1 {+} \dot{\mathfrak{T}}) E[\chi|\bar\chi^{z,T}](t).
	\end{align}
\end{lemma}

\subsection{A priori stability estimate up to extinction time}
The penultimate step of our strategy is simply a summary
of our estimates from Subsections~\ref{subsec:timeSplitting}--\ref{subsec:stabilityEstimateGraph} (see Section \ref{subsec:proofstability} for the proof).

\begin{theorem}\label{theo:stability}
	Fix a decay exponent $\alpha\in (1,5)$ 
	and a time $\widetilde{t}_\chi \in (0,\tchi)$. 
	One may choose the constant $C_\zeta \gg_{\alpha} 1$ from
	Construction~\ref{DefinitionErrorHeights} as well as
	constants $\delta_{\mathrm{asymp}} \ll_{\alpha} \frac{1}{2}$
	and $\delta \ll_{\alpha,C_\zeta} \frac{1}{2}$
	(all independent of~$\widetilde{t}_\chi$) such that if for all $t \in (0,T_{ext})$ 
	the interior of $\{\bar\chi_1(\cdot,t) {=} 1\} \subset \mathbb{R}^2$ is 
	$\delta_{\mathrm{asymp}}$-close to a circle with radius $r(t):=\sqrt{2(T_{ext}{-}t)}$
	in the sense of Definition~\ref{DefinitionShrinkingCircle} and
	\begin{align} 
		\label{hpEt}
		E[\mathcal{V}, \chi| \bar \chi^{z,T}](t) &\leq  
		\delta r_T(t)
		\text{ for all } t \in [0,\widetilde{t}_\chi),
	\end{align} 
	then it holds for all $[s,\tau] \subset [0,\widetilde{t}_\chi)$
	and all $\psi \in C^1_{cpt}([0,\tchi);[0,\infty))$
	\begin{equation}
	\begin{aligned}\label{ineq:aprioristability} 
		&\psi(\tau)
		E[\mathcal{V}, \chi| \bar \chi^{z,T}](\tau) + \int_s^{\tau} \psi(t)
		\frac{\alpha}{r^2_T(t)} \big(1 {+} \dot{\mathfrak{T}}(t)\big) E[\mathcal{V}, \chi| \bar \chi^{z,T}](t) \,dt 
		\\&~~~
		\leq \psi(s)
		E[\mathcal{V}, \chi| \bar \chi^{z,T}](s)
		+ \int_{s}^{\tau} \Big(\frac{d}{dt}\psi(t)\Big) 
		E[\mathcal{V}, \chi| \bar \chi^{z,T}](t) \,dt.
	\end{aligned}
	\end{equation}
\end{theorem}

The unconditional decay estimate~\eqref{ineq:decayE} from Theorem~\ref{theo:mainResultC} 
now follows by means of a simple ODE argument (cf.\ Section~\ref{sec:proofmain}).
The asserted estimates~\eqref{eq:upperBoundTranslation}--\eqref{eq:upperBoundTimeDilation} 
on the space-time shifts are in turn the content of the following result
(see Subsection~\ref{subsec:boundsShifts} for a proof) .

\begin{lemma}\label{prop:boundzT}
	In the setting of Theorem \ref{theo:stability},
	one may choose the constants such that
	assumption~\eqref{hpEt} implies
	\begin{align} \label{eq:boundzT}
		\frac{1}{r_0}\| z \|_{L^\infty_t(0,\tchi)} &\leq 
		\sqrt{\frac{1}{r_0}E[\mathcal{V}, \chi_0|\bar\chi_0]}, \\
		\label{eq:boundzT2}
		\frac{1}{T_{ext}}	\| T - \mathrm{id}\|_{L^\infty_t(0,\tchi)} &\leq 
		\sqrt{\frac{1}{r_0}E[\mathcal{V}, \chi_0|\bar\chi_0]}.
	\end{align}
\end{lemma}


\section{Proof of the main theorem}\label{sec:proofmain}
We proceed in two steps. For the whole proof, fix $\alpha \in (1,5)$, 
and choose $C_\zeta \gg_{\alpha} 1$,
$\delta_{\mathrm{asymp}} \ll_{\alpha} \frac{1}{2}$
and $\delta \ll_{\alpha,C_\zeta} \frac{1}{2}$
such that Theorem~\ref{theo:stability} applies.
We then also fix an auxiliary constant $\kappa \in (0,\delta r_0)$.

\textit{Step~1: Post-processed a priori stability estimate.}
Let the conclusion of Theorem~\ref{theo:stability} hold true
for some $\widetilde{t}_\chi \in (0,t_\chi)$. We
then claim that
for a.e.\ $t \in (0,\widetilde{t}_\chi)$
\begin{align}
	\label{eq:auxProofMainResult1}
	E[\mathcal{V}, \chi|\bar\chi^{z,T}](t) \leq \big(E[\mathcal{V}_0,\chi_0|\bar\chi_0] {+} \kappa \big)
	\Big(\frac{r_T(t)}{r_0}\Big)^\alpha =: e(t) \;
	\text{ for any } \kappa>0.
\end{align}

For a proof of~\eqref{eq:auxProofMainResult1}, 
we first note that $(0,\tchi) \ni t \mapsto E[\mathcal{V},\chi|\bar\chi^{z,T}](t) \in (0,\infty)$
is of bounded variation in~$(0,\tchi)$ due to the conditions satisfied 
by~$\chi$ being a varifold-BV~solution for multiphase mean curvature flow
in the sense of Definition~\ref{DefinitionVarSolution}, and the identity
\begin{equation}
\label{eq:decompError}
\begin{aligned}
	E[\mathcal{V},\chi|\bar\chi^{z,T}](t)
	&= E[\mathcal{V},\chi(\cdot,t)] - \sum_{i=1}^P \int_{\mathbb{R}^2} 
	\chi_i(\cdot,t) (\nabla \cdot \xi^{z,T})(\cdot,t) \,dx
	\\&~~~
	+ \sum_{i=1}^{P{-}1} \int_{\mathbb{R}^2} (\chi_i{-}\bar\chi_i^{z,T})(\cdot,t)
	\vartheta_i^{z,T}(\cdot,t) \,dx.
\end{aligned}
\end{equation}
By slight abuse of notation, we denote the associated distributional
derivative by $\frac{d}{dt} E[\mathcal{V},\chi|\bar\chi^{z,T}]$.
It then follows from~\eqref{ineq:aprioristability} that
\begin{align}
\label{eq:estimateDistDerivative}
\frac{d}{dt} E[\mathcal{V},\chi|\bar\chi^{z,T}] \leq 
- \frac{\alpha}{r^2_T} \big(1 {+} \dot{\mathfrak{T}}\big) E[\mathcal{V}, \chi| \bar \chi^{z,T}]
\quad\text{in distributional sense}.
\end{align}
Since $t \mapsto e(t)$ is a smooth function, one may infer from the 
product rule of distributional derivatives that 
$\frac{d}{dt} \frac{E[\mathcal{V},\chi|\bar\chi^{z,T}]}{e}
= \frac{(\frac{d}{dt} E[\mathcal{V},\chi|\bar\chi^{z,T}]) e - E[\mathcal{V},\chi|\bar\chi^{z,T}]
\frac{d}{dt} e}{e^2}$, so that~\eqref{eq:estimateDistDerivative}
and $\frac{d}{dt} e = -\frac{\alpha}{r^2_T} \big(1 {+} \dot{\mathfrak{T}}\big) e$
imply
\begin{align}
\label{eq:estimateQuotient}
\frac{d}{dt} \frac{E[\mathcal{V},\chi|\bar\chi^{z,T}]}{e} \leq 0
\quad\text{in distributional sense}.
\end{align}
Testing~\eqref{eq:estimateQuotient} by standard bump functions,
we obtain for a.e.\ $t \in (0,\widetilde{t}_\chi)$ and a.e.\ $s \in (0,t)$ that
$\frac{E[\mathcal{V},\chi|\bar\chi^{z,T}](t)}{e(t)} 
\leq \frac{E[\mathcal{V},\chi|\bar\chi^{z,T}](s)}{e(s)}$.
Since~\eqref{ineq:aprioristability}
furthermore implies that $[0,\widetilde{t}_\chi) \ni t \mapsto E[\mathcal{V},\chi|\bar\chi^{z,T}](t)$
is non-increasing, we obtain from taking $(0,t) \ni s \downarrow 0$ that~\eqref{eq:auxProofMainResult1}
indeed holds true.

\textit{Step~2: Proof of~\eqref{ineq:decayE} under assumption~\eqref{ineq:initialhp}.}
We define
\begin{align}
	\mathcal{T} := \big\{t \in (0,\tchi)\colon
	E[\mathcal{V},\chi|\bar\chi^{z,T}](t) > e(t)\big\},
\end{align}
and we argue in favor of~\eqref{ineq:decayE} by contradiction.
Hence, we assume $\mathcal{T} \neq \emptyset$ and define
$\widetilde{t}_\chi := \inf \mathcal{T} \in [0,\tchi)$. 
Since $E[\mathcal{V}_0,\chi_0|\bar\chi^{z,T}_0] < e(0)$, it is not hard to show
that $\widetilde{t}_\chi \neq 0$. Then, by construction and hypothesis~\eqref{ineq:initialhp}, 
we observe that assumption~\eqref{hpEt} is in place for all $t \in [0,\widetilde{t}_\chi)$.
In other words, the estimate~\eqref{eq:auxProofMainResult1} applies 
on~$(0,\widetilde{t}_\chi)$. However, on the other side
$E[\mathcal{V},\chi|\bar\chi^{z,T}](\widetilde{t}_\chi) \leq \lim_{t \uparrow \widetilde{t}_\chi}
E[\mathcal{V},\chi|\bar\chi^{z,T}](t)$ by virtue of the energy~$E[\mathcal{V},\chi]$
being non-increasing and the remaining constituents of~$E[\mathcal{V},\chi|\bar\chi^{z,T}]$
from~\eqref{eq:decompError} being absolutely continuous. In other words,
$\frac{E[\mathcal{V},\chi|\bar\chi^{z,T}](\widetilde{t}_\chi)}{e(\widetilde{t}_\chi)} \leq 1$
contradicting our assumption $\mathcal{T} \neq \emptyset$.
Hence, $\mathcal{T} = \emptyset$
and taking the limit $\kappa \downarrow 0$ implies the decay estimate~\eqref{ineq:decayE}.
Finally, the bounds~\eqref{eq:upperBoundTranslation} and~\eqref{eq:upperBoundTimeDilation}
follow from Lemma~\ref{prop:boundzT}.
\qed

\section{Weak-strong stability estimates}

\subsection{Proof of Lemma~\ref{lem:preliminaryStability}: Preliminary stability estimate}
Without the additional non-negative test function in time~$\psi$
(or more precisely, in case $\psi \equiv 1$ on $[s,\tau]$),
the proof of~\eqref{eq:prelimStabilityRelEntropy} and~\eqref{eq:prelimStabilityBulk}
directly follows from the arguments used in the proofs 
of~\cite[Proposition~17]{FischerHenselLauxSimon} and~\cite[Lemma~20]{FischerHenselLauxSimon}, respectively,
and the arguments of~\cite[Section~4.4]{FischerHenselLauxSimon}. The extension to general
$\psi \in C^1_{cpt}([0,t_\chi);[0,\infty))$ in turn follows along the same lines,
with the only additional ingredient being that one has to rely on the general
form of Brakke's inequality (see~\cite[Definition~2.1~(d)]{StuvardTonegawa}) instead
of just the energy dissipation inequality~\eqref{eq:energyDissip}.
\qed

\subsection{Proof of Lemma~\ref{lemma:stabilitybad} 
	and Corollary~\ref{cor:nonRegular}: Stability at non-regular times}\label{proof:stabilitybad}
Before we turn to the proofs of Lemma~\ref{lemma:stabilitybad}
and Corollary~\ref{cor:nonRegular}, respectively,
we start with two useful auxiliary results. The first is concerned
with bounds for our gradient flow calibration.

\begin{lemma}\label{lemma:propgradfllow}
	Consider the gradient flow calibration $((\xi_i)_{i=1,\ldots,P},(\vartheta_i)_{i=1,\ldots,P{-}1},B)$
	from Construction~\ref{DefinitionCalibration} and recall that $\xi_{i,j} := \xi_{i} - \xi_{j}$ for all distinct $i,j\in\{1,\ldots,P\}$. 
	There exists a universal constant $\widetilde{C} \in [1,\infty)$ such that
	for all $t \in [0,\tchi)$ and all $i,j\in\{1,\ldots,P\}$ with $i \neq j$, it holds
	\begin{align} \label{eq:esttimexi}
		\big\|\big(\partial_t \xi_{i,j}^{z,T}\big)(\cdot,t)\big\|_{L^\infty(\mathbb{R}^2)}
		&\leq \frac{\widetilde{C}}{r_T^2(t)},
		\\ \label{eq:estdivxi}
		\big\|\big(\nabla \cdot \xi_{i,j}^{z,T}\big)(\cdot,t)\big\|_{L^\infty(\mathbb{R}^2)}
		&\leq \frac{\widetilde{C}}{r_T(t)},
		\\ \label{eq:esttimevartheta} 
		\big\|\big(\partial_t \vartheta_i^{z,T}\big)(\cdot,t)\big\|_{L^\infty(\mathbb{R}^2)}
		&\leq \frac{\widetilde{C}}{r_T^3(t)}.
	\end{align}
\end{lemma}

\begin{proof}[Proof of Lemma~\ref{lemma:propgradfllow}]
	Fix $t \in [0,\tchi)$ and $i,j\in\{1,\ldots,P\}$ with $i \neq j$.
	Recalling Construction~\ref{DefinitionCalibration},
	we observe that $\xi_{i,j} \in \{\pm \xi, \pm \smash{\frac{1}{2}}\xi,0\}$,
	so that it suffices to estimate in terms of the vector field~$\xi$. 
	Recalling its definition~\eqref{eq:defxi} it follows that
	\begin{equation}
		\begin{aligned}
			\label{eq:shiftedXi}
			\xi^{z,T}(\cdot,t) &= 
			\eta\Big(\frac{\sdist_{\bar I^{z,T}}(\cdot,t)}{r_T(t)}\Big)
			\n_{\bar{I}^{z,T}}\big(P_{\bar{I}^{z,T}}(\cdot,t),t\big)
			\\&
			= \eta\Big(\frac{\sdist_{\bar I^{z,T}}(\cdot,t)}{r_T(t)}\Big)
			\big(\nabla \sdist_{\bar I^{z,T}}\big)(\cdot,t),
		\end{aligned}
	\end{equation}
	we directly compute
	\begin{align*}
		\big(\nabla \cdot \xi^{z,T}\big)(\cdot,t) 
		&= \frac{1}{r_T(t)} \eta'\Big(\frac{\sdist_{\bar I^{z,T}}(\cdot,t)}{r_T(t)}\Big)
		\\&~~~
		- \eta\Big(\frac{\sdist_{\bar I^{z,T}}(\cdot,t)}{r_T(t)}\Big)
		\frac{H_{\bar{I}^{z,T}}\big(P_{\bar{I}^{z,T}}(\cdot,t),t\big)}
		{1-H_{\bar{I}^{z,T}}\big(P_{\bar{I}^{z,T}}(\cdot,t),t\big)\sdist_{\bar I^{z,T}}(\cdot,t)},
	\end{align*}
	so that~\eqref{eq:estdivxi} follows from $|\eta'|\leq 16$,
	$|H_{\bar{I}}(\cdot,t)| \leq 2/r(t)$ due to Definition~\ref{DefinitionShrinkingCircle},
	and $|\sdist_{\bar I}(\cdot,t)|
	\leq r(t)/4$ on $\supp\xi(\cdot,t)$, $t \in [0,T_{ext})$.
	
	Furthermore, since
	\begin{align}
		\label{eq:PDEshiftedSignedDistance}
		\partial_t \sdist_{\bar{I}^{z,T}}(\cdot,t)
		= -H_{\bar{I}^{z,T}}\big(P_{\bar{I}^{z,T}}(\cdot,t),t\big)
		(1{+}\dot{\mathfrak{T}}) - \n_{\bar{I}^{z,T}}\big(P_{\bar{I}^{z,T}}(\cdot,t),t\big)\cdot\dot{z},
	\end{align}
	which itself one may either directly read off from the obvious
	generalization of~\eqref{eq:heuristics17} or
	alternatively from~\eqref{eq:shiftedSignedDistance},
	we also get
	\begin{align}
		\nonumber
		\big(\partial_t\xi^{z,T}\big)(\cdot,t)
		&= \eta\Big(\frac{\sdist_{\bar I^{z,T}}(\cdot,t)}{r_T(t)}\Big)
		\big(\nabla \partial_t\sdist_{\bar I^{z,T}}\big)(\cdot,t)
		\\&~~~
		\nonumber
		+ \frac{1}{r_T(t)} \eta'\Big(\frac{\sdist_{\bar I^{z,T}}(\cdot,t)}{r_T(t)}\Big)
		\big(\partial_t\sdist_{\bar I^{z,T}}\big)(\cdot,t)
		\n_{\bar{I}^{z,T}}\big(P_{\bar{I}^{z,T}}(\cdot,t),t\big)
		\\&~~~
		\nonumber
		+ \eta'\Big(\frac{\sdist_{\bar I^{z,T}}(\cdot,t)}{r_T(t)}\Big)
		\frac{\sdist_{\bar I^{z,T}}(\cdot,t)}{r_T^3(t)}(1{+}\dot{\mathfrak{T}})
		\n_{\bar{I}^{z,T}}\big(P_{\bar{I}^{z,T}}(\cdot,t),t\big)
		\\&
		\label{eq:PDEshiftedXi}
		= -\eta\Big(\frac{\sdist_{\bar I^{z,T}}(\cdot,t)}{r_T(t)}\Big)
		\frac{H'_{\bar{I}^{z,T}}\big(P_{\bar{I}^{z,T}}(\cdot,t),t\big) (1{+}\dot{\mathfrak{T}})}
		{1-H_{\bar{I}^{z,T}}\big(P_{\bar{I}^{z,T}}(\cdot,t),t\big)\sdist_{\bar I^{z,T}}(\cdot,t)}
		\\&~~~
		\nonumber
		+ \eta\Big(\frac{\sdist_{\bar I^{z,T}}(\cdot,t)}{r_T(t)}\Big)
		\frac{H_{\bar{I}^{z,T}}\big(P_{\bar{I}^{z,T}}(\cdot,t),t\big)
			\tau_{\bar{I}^{z,T}}\big(P_{\bar{I}^{z,T}}(\cdot,t),t\big)\cdot\dot{z}}
		{1-H_{\bar{I}^{z,T}}\big(P_{\bar{I}^{z,T}}(\cdot,t),t\big)\sdist_{\bar I^{z,T}}(\cdot,t)}
		\\&~~~
		\nonumber
		+ \frac{1}{r_T(t)} \eta'\Big(\frac{\sdist_{\bar I^{z,T}}(\cdot,t)}{r_T(t)}\Big)
		\big(\partial_t\sdist_{\bar I^{z,T}}\big)(\cdot,t)
		\n_{\bar{I}^{z,T}}\big(P_{\bar{I}^{z,T}}(\cdot,t),t\big)
		\\&~~~
		\nonumber
		+ \eta'\Big(\frac{\sdist_{\bar I^{z,T}}(\cdot,t)}{r_T(t)}\Big)
		\frac{\sdist_{\bar I^{z,T}}(\cdot,t)}{r_T^3(t)}(1{+}\dot{\mathfrak{T}})
		\n_{\bar{I}^{z,T}}\big(P_{\bar{I}^{z,T}}(\cdot,t),t\big).
	\end{align}
	Hence, \eqref{eq:esttimexi} follows from~\eqref{eq:PDEshiftedXi},
	\eqref{eq:PDEshiftedSignedDistance}, \eqref{eq:asymptotically_circular4},
	\eqref{ineq:zTprime}, and the estimates used for the derivation of~\eqref{eq:estdivxi}.
	
	Recalling Construction~\ref{DefinitionCalibration},
	we observe that $\vartheta_i \in \{\vartheta, 1\}$,
	so that it suffices to estimate in terms of the function~$\vartheta$.
	Recalling its definition~\eqref{eq:defxi} in the form of
	\begin{align}
		\label{eq:defWeightShifted}
		\vartheta^{z,T}(\cdot,t) = \frac{1}{r_T(t)}\bar\vartheta\Big(\frac{\sdist_{\bar I^{z,T}}(\cdot,t)}{r_T(t)}\Big),
	\end{align}
	we obtain
	\begin{equation}
		\label{eq:PDEWeightShifted}
		\begin{aligned}
			\big(\partial_t \vartheta^{z,T}\big)(\cdot,t) &= 
			\frac{(1{+}\dot{\mathfrak{T}})}{r_T^3(t)}
			\bar\vartheta\Big(\frac{\sdist_{\bar I^{z,T}}(\cdot,t)}{r_T(t)}\Big)
			\\&~~~
			+ \frac{1}{r_T(t)} \bar\vartheta'\Big(\frac{\sdist_{\bar I^{z,T}}(\cdot,t)}{r_T(t)}\Big)
			\frac{\big(\partial_t\sdist_{\bar I^{z,T}}\big)(\cdot,t)}{r_T(t)}
			\\&~~~
			+ \frac{1}{r_T(t)} \bar\vartheta'\Big(\frac{\sdist_{\bar I^{z,T}}(\cdot,t)}{r_T(t)}\Big)
			\frac{\sdist_{\bar I^{z,T}}(\cdot,t)}{r_T^3(t)}(1{+}\dot{\mathfrak{T}}),
		\end{aligned}
	\end{equation}
	so that, based on the previous ingredients, we have $|\bar\vartheta'|\leq 2$
	and $|\sdist_{\bar I}(\cdot,t)| \leq r(t)/4$ on $\supp\bar\vartheta'\big(\sdist_{\bar I}(\cdot,t)/(r(t)/4)\big)$, 
	$t \in [0,T_{ext})$, and we may deduce~\eqref{eq:esttimevartheta}.
\end{proof}

The second result is concerned with a crude upper bound for the mass
of the varifold~$\mathcal{V}_t$ in terms of the length scale~$r$
of the strong solution.

\begin{lemma}
	\label{prop:regInterfaceWeakSol}
	Fix $t \in (0,\tchi)$. The condition 
	$E[\mathcal{V}, \chi | \bar \chi^{z,T}](t)  \leq \delta {r_T(t)}$
	together with the requirements of Definition~\ref{DefinitionShrinkingCircle} implies
	\begin{align}\label{eq:propweaksol1} 		
		\frac12 \sum_{i, j=1, i \neq j}^P \mathcal{H}^{1}({I}_{i,j}(t))
		\leq \mu_t ({\Rd[2]} )
		\leq \widetilde{C } r_T(t),
	\end{align}
	where $\widetilde{C }\in [0,\infty)$ is a universal constant. 
\end{lemma}

\begin{proof}
	One can compute
	\begin{align*}
		&\sum_{i, j=1, i \neq j}^P \frac12 \int_{{I}_{i,j}(t)} 1 \, \dH
		\\&~~~
		\leq \mu_t ({\Rd[2]} ) = E_{\mathrm{int}}[\mathcal{V}, \chi| \bar\chi^{z,T}](t) +	
		\sum_{i=1}^P \int_{\Rd[2]}  \chi_i(\cdot,t) \nabla \cdot \xi_i^{z,T}(\cdot,t) \, \dx.
	\end{align*}
	Due to the computation below~\eqref{eq:shiftedXi}, the estimate~\eqref{eq:estdivxi},
	the properties of the cutoff function $\eta$ from Construction~\ref{DefinitionCalibration},
	and Definition~\ref{DefinitionShrinkingCircle}, it follows
	\begin{align*}
	\sum_{i=1}^P \int_{\Rd}  \chi_i(\cdot,t) \nabla \cdot \xi^{z,T}_i(\cdot,t) \, \dx
	\leq \int_{\{|\sdist_{\bar I^{z,T}}(\cdot,t)| \leq r_T(t)\}}
	|\nabla\cdot\xi^{z,T}| \,dx
	\lesssim r_T(t).
	\end{align*}
	Whence, we can deduce \eqref{eq:propweaksol1} from the previous two displays
	and the assumption $E[\mathcal{V}, \chi | \bar \chi^{z,T}](t)  \leq \delta {r_T(t)}$.
\end{proof} 

\begin{proof}[Proof of Lemma~\ref{lemma:stabilitybad}]
	Fix $t \in \mathcal{T}_{\mathrm{non\text{-}reg}}(\Lambda)$
	for yet to be chosen $\Lambda \gg 1$.
	For notational simplicity, let us in the sequel drop the dependence on~$t$ of all quantities.
	Since the definition of the error functionals is independent of
	the actual choice of the vector field~$B$, we may interpret
	the right hand sides of~\eqref{eq:prelimStabilityRelEntropy}
	and~\eqref{eq:prelimStabilityBulk} with $B \equiv 0$
	and therefore obtain for all $i,j\in\{1,\ldots,P\}$ with $i \neq j$
	\begin{equation}
		\begin{aligned}
			\label{eq:step0relineq100}
			&-\mathcal{D}_{i,j}[\chi|\bar\chi^{z,T}]
			+  RHS_{i,j}^{\mathrm{int}}[\chi|\bar\chi^{z,T}]
			\\&~~~
			=  - \int_{I_{i, j}} \left|V_{i, j}\right|^{2} \,d\mathcal{H}^{1}
			-  \int_{I_{i, j}}  V_{i, j}\nabla \cdot \xi^{z, T}_{i, j} \,d\mathcal{H}^{1}
			- \int_{I_{i, j}} \n_{i, j} \cdot \partial_{t} \xi^{z, T}_{i, j} \,d\mathcal{H}^{1},
		\end{aligned}
	\end{equation}
\begin{align}	\label{eq:step0relineqvar}
 RHS^{\operatorname{var-BV}}[\mathcal{V},\chi|\bar\chi^{z,T}](t)  = - \int_{\Rd} |H_\mu|^2 \Big(1- \frac12 \sum_{i=1}^P \rho_i \Big) \, {d}\mu_t 	\leq 0 
	,
\end{align}
	and for all $i\in\{1,\ldots,P{-}1\}$
	\begin{equation}
		\begin{aligned}
			\label{eq:step0relineq2}
			& RHS_{i}^{\mathrm{bulk}}[\chi|\bar\chi^{z,T}]
			=  \int_{\mathbb{R}^2}  (\chi_i - \bar \chi_i^{z,T} ) \partial_t \vartheta^{z,T}_i \,dx
			+  \sum_{j=1, \, i \neq j}^P  \int_{I_{i,j}} V_{i,j} \vartheta^{z,T}_{i}   \,d\mathcal{H}^{1}.
		\end{aligned}
	\end{equation}
Note that, since $\frac12 \sum_{i=1}^P \rho_i  \leq 1$, the right hand side of \eqref{eq:step0relineqvar} is nonpositive.
	Before we start estimating the right hand sides of~\eqref{eq:step0relineq100}
	and~\eqref{eq:step0relineq2}, we fix $\delta,\delta_{\mathrm{asymp}} \ll 1$ such that the conclusion of
	Lemma~\ref{prop:regInterfaceWeakSol} applies for the choice $\delta_{\mathrm{err}}=\frac{1}{2}$.
	
	From H\"{o}lder's inequality, \eqref{eq:propweaksol1} and~\eqref{eq:estdivxi}, we then directly infer
	\begin{align*}
		\bigg|\int_{I_{i, j}}  V_{i, j}\nabla \cdot \xi^{z, T}_{i, j} \,d\mathcal{H}^{1}\bigg|
		\lesssim \frac{1}{\sqrt{r_T}}
		\bigg(\int_{I_{i, j}}  |V_{i, j}|^2 \,d\mathcal{H}^{1}\bigg)^\frac{1}{2}.
	\end{align*}
	Similarly, we may estimate due to~\eqref{eq:esttimexi}
	\begin{align*}
		\bigg|\int_{I_{i, j}}  \n_{i, j} \cdot \partial_{t} \xi^{z, T}_{i, j} \,d\mathcal{H}^{1}\bigg|
		\lesssim \frac{1}{r_T}
	\end{align*}
	and, since $|\vartheta^{z,T}_i| \leq 1/r_T$, also
	\begin{align*}
		\bigg|\int_{I_{i, j}}  V_{i,j} \vartheta^{z,T}_{i} \,d\mathcal{H}^{1}\bigg|
		\lesssim \frac{1}{\sqrt{r_T}}
		\bigg(\int_{I_{i, j}}  |V_{i, j}|^2 \,d\mathcal{H}^{1}\bigg)^\frac{1}{2}.
	\end{align*}
	Finally, the estimates~\eqref{eq:propweaksol1}, \eqref{eq:asymptotically_circular1}
	and~\eqref{eq:esttimevartheta} together with the isoperimetric inequality imply
	\begin{align*}
		\bigg|\int_{\mathbb{R}^2}  (\chi_i - \bar \chi_i^{z,T} ) \partial_t \vartheta^{z,T}_i \,dx\bigg|
		\lesssim \frac{1}{r_T}.
	\end{align*}
	Plugging these estimates back into~\eqref{eq:step0relineq100}
	and~\eqref{eq:step0relineq2}, we may infer
	by an absorption argument
	the claim~\eqref{ineq:stabilitybad}
	from employing the defining condition~\eqref{def:badTimes} of non-regular times
	for $\Lambda \gg 1$.
\end{proof}

\begin{proof}[Proof of Corollary~\ref{cor:nonRegular}]
	Denote by $\widetilde\Lambda$ and $(\widetilde\delta,\widetilde\delta_{\mathrm{asymp}})$ the constants from Lemma~\ref{lemma:stabilitybad}.
	The choices $\Lambda := \max\{\widetilde\Lambda,10\}$ and 
	$(\delta,\delta_{\mathrm{asymp}}):=(\widetilde\delta,\widetilde\delta_{\mathrm{asymp}})$ then imply the claim.
	Indeed, for $t \in \mathcal{T}_{\mathrm{non\text{-}reg}}(\Lambda)$
	satisfying the assumption $E[\mathcal{V},\chi|\bar\chi^{z,T}](t) \leq \delta r_T(t)$, it follows
	from the defining condition~\eqref{def:badTimes}
	\begin{align*}
		\frac{5}{r_T^2(t)} E[\mathcal{V},\chi|\bar\chi^{z,T}](t)
		\leq \frac{5}{r_T(t)} \leq \frac{1}{2}\frac{\Lambda}{r_T(t)}
		\leq \int_{\Rd} |H_\mu(\cdot,t)|^2 \, d\mu_t
		,
	\end{align*}
	so that the validity of~\eqref{ineq:stabilitybad} implies~\eqref{ineq:stabilitybad2}.
\end{proof}

\subsection{Proof of Lemma \ref{lemma:relenineqgraph}: Stability 
	estimate in perturbative setting~I}\label{proof:perreg1}
The asserted bound~\eqref{ineq:graphset} 
follows directly from the estimates~\eqref{eq:aux1}--\eqref{eq:aux14}
established in Lemma~\ref{lem:perturbativeComputations} in Section~\ref{sec:appendix}
and the fact that $RHS^{\operatorname{var-BV}}[\mathcal{V},\chi|\bar\chi^{z,T}](t) = 0$
due to the conclusions of Proposition~\ref{theo:graph}.
\qed

\subsection{Proof of Lemma~\ref{lemma:relenineqpolar}: Stability 
	estimate in perturbative setting~II}\label{proof:perreg2}
For notational simplicity, we again neglect the dependence on~$t$ of all quantities.
Our proof of the estimate~\eqref{ineq:freaze} proceeds in several steps.

\textit{Step~1: Leading order terms involving~$H_{\bar{I}^{z,T}}'$.}
We start by providing a preliminary estimate for the last three
right hand side terms of~$R_{l.o.t.}$ from Lemma~\ref{lemma:relenineqgraph}.
To this end, for each of the three terms we make use of Definition~\ref{DefinitionShrinkingCircle}
in the form of $|H_{\bar{I}^{z,T}}'|\leq \delta_{\mathrm{asymp}}/r_T^2$.
Hence, by Young's inequality and $|H_{\bar{I}^{z,T}}| \leq 2/r_T$ 
\begin{align}
	\label{eq:auxFrozen1}
	\int_{\bar{I}^{z,T}} 2 H_{\bar{I}^{z,T}} H_{\bar{I}^{z,T}}' h h' \,d\mathcal{H}^1
	\lesssim \delta_{\mathrm{asymp}} \int_{\bar{I}^{z,T}}
	\frac{1}{r_T^2} (h')^2 + \frac{1}{r_T^4} h^2 \,d\mathcal{H}^1.
\end{align}
Furthermore, by the defining ODE for the space-time shift
in the form of~\eqref{eq:evzTgraph}, Jensen's inequality and~\eqref{eq:asymptotically_circular1},
we obtain
\begin{align}
	\label{eq:auxFrozen2}
	- \int_{\bar{I}^{z,T}} H_{\bar{I}^{z,T}}'\big(\tau_{\bar{I}^{z,T}}\cdot\dot{z}\big) h \,d\mathcal{H}^1 
	\lesssim \delta_{\mathrm{asymp}}\int_{\bar{I}^{z,T}} \frac{1}{r_T^4} h^2 \,d\mathcal{H}^1 
\end{align}
as well as
\begin{align}
	\label{eq:auxFrozen3}
	- \int_{\bar{I}^{z,T}} H_{\bar{I}^{z,T}}'\dot{\mathfrak{T}} h' \,d\mathcal{H}^1 
	\lesssim \delta_{\mathrm{asymp}}\int_{\bar{I}^{z,T}}
	\frac{1}{r_T^2} (h')^2 + \frac{1}{r_T^4} h^2 \,d\mathcal{H}^1,
\end{align}
where for the latter we also used Young's inequality.

\textit{Step~2: Freezing of coefficients in leading order quadratic terms.}
As a simple consequence of~\eqref{eq:asymptotically_circular3},
$|H_{\bar{I}^{z,T}}|\leq 2/r_T$ and $a^2-b^2=(a-b)(a+b)$, it holds
\begin{equation}
	\label{eq:auxFrozen4}
	\begin{aligned}
		&\int_{\bar{I}^{z,T}} \bigg(\frac{3}{2} H_{\bar{I}^{z,T}}^2 - \frac{1}{r_T^2}\bigg) (h')^2 \,d\mathcal{H}^1
		+ \int_{\bar{I}^{z,T}} \frac{1}{r_T^2}\bigg(\frac{1}{2}
		H_{\bar{I}^{z,T}}^2 + \frac{1}{r_T^2}\bigg)	h^2\,d\mathcal{H}^1
		\\&~~~
		\leq \int_{\bar{I}^{z,T}} \frac{1}{2} \frac{1}{r_T^2} (h')^2 
		+ \frac{3}{2} \frac{1}{r_T^4} h^2 \,d\mathcal{H}^1
		+ \frac{9}{2}\delta_{\mathrm{asymp}}\int_{\bar{I}^{z,T}}
		\frac{1}{r_T^2} (h')^2 + \frac{1}{r_T^4} h^2 \,d\mathcal{H}^1.
	\end{aligned}
\end{equation}

\textit{Step~3: Freezing of coefficients in leading order correction terms.}
By the arguments from the previous two steps, we may estimate
\begin{equation}
	\label{eq:auxFrozen5}
	\begin{aligned}
		&- \int_{\bar{I}^{z,T}} \Big(\frac{1}{r_T^2} + H_{\bar{I}^{z,T}}^2 \Big) h
		\n_{\bar{I}^{z,T}} \cdot \dot{z} \,d\mathcal{H}^1 
		- \int_{\bar{I}^{z,T}} \frac{1}{r^2_T} H_{\bar{I}^{z,T}}
		h \dot{\mathfrak{T}}\,\,d\mathcal{H}^1
		\\&~~~
		\leq - \int_{\bar{I}^{z,T}} \frac{2}{r_T^2} h
		\n_{\bar{I}^{z,T}} \cdot \dot{z} \,d\mathcal{H}^1 
		- \int_{\bar{I}^{z,T}} \frac{1}{r^3_T}
		h \dot{\mathfrak{T}}\,\,d\mathcal{H}^1
		\\&~~~~~~
		+ \widetilde{C} \delta_{\mathrm{asymp}}\int_{\bar{I}^{z,T}}
		\frac{1}{r_T^4} h^2 \,d\mathcal{H}^1,
	\end{aligned}
\end{equation}
where $\widetilde C > 0$ is some universal constant.

\textit{Step~4: Change of variables in quadratic terms.}
Recalling the definition
$[0,2\pi) \ni \theta \mapsto h\big(\bar{\gamma}^{z,T}(\frac{L_{\bar{\gamma}^{z,T}}}{2\pi}\theta)\big)$,
a simple change of variables together with condition~\eqref{eq:asymptotically_circular1} entails
\begin{align}
	\label{eq:auxFrozen6a}
	\frac{1}{(1 {+} \delta_{\mathrm{asymp}})^3} \frac{1}{r_T^3} \int_{0}^{2\pi } (\partial_\theta ^2 \widetilde h)^{2}\,d\theta
	&\leq \int_{\bar{I}^{z,T}} (h'')^2 \,d\mathcal{H}^1,
	\\
	\label{eq:auxFrozen6b}
	\int_{\bar{I}^{z,T}} (h'')^2 \,d\mathcal{H}^1
	&\leq \frac{1}{(1 {-} \delta_{\mathrm{asymp}})^3} \frac{1}{r_T^3} \int_{0}^{2\pi } (\partial_\theta ^2 \widetilde h)^{2}\,d\theta,
	\\
	\label{eq:auxFrozen7}
	\int_{\bar{I}^{z,T}} \frac{1}{r_T^2} (h')^2 \,d\mathcal{H}^1
	&\leq \frac{1}{(1 {-} \delta_{\mathrm{asymp}})} \frac{1}{r_T^3} \int_{0}^{2\pi } (\partial_\theta \widetilde h)^{2}\,d\theta,
	\\
	\label{eq:auxFrozen8}
	\int_{\bar{I}^{z,T}} \frac{1}{r_T^4} h^2 \,d\mathcal{H}^1
	&\leq (1 {+} \delta_{\mathrm{asymp}}) \frac{1}{r_T^3} \int_{0}^{2\pi} \widetilde h^2 \,d\theta. 
\end{align}

\textit{Step~5: Change of variables in correction terms.}
We claim that
\begin{equation}
	\begin{aligned}
		&\bigg|- \int_{\bar{I}^{z,T}} \frac{2}{r_T^2} h
		\n_{\bar{I}^{z,T}} \cdot \dot{z} \,d\mathcal{H}^1
		- \bigg(- 6 \frac{1}{r_T^3(t)} \Big|\frac{1}{\sqrt{\pi}}\int_0^{2\pi} \widetilde h(\cdot,t) e^{i\theta} \,d \theta \Big|^2\bigg)\bigg|
		\\&~~~
		+ \bigg|- \int_{\bar{I}^{z,T}} \frac{1}{r^2_T} H_{\bar{I}^{z,T}}
		h \dot{\mathfrak{T}}\,\,d\mathcal{H}^1
		- \bigg(- 4 \frac{1}{r_T^3(t)} \Big(\frac{1}{\sqrt{2\pi}}\int_0^{2\pi} \widetilde h(\cdot,t)  \,d \theta \Big)^2\bigg)\bigg|
		\\&~~~~~~~~~
		\leq \widetilde{C}\delta_{\mathrm{asymp}}
		\frac{1}{r_T^3} \int_{0}^{2\pi } \widetilde h^2 \,d\theta,
	\end{aligned}
\end{equation}
where $\widetilde{C} > 0$ is some universal constant.
Indeed, this follows similarly to the previous steps,
exploiting in the process the two conditions~\eqref{eq:asymptotically_circular1}
and~\eqref{eq:asymptotically_circular2} as well as the
defining ODE of the space-time shift~\eqref{eq:evzTgraph}.

\textit{Step~6: Conclusion.} Based on the previous steps,
we infer that, for given $\delta_{\mathrm{err}} \in (0,1)$, 
one may choose $\delta_{\mathrm{asymp}} \ll_{\delta_{\mathrm{err}}} 1$ 
such that the leading order contribution $R_{l.o.t.}$
from Lemma~\ref{lemma:relenineqgraph} is 
estimated by $\widetilde{R}_{l.o.t.} + \frac{1}{2}\widetilde{R}_{h.o.t.}$.
Since the higher order contribution $R_{h.o.t.}$
from Lemma~\ref{lemma:relenineqgraph} can be easily estimated
in terms of $\frac{1}{2}\widetilde{R}_{h.o.t.}$ for a suitable choice of
$\delta_{\mathrm{asymp}} \ll_{\delta_{\mathrm{err}}} 1$ by
means of the previous arguments, this concludes the proof 
of Lemma~\ref{lemma:relenineqpolar}.
\qed

\subsection{Proof of Lemma \ref{lemma:stabilitygood}: Final stability estimate in perturbative setting} \label{proof:perreg3}
First, we observe that by Lemma~\ref{lemma:relentfreaze} and the
estimates~\eqref{eq:auxFrozen7}--\eqref{eq:auxFrozen8} from the previous proof
that, for given $\delta_{\mathrm{err}} \in (0,1)$, one may choose
$C,C'\gg_{\delta_{\mathrm{err}}} 1$ and $\delta_{\mathrm{asymp}} \ll_{\delta_{\mathrm{err}}} 1$
such that
\begin{align}
	\label{eq:auxFinalStabilityGraph1}
	E[\chi|\bar\chi^{z,T}] \leq (1 {+} \delta_{\mathrm{err}}) 
	\frac{1}{r_T} \frac{1}{2} \big\|\widetilde h\big\|^2_{H^1(0,2\pi)}
	=: RHS.
\end{align}
Second, thanks to Lemma~\ref{lemma:relenineqpolar}, for given $\delta_{\mathrm{err}} \in (0,1)$, one may choose
$C,C'\gg_{\delta_{\mathrm{err}}} 1$ and $\delta_{\mathrm{asymp}} \ll_{\delta_{\mathrm{err}}} 1$
such that
\begin{align}
	\nonumber
	&\sum_{i,j=1,i\neq j}^{P} \frac{1}{2} 
	\Big( -\mathcal{D}_{i,j}[\chi|\bar\chi^{z,T}]
	+  \frac{1}{2} RHS_{i,j}^{\mathrm{int}}[\chi|\bar\chi^{z,T}] \Big)
	+ \sum_{i=1}^{P-1} RHS_{i}^{\mathrm{bulk}}[\chi|\bar\chi^{z,T}]
	\\&~~~ \label{eq:auxFinalStabilityGraph2}
	\leq - \frac{1}{r^3_T}\int_{0}^{2\pi } (1 {-} \delta_{\mathrm{err}}) (\partial_\theta ^2 \widetilde h)^{2}
	- (1 {+} \delta_{\mathrm{err}}) \frac{1}{2} (\partial_\theta  \widetilde h)^{2}
	- (1 {+} \delta_{\mathrm{err}}) \frac{3}{2} \widetilde h^2 \,d \theta
	\\&~~~~~~ \nonumber
	- 4 \frac{1}{r_T^3} \bigg(\frac{1}{\sqrt{2\pi}}\int_0^{2\pi} \widetilde h \,d \theta \bigg) ^2
	- 6 \frac{1}{r_T^3} \bigg|\frac{1}{\sqrt{\pi}}\int_0^{2\pi} \widetilde h e^{i\theta} \,d \theta \bigg| ^2
	\\&~~~ \nonumber
	=: LHS.
\end{align}
Now, fix $\alpha \in (1,5)$. We claim that there exist
$\delta_{\mathrm{err}} \ll_\alpha 1$ as well as a choice of the
constant $C_\zeta \gg_\alpha 1$ from Construction~\ref{DefinitionErrorHeights}
such that
\begin{align}
	\label{eq:auxFinalStabilityGraph3}
	LHS \leq 
	- \frac{\alpha}{r_T^2} (1{+}\dot{\mathfrak{T}}) RHS,
\end{align}
so that the claim~\eqref{ineq:stabilityregular} follows 
from~\eqref{eq:auxFinalStabilityGraph1}--\eqref{eq:auxFinalStabilityGraph3}.
Fourier decomposing both sides of the asserted inequality~\eqref{eq:auxFinalStabilityGraph3},
we may indeed derive the validity of~\eqref{eq:auxFinalStabilityGraph3} 
for suitably chosen $\delta_{\mathrm{err}} \ll_{\alpha} 1$
and $C_\zeta \gg_\alpha 1$
analogously to our analysis towards the end of Subsection~\ref{subsec:heuristicsDecay}
(cf.\ \eqref{eq:heuristics28}--\eqref{eq:heuristics29b}), exploiting in the process
also the bound~\eqref{ineq:zTprime}. 
\qed

\subsection{Proof of Theorem \ref{theo:stability}: Overall a priori stability estimate}\label{subsec:proofstability}
The stability estimate~\eqref{ineq:aprioristability} follows directly from combining all results 
from Subsections~\ref{subsec:prelimStability}--\ref{subsec:stabilityEstimateGraph}, in particular Lemma~\ref{lem:preliminaryStability}, Corollary~\ref{cor:nonRegular}, and Lemma~\ref{lemma:stabilitygood}. 
\qed

Note in this
context that assumption~\eqref{hpEt} implies
$E[\mathcal{V},\chi|\bar{\chi}^{z,T}](t) \leq \delta r_T(t)$ for all $t \in (0,\tchi)$.

\section{Construction and properties of space-time shifts}

\subsection{Proof of Lemma \ref{lemma:existzT}: Existence of space-time shifts} \label{subsec:existencespacetimeshifts}
Our aim is to prove the existence of a time horizon~$\tchi \in (0,\infty)$, a locally Lipschitz map
$z\colon [0,\tchi) \rightarrow \Rd$ and a strictly increasing Lipschitz map
$T  =: \mathrm{id} {+} \mathfrak{T}\colon[0,\tchi) \rightarrow [0, \infty)$ 
such that $(z(0), T(0))=(0,0)$ and
\begin{align}
	\label{eq:auxExistenceShifts1}
	\tchi &= \sup \Big\{t: T(t) < \frac12 r_0^2 = T_{ext}\Big\},
	\\
	\label{eq:auxExistenceShifts2}
	\begin{bmatrix}
		\dot{z}(t) \\ \dot{\mathfrak{T}}(t)
	\end{bmatrix} 
	&= 	F(z(t),T(t),t), \quad t \in (0,\tchi),
\end{align}
where 
\begin{align}
	\label{eq:auxExistenceShifts3}
	F(z,T,t) := 
	\begin{bmatrix} \frac{6}{r_T^2(t)} \dashint_{\bar I^{z,T}(t)} 
		\rho(\cdot,t;z,T) \n_{\bar{I}^{z,T}}(\cdot,t) \,d\mathcal{H}^1 \\ 
		\frac{4}{r_T(t)} \dashint_{\bar I^{z,T}(t)} \rho(\cdot,t;z,T) \,d\mathcal{H}^1
	\end{bmatrix}.
\end{align}
Note that the asserted Lipschitz bounds~\eqref{ineq:zTprime}
are then immediate consequences of integrating~\eqref{eq:auxExistenceShifts2}
and $|\rho(x,t;z,T)| \leq r_T(t)/(8C_{\zeta})$, cf.\ Construction~\ref{DefinitionErrorHeights}	.

The proof of existence of the solution is obtained by successive approximations and an application of
the Picard--Lindel\"{o}f argument. 
We have to resort to an approximation argument to circumvent blowing up constants
(originating from negative powers of $r_T(t)$ for $t \to t_\chi$) 
preventing the use of the Picard--Lindel\"{o}f argument.
To this end, we introduce an auxiliary version of our problem labeled by integers $k \geq 1$, 
which reads as
\begin{align} \label{eq:truncprob}
	\begin{bmatrix}
		\dot{z}_k(t) \\ \dot{\mathfrak{T}}_k(t)
	\end{bmatrix} 
	= 		F_k(z_k(t),T_k(t),t), 
	\quad (z_k(0),T_k(0)) = (0,0),
\end{align}
where the right hand side $F_k\colon \Rd[2] {\times}
[0,\infty) {\times} [0,\infty) \to \mathbb{R}^3$ 
is defined by truncation:
\begin{align}
	\label{eq:auxExistenceShifts4}
	F_k(z,T,t) = F\big(z,\min \big\{T, \tfrac12 r_0^2 (1 {-} \tfrac1k ) \big\},t\big),
	\quad t \in [0,\infty).
\end{align}

We will show below that the fixed point equation obtained from
integrating~\eqref{eq:truncprob} admits a unique solution
$z_k \in C_b([0,\infty);\mathbb{R}^2)$ and $T_k \in C_b([0,\infty);[0,\infty))$,
where $t \mapsto T_k(t)$ is strictly increasing such that 
\begin{align}
	\label{eq:auxExistenceShifts8}
	\frac{1}{2}t \leq T_k(t) \leq \frac{3}{2}t.
\end{align}
(The latter two properties are consequences of $|\dot{\mathfrak{T}}_k|\leq \frac{1}{2}$
due to the estimate $|\rho(x,t;z_k,\min \{T_k, \tfrac12 r_0^2 (1 {-} \tfrac1k ) \})| \leq 
r_T(t)/(8C_\zeta)$ and $C_\zeta \geq 1$.)

Taking the existence of such a sequence of solutions~$(z_k,T_k)_{k\geq 1}$ for granted for the moment, we
then define $t_0 := 0$ and for $k \geq 1$
\begin{align}
	\label{eq:auxExistenceShifts5}
	t_k := \sup \big\{t\colon T_k(t) < T_{ext} (1 {-} \tfrac1k ) \big\}.
\end{align}
By the properties of~$T_k$, the uniqueness of solutions to~\eqref{eq:truncprob},
as well as the definitions~\eqref{eq:auxExistenceShifts4} and~\eqref{eq:auxExistenceShifts5}, 
the sequence $(t_k)_{k\geq 1}$ is strictly increasing and bounded.
The solution to~\eqref{eq:auxExistenceShifts2} is then constructed by
\begin{align}
	\label{eq:auxExistenceShifts6}
	(z(t),T(t)) &:= 	(z_k(t),T_k(t)), \quad t \in [0, t_k ), 
	\\ \label{eq:auxExistenceShifts7}
	\tchi &:= \sup_{k \geq 1} t_k = \lim_{k\to\infty} t_k < \infty.
\end{align}
Note that~\eqref{eq:auxExistenceShifts6} is indeed well-defined
by uniqueness of solutions to~\eqref{eq:truncprob}, and that
the identities~\eqref{eq:auxExistenceShifts1}--\eqref{eq:auxExistenceShifts2}
hold true by construction. Hence, it remains to verify the existence of
solutions to~\eqref{eq:truncprob} with the asserted properties.

Fix an integer $k \geq 1$.
In order to apply the Picard--Lindel\"{o}f argument, we have to show
that for given $t \in (0,\infty)$, the function $(z,T)\rightarrow F_k(z,T,t)$ 
is globally Lipschitz with Lipschitz constant independent of~$t$. 
For notational convenience, we abbreviate the truncation
by $\widehat{T} := \big\{T, \tfrac12 r_0^2 (1 {-} \tfrac1k )\}$.
First, we compute
\begin{align*}
	\frac{1}{r_{\widehat{T}}^2(t)} \frac{1}{\mathcal{H}^1(\bar{I}^{z,\widehat{T}}(t))}
	= \frac{1}{2\pi}\frac{1}{r_{\widehat{T}}^3(t)} 
	\frac{2\pi r_{\widehat{T}}(t)}{\mathcal{H}^1(\bar{I}^{0,\widehat{T}}(t))},
\end{align*}
so that the normalization factor has the required Lipschitz regularity due to
the action of the truncation and the smoothness of the evolution of~$\bar\chi$. 
Second, since the Jacobian of the tubular neighborhood diffeomorphism
$$x \mapsto (P_{\bar{I}^{z,\widehat{T}}}(x,t),\sdist_{\bar{I}^{z,\widehat{T}}}(x,t))$$ is given 
by $x \mapsto 1/(1 {-} (H_{\bar{I}^{z,\widehat{T}}} \circ P_{\bar{I}^{z,\widehat{T}}})(x,t)
\sdist_{\bar{I}^{z,\widehat{T}}}(x,t))$,
plugging in the definition~\eqref{eq:defrho_plus} together with
a change of variables yields
\begin{align*}
	&\int_{\bar I^{z,\widehat{T}}(t)} 
	\rho_+(\cdot,t;z,\widehat{T}) \n_{\bar{I}^{z,\widehat{T}}}(\cdot,t) \,d\mathcal{H}^1
	\\&
	= \int_{\{0 \leq \sdist_{\bar{I}^{z,\widehat{T}}}(\cdot,t) \leq r_{\widehat{T}}(t)/8\}}
	\big(1 {-} (H_{\bar{I}^{z,\widehat{T}}} \circ P_{\bar{I}^{z,\widehat{T}}})(\cdot,t)
	\sdist_{\bar{I}^{z,\widehat{T}}}(\cdot,t)\big) 
	\\&~~~~~~~~~~~~~~~~~~~~~~~~~~~~~~~~~~~~ 
	\big(\bar{\chi}^{z,\widehat{T}}_1 {-} \chi_1\big)(\cdot,t)
	\zeta\Big(\frac{\sdist_{\bar{I}^{z,\widehat{T}}}(\cdot,t)}{r_{\widehat{T}}(t)}\Big)
	\nabla \sdist_{\bar{I}^{z,\widehat{T}}}(\cdot,t) \,dx.
\end{align*}
Shifting variables in space, we obtain from 
the relations~\eqref{eq:shiftedSignedDistance}--\eqref{eq:shiftedProjection}
\begin{align*}
	&\int_{\bar I^{z,\widehat{T}}(t)} 
	\rho_+(\cdot,t;z,\widehat{T}) \n_{\bar{I}^{z,\widehat{T}}}(\cdot,t) \,d\mathcal{H}^1
	\\&
	= \int_{\{0 \leq \sdist_{\bar{I}^{0,\widehat{T}}}(\cdot,t) \leq r_{\widehat{T}}(t)/8\}}
	\big(1 {-} (H_{\bar{I}^{0,\widehat{T}}} \circ P_{\bar{I}^{0,\widehat{T}}})(\cdot,t)
	s_{\bar{I}^{0,\widehat{T}}}(\cdot,t)\big) 
	\\&~~~~~~~~~~~~~~~~~~~~~~~~~~~~~~~~~~~
	\big(\bar{\chi}^{0,\widehat{T}}_1 {-} \chi^{-z,\mathrm{id}}_1\big)(\cdot,t)
	\zeta\Big(\frac{\sdist_{\bar{I}^{0,\widehat{T}}}(\cdot,t)}{r_{\widehat{T}}(t)}\Big)
	\nabla \sdist_{\bar{I}^{0,\widehat{T}}}(\cdot,t) \,dx,
\end{align*}
and the required estimate follows from this representation
by smoothness of the evolution of~$\bar\chi$, the action of the truncation,
and Lipschitz continuity of translations of volumes.
Since an analogous formula also holds for~$\rho_+$ replaced by~$\rho_-$,
this concludes the proof (by resorting to Banach's fixed point theorem
exactly as in the proof of the Picard-Lindel\"{o}f theorem). \qed

\subsection{Proof of Lemma~\ref{prop:boundzT}: Bounds for space-time shifts}\label{subsec:boundsShifts}
Our goal is to prove~\eqref{eq:boundzT}--\eqref{eq:boundzT2}.
Fix $t \in (0,\tchi)$. Note that from~\eqref{def:bulkError}, 
Construction~\ref{DefinitionCalibration}, a change to
tubular neighborhood coordinates, $|H_{\bar{I}^{z,T}}(\cdot,t)| \leq 2/r_T(t)$,
and~\eqref{eq:defrho}	it follows that
\begin{equation}
	\label{eq:aux1BoundsShifts}
	\begin{aligned}
		&E_{\mathrm{bulk}}[\chi | \bar{\chi}^{z,T}](t)
		\\&
		\geq \int_{\{\dist(\cdot,\bar{I}^{z,T}(t)) < r_T(t)/8\}}
		\big|\chi_1(\cdot,t) {-} \bar\chi_1^{z,T}(\cdot,t)\big| \big|\vartheta_1(\cdot,t)\big|\,  dx 
		\\&
		=  \int_{\bar{I}^{z,T}(t)} \int_{-\frac{r_T(t)}{8}}^{\frac{r_T(t)}{8}}
		\frac{(\chi_1 {-} \bar\chi_1^{z,T})\big(\cdot{+}s\n_{\bar{I}^{z,T}}(\cdot,t),t\big)}
		{1 {-} H_{\bar{I}^{z,T}}(\cdot,t)s} \frac{-s}{r_T^2(t)} \,ds d\mathcal{H}^1
		\\&
		\geq \frac{4}{5} \int_{\bar{I}^{z,T}(t)} \frac{1}{2} \frac{1}{r_T^2(t)} 
		\Big(\rho_+^2(\cdot,t;z,T) + \rho_-^2(\cdot,t;z,T)\Big) \,d\mathcal{H}^1
		\\&
		\geq \frac{1}{10} \int_{\bar{I}^{z,T}(t)} \Big(\frac{\rho(\cdot,t;z,T)}{r_T(t)}\Big)^2
		\,d\mathcal{H}^1.
	\end{aligned}
\end{equation}
Hence, plugging in~\eqref{eq:evzT}, 
recalling~\eqref{eq:asymptotically_circular1},
and using 
Jensen's inequality
\begin{equation}
	\label{eq:aux2BoundsShifts}
	\begin{aligned}
		\frac{1}{r_0} |z(t)| &\leq \int_0^t \frac{1}{r_0} |\dot{z}(s)| \,ds
		\\& 
		\lesssim  (1+ \delta_{\mathrm{asymp}})\frac{1}{{r_0}} \int_0^t  \frac{1}{r_T^2(s)}
		\bigg(\,\dashint_{\bar{I}^{z,T}(s)} \big|\rho(\cdot,s;z,T)\big|^2 \,d\mathcal{H}^1\bigg)^\frac{1}{2} \,ds
		\\&
		\lesssim  (1+ \delta_{\mathrm{asymp}})  \frac{1}{{r_0}} \int_0^t \frac{1}{r_T^{3/2}(s)}
		\bigg(\int_{\bar{I}^{z,T}(s)} \Big(\frac{\rho(\cdot,s;z,T)}{r_T(s)}\Big)^2 
		\,d\mathcal{H}^1\bigg)^\frac{1}{2} \,ds.
	\end{aligned}
\end{equation}
Inserting the estimate~\eqref{eq:aux2BoundsShifts} into~\eqref{eq:aux1BoundsShifts}
and afterwards exploiting the assumption~\eqref{hpEt} together with \eqref{ineq:aprioristability} and  \eqref{ineq:zTprime} further entail
\begin{align*}
	\frac{1}{r_0} |z(t)| &\lesssim \sqrt{\frac{1}{r_0} E[\mathcal{V}_0, \chi_0|\bar\chi_0]}
	\int_0^t  \frac{1}{r_T^2(s)} \Big(\frac{r_T(s)}{r_0}\Big)^\frac{1}{2} \,ds,
\end{align*}
which in turn by~\eqref{ineq:zTprime} upgrades to
\begin{align*}
	\frac{1}{r_0} |z(t)| &\lesssim 
	\sqrt{\frac{1}{r_0} E[\mathcal{V}_0,\chi_0|\bar\chi_0]}
	\int_0^t  \frac{1}{r_T^2(s)} \Big(\frac{r_T(s)}{r_0}\Big)^\frac{1}{2}
	\big(1 {+} \dot{\mathfrak{T}}(s)) \,ds
	\\&
	= - \sqrt{\frac{1}{r_0} E[\mathcal{V}_0, \chi_0|\bar\chi_0]}
	\int_0^t \frac{d}{ds} \Big(\frac{r_T(s)}{r_0}\Big)^\frac{1}{2} \,ds
	\\&
	\lesssim  \sqrt{\frac{1}{r_0} E[\mathcal{V}_0, \chi_0|\bar\chi_0]} 
	\bigg(1 - \Big(\frac{r_T(t)}{r_0}\Big)^\frac{1}{2}\bigg).
\end{align*}
Now, choosing~$\delta_{\mathrm{err}} \ll 1$ such that the implicit
constant in the last estimate gets canceled, we obtain the claim
for the path of translations~$z$.
Analogously, one derives a bound of same type for $\frac{1}{T_{ext}}|\mathfrak{T}(t)|$.
\qed

\section{Reduction to perturbative graph setting}

\subsection{Proof of Proposition~\ref{theo:graph}: Strategy and intermediate results} \label{proof:reductiontograph}
We fix $\Lambda > 0$ and let $t \in \mathcal{T}_{\mathrm{reg}}(\Lambda)$, namely
$t \in (0,\tchi)$ such that \eqref{boundgoodtimes}, i.e.,
\begin{align*} 
	\int_{\Rd} |H_\mu(\cdot,t)|^2 \,d\mu_t 
	< \Lambda\frac{2\pi}{r_T(t)},
\end{align*}
holds. Given $C_\zeta \geq 1$ from Construction~\ref{DefinitionErrorHeights} and given any
$C, C' \geq 1$ 
(representing the constants from~\eqref{eq:smallC0norm}--\eqref{eq:smallC1norm}), 
we aim to find a 
constant $\delta \ll_{\Lambda,C,C',C_\zeta} \frac{1}{2}$  such that 
the assumption~\eqref{hp:regularitygoodtimes}, i.e.,
\begin{align*}
	E[\mathcal{V}, \chi | \bar \chi^{z,T}](t)  \leq \delta {r_T(t)} 
\end{align*}
implies the conclusions of Proposition~\ref{theo:graph}. 
From now on, let us suppress the dependence of all quantities on $t$.

The proof of Proposition~\ref{theo:graph} leverages on two results
from the literature (for the precise statements, see Theorem \ref{theo:Allard} and Theorem \ref{theo:Jordan} in Appendix \ref{appendix}):
\begin{itemize}
	\item Allard's regularity theory \cite[Chapter 5, Theorem 23.1 and Remark 23.2(a)]{SimonLectures}
	\item The decomposition of the reduced boundary of a set of finite
	perimeter in~$\mathbb{R}^2$ into a countable family of rectifiable 
	Jordan-Lipschitz curves \cite[Section 6, Theorem 4]{Ambrosio}.
\end{itemize}

The idea for the proof is then roughly speaking the following.
Our assumptions~\eqref{boundgoodtimes}--\eqref{hp:regularitygoodtimes} 
and Allard's regularity theory first provide us with a scale~$\varrho \ll r_T$
(uniform in $x_0 \in \supp\mu$) such that~$\supp\mu$ admits a local graph representation
on that scale at each $x_0 \in \supp\mu$. In addition, the assumption~\eqref{hp:regularitygoodtimes}
together with the coercivity properties of the error functional~$E[\mathcal{V}, \chi | \bar \chi^{z,T}]$
allow to show that any part of the set $\supp\mu$ not being in accordance with the
asserted graph representation of Proposition~\ref{theo:graph} has to be of sufficiently small mass.
In fact, by a suitable choice of the constant~$\delta$ from assumption~\eqref{hp:regularitygoodtimes}
and exploiting the decomposition result from~\cite{Ambrosio}
(i.e., the Lipschitz parametrization of the curves), one may trap any undesired
behavior of $\supp\mu$ within balls of radius $\varrho /2$. This, however, contradicts
the local graph property of $\supp\mu$ within balls of radius $\varrho$.

Keeping this heuristic in mind, we proceed by stating several
intermediate results which combined entail a proof of Proposition~\ref{theo:graph}.
As a first step, we ensure that our assumptions \eqref{boundgoodtimes}--\eqref{hp:regularitygoodtimes}
imply the applicability of Allard's regularity theory.

\begin{lemma}[Applicability of Allard's regularity theory]\label{lemma:hpAllard}
  Let $(\varepsilon,\gamma,C_{Allard})$ be the constants from Theorem~\ref{theo:Allard}.
	There exist $\tilde \varepsilon \in (0,\varepsilon)$, 
	$\delta_{\mathrm{asymp}} \ll 1$, $\delta \ll_{\tilde \varepsilon, \Lambda}1$
	and $\tilde C \gg 1$ such that, for all $x_0 \in \supp \mu $ and 
	\begin{align*}
		\begin{cases}
			G= \Tan_{P_{\bar I^{z,T} }(x_0)} \bar I^{z,T} & \text{if } x_0 \in \{|\xi^{z,T}|\leq 1/2\},
			\\
			G \text{ arbitrary} & \text{if } x_0 \in \{|\xi^{z,T}|> 1/2\},
		\end{cases}
	\end{align*}
the assumptions of Allard's regularity theorem (see Theorem \ref{theo:Allard} for $p=2$) are fulfilled
at scale $\varrho := \frac{1}{\tilde C} \frac{\tilde\varepsilon^2}{2\pi (\Lambda {+} 1)} r_T$
in the stronger form of
	\begin{align*}
	\frac{\mu(B_{\varrho}(x_0))}{\mathrm{Vol}_1 \varrho}\leq 1+  \tilde  \varepsilon, \quad 
 \Big(\int_{B_{\varrho}(x_0)} |H_\mu|^2 \, \dx\Big)^\frac12  \varrho^\frac12 \leq \tilde \varepsilon , \quad 
  E_{\operatorname{tilt}}[x_0, \varrho, G] \leq \tilde \varepsilon^2.
\end{align*}
Furthermore, one may choose $\tilde \varepsilon \in (0, \varepsilon)$, 
$\delta_{\mathrm{asymp}} \ll 1$ such that the following properties are satisfied:
\begin{itemize}
	\item $\tilde \varepsilon \in (0, \varepsilon)$ is such that
	\begin{align}\label{eq:boundeps}
		\tilde \varepsilon \leq \frac23 \frac{1}{2 C_{Allard} C'}, 
		\quad \tilde \varepsilon \leq \frac14 \frac{1}{2 C_{Allard}} \frac{1}{16 \max\{C, C_\zeta\}},
	\end{align}
 \item for all $x, \tilde x_0 \in \bar I^{z,T}$ such that $|(x- \tilde x_0)\cdot \tau_{\bar I^{z,T}}(\tilde x_0)| \leq \gamma \varrho $, it holds
 \begin{align} \label{eq:boundsp}
 	\frac{|(x-\tilde x_0) \cdot \n_{\bar I^{z,T}}(\tilde x_0 )|}{r_T} \leq \frac14 \frac{1}{16 \max\{C, C_\zeta\}},
 \end{align}
\item defining $\tilde \alpha $ by $2 C_{Allard} \tilde \varepsilon = \tan \tilde \alpha $, for all $x_0 \in \{|\xi^{z,T}| \leq 1/2\}$ and all $x \in \partial B_{\gamma \varrho /2} (x_0)$ satisfying $|(x- x_0)\cdot \tau_{\bar I^{z,T}}(P_{\bar I^{z,T}} (x_0))| \geq  \frac{\gamma \varrho}{2} \cos \tilde \alpha $, it holds 
	\begin{align} \label{eq:boundP}
		|P_{\bar I^{z,T}} (x) - P_{\bar I^{z,T}} (x_0)| \geq \frac12  \frac{\gamma \varrho}{2} \cos \tilde \alpha,
	\end{align}
\end{itemize}
where $C,C'>0$ are from Proposition~\ref{theo:graph}, and $C_\zeta>0$ is from Construction~\ref{DefinitionErrorHeights}.
\end{lemma}
 We remark that the first bound in \eqref{eq:boundeps} will be needed to prove \eqref{eq:smallC1norm},
	the second bound in \eqref{eq:boundeps} together with \eqref{eq:boundsp} 
	will be needed to prove \eqref{eq:smallC0norm},
	whereas \eqref{eq:boundP} will be needed to prove Lemma \ref{lemma:candidategraph} below. 

As a second step, we show that the geometry 
of the varifold-BV solution~$(\mathcal{V},\chi)$ reduces
to the geometry of a two-phase BV solution.

\begin{lemma}[No other phases, hidden boundaries, and higher-multiplicity interfaces]\label{lemma:noothers}
	It holds
	\begin{align}\label{eq:noothers}
		\chi_2=...=\chi_{P-1}= 0, 
	\end{align}
and
\begin{align} \label{eq:unitdensity}
	\mathcal{V}= (\mathcal{H}^1\llcorner \supp |\nabla \chi_1|) 
	\otimes (\delta_{\Tan_x \supp |\nabla \chi_1| })_{x \in \supp |\nabla \chi_1|}.
\end{align}
\end{lemma}

Next, we guarantee that the (remaining) interface of the weak solution
is not located too far away from where we expect it to be (i.e., close
to the interface of the strong solution).

\begin{lemma}[No interface far away from $\bar I^{z,T}$]\label{lemma:noaway}
	It holds
	\begin{align}\label{eq:xiinclusion}
		\supp |\nabla \chi_1| \subseteq \{|\xi^{z,T}| > 1/2\} \subseteq \{\dist(\cdot, \bar I^{z,T}) \leq r_T/4\}.
	\end{align}
\end{lemma}

So far, we only argued that certain features of the weak solution
(contradicting the conclusions of Proposition~\ref{theo:graph}) are not present.
We now turn to the part of the argument guaranteeing in turn the existence of
a subset of the interface satisfying the required graph representation.

\begin{lemma}[Construction of a graph candidate]\label{lemma:candidategraph}
	There exists a Jordan-Lipschitz curve $J \subseteq \supp |\nabla \chi_1| \subseteq  \{|\xi^{z,T}| > 1/2\}$ such that $J$ can be considered as a graph over $\bar I^{z,T}$.
	In particular, 	there exists a height function $
	h\colon \bar{I}^{z,T} \rightarrow [- r_T/4, r_T/4]$
	such that 
	\begin{align} \label{eq:Jgraph}
		J= \{x \in \bar{I}^{z,T} : x + h(x) \n_{\bar{I}^{z,T}}(x)\}. 
	\end{align}
\end{lemma}

We then show that the previously found candidate for the graph representation
in fact saturates the whole interface of the weak solution.

\begin{lemma}[Interface is a graph over $\bar I^{z,T}$]\label{lemma:weakgraph}
	It holds 
	\begin{align}
		\supp |\nabla \chi_1| = J, 
	\end{align}
	where $J$ is the Jordan-Lipschitz curve from Lemma \ref{lemma:candidategraph}. 
\end{lemma}

Finally, we show that the associated height funcion $h$ over $\bar I^{z,T}$ 
satisfies the bounds \eqref{eq:smallC0norm}--\eqref{eq:smallC1norm} from the 
conclusions of Proposition \ref{theo:graph}. 

\begin{lemma}[Height function estimates]\label{lemma:heightestimates}
	The height function 
	$	h: \bar{I}^{z,T} \mapsto [- r_T/4, r_T/4]$
satisfies the regularity \eqref{eq:graphReg}, i.e.,
$$ 
h\in H^2(\bar{I}^{z,T}) $$ and 
the bounds \emph{\eqref{eq:smallC0norm}--\eqref{eq:smallC1norm}}, namely
	\begin{align}
		\label{eq:smallC0norm2}
		\|h\|_{L^\infty(\bar{I}^{z,T})} &\leq \frac{r_T}{16 \max\{C,C_\zeta\}},
		\\
		\label{eq:smallC1norm2}
		\|h'\|_{L^\infty(\bar{I}^{z,T}(t))} &\leq \frac{1}{C'}.
	\end{align}
\end{lemma}

Recalling the claims of Proposition~\ref{theo:graph}, it is immediate
that its proof simply follows now from a combination of the previous lemmas,
so that at this stage it remains to provide a proof for all the intermediate results
of this subsection (which is done in the next subsection).

As a technical ingredient for some of the previous auxiliary results, 
a specific coercivity property of the error functional is exploited.
More precisely, we will use the fact that our interface 
error $E_\mathrm{int}$ (cf. \eqref{varinterfaceerror}) controls the folding of the interface
in the following sense.

\begin{lemma}[Error control]\label{lemma:error}
	Let $\Omega_1\subseteq \mathbb{R}^2$ be a set of finite perimeter such that $\partial^\ast \Omega_1 = \supp |\nabla \chi_1|$.
	Fix $x_0 \in \{\dist(\cdot, \bar I^{z,T}) \leq r_T/4\}$ and consider $\varrho \ll r_T/4$ such that, for all $x \in B_{\varrho}(x_0)$, it holds $|\xi^{z,T}(x)- \n_{\bar I^{z,T}} (P_{\bar I^{z,T}} (x_0))| \leq 1/4$. 
	Define $G_{x_0}:=x_0 +  \Tan_{P_{\bar I^{z,T}} (x_0)} \bar I^{z,T} $ and denote by $P_{G_{x_0}}$ the nearest point projection onto~$G_{x_0}$.
	For $x \in \mathbb{R}^2$, we denote by  $(\Omega_1)_{P_{G_{x_0}} (x)}$  the one-dimensional slice 
	\[
	\Omega_1 \cap \{ P_{G_{x_0}} (x) + y\n_{\bar I^{z,T}}(P_{\bar I^{z,T}} (x_0))  : |y | \leq r_T/2\}.
	\]
	Then, there exists a constant $C_{\mathrm{err}}>0$ such that 
	\begin{align}
		\mathcal{H}^1\left( 
		B_{\varrho}(x_0) \cap
		\partial^\ast \Omega_1
		\cap 
		\{x: \mathcal{H}^0(\partial^\ast(\Omega_1)_{P_{G_{x_0}} (x)} ) >1\}
		\right)
		\leq C_{\mathrm{err}} E_\mathrm{int}[\chi|\bar \chi^{z,T}].
	\end{align}
\end{lemma}

\subsection{Proofs of the intermediate Lemmas}
We provide the proofs of the several intermediate results from
the previous subsection.

\begin{proof}[Proof of Lemma~\ref{lemma:hpAllard} (Applicability of Allard's regularity theory)]
Let $(\varepsilon,\gamma,C_{Allard})$ be the constants from Theorem~\ref{theo:Allard}.
	One may choose $\tilde \varepsilon \in (0, \varepsilon)$ so that the bounds in \eqref{eq:boundeps} are satisfied. 
	In addition, the bounds \eqref{eq:boundsp}--\eqref{eq:boundP} can be guaranteed by choosing $\tilde \varepsilon \ll \varepsilon$ and $\delta_{\mathrm{asymp}} \ll 1$ due to the uniform smoothness of the ball. 
	
	For the proof of the remaining assertions we distnguish between two cases:
	\begin{itemize}
		\item[(a)] $x_0 \in \supp \mu \cap \{|\xi^{z,T}| \leq 1/2\}$, i.e., points 
		located sufficiently far away from the interface~$\bar{I}^{z,T}$,
		 	\item[(b)] $x_0 \in \supp \mu \cap \{|\xi^{z,T}| > 1/2\}$, i.e., points 
		located sufficiently close to~$\bar{I}^{z,T}$.
	\end{itemize}

	\textit{Case (a):} Let $\tilde C>1$ and let $\varrho = \frac{1}{\tilde C}\frac{\tilde \varepsilon^2}{2 \pi (\Lambda +1) } r_T$. From the definition \eqref{boundgoodtimes} it follows that 
	\begin{align*} 
	\Big(	\int_{B_{\varrho}(x_0)} |H_\mu|^2 \,d\mu \Big) \varrho 
		< 	 \Lambda\frac{2\pi}{r_T} \varrho <  {\tilde \varepsilon}^2.
	\end{align*}
	We may further assume that $\tilde \varepsilon \in (0, \varepsilon)$ is small enough such that for $x_0 \in \{|\xi^{z,T}|\leq 1/2\}$ it holds $B_{\varrho}(x_0) \subseteq \{|\xi^{z,T}|\leq 3/4\}$. Hence, recalling \eqref{eq:multiplicity}, we have
	\begin{align*}
		\mu(B_{\varrho}(x_0)) &\leq \mu(B_{\varrho}(x_0) \cap \{\omega \leq 1/2 \})  + \mu(B_{\varrho}(x_0) \cap \{\omega =1 \}) \\
&	\leq 2 \int_{\mathbb{R}^2\cap  \{\omega \leq 1/2 \}} 1- \omega \, d \mu + 4 \int_{I_{1,P}} 1- \n_{P,1} \cdot \xi^{z,T}  \, d \mathcal{H}^1 \\
&\quad + 2 \sum_{i,j \notin \{1,P\} } \frac12 \int_{{I}_{i,j}} 1- \n_{i,j} \cdot \xi^{z,T}_{i,j} \, d \mathcal{H}^1\\
&\leq 4 E_\mathrm{int}[\mathcal{V}, \chi| \chi^{z,T}], 
	\end{align*}
	where we used the fact that $|\xi_{i,j}^{z,T}|\leq 1/2$ for $\{i,j\} \notin \{1,P\}$. It follows that
	\begin{align*}
	\frac{\mu(B_{\varrho}(x_0)) }{\mathrm{Vol}_1 \varrho } \leq 	\frac{8 \pi (\Lambda +1) \tilde C }{\mathrm{Vol}_1 \tilde \varepsilon^2 r_T } E_\mathrm{int}[\mathcal{V}, \chi| \chi^{z,T}]
	\end{align*}
	and that there exists $\delta \ll_{\tilde \varepsilon, \Lambda } 1 $ such that the assumption \eqref{hp:regularitygoodtimes} implies 
		\begin{align*}
		\frac{\mu(B_{\varrho}(x_0)) }{\mathrm{Vol}_1 \varrho } \leq 	1+ \tilde \varepsilon .
	\end{align*}
Similarly, one can prove that there exists $\delta \ll_{\tilde \varepsilon, \Lambda } 1 $ such that the assumption \eqref{hp:regularitygoodtimes} implies 
\begin{align*}
	E_{\operatorname{tilt}}[x_0, \varrho, \mathbb{R}\times \{0\}] \leq  C_\mathrm{tilt} \varrho^{-1} \mu (B_{\varrho}(x_0)) \leq \tilde \varepsilon^2
\end{align*}
for some constant $C_\mathrm{tilt}>0$.

	\textit{Case (b):} The estimate for the curvature term works as in case (a) as the argument does not rely on the assumption on $x_0 \in \supp \mu$. 
	
	By uniform smoothness of the ball, one may choose $\tilde \varepsilon \ll \varepsilon$, $\delta_{\mathrm{asymp}}\ll 1$, and $\tilde C \gg 1$ such that for all $ x_0 \in \{|\xi^{z,T}|>1/2\}$ and 
	for all $x \in B_{\varrho }(x_0)$ it holds 
	\begin{align} 	\label{eq:proofb0}
		|\xi^{z,T}(x)- \n_{\bar{I}^{z,T}}(P_{\bar I^{z,T}} (x_0))| \leq 
		\frac{\tilde \varepsilon/16}{1+ \tilde \varepsilon/8}.
	\end{align}
	
	We have that 
		\begin{align*}
		\mu(B_{\varrho }(x_0)) &\leq  2 E_\mathrm{int}[\mathcal{V}, \chi| \chi^{z,T}] + \int_{I_{1,P}\cap B_{\varrho } (x_0)} 1 \, d \mathcal{H}^1 , 
	\end{align*}
  where the second term has to be estimated. Hence, we further decompose
  \begin{align*}
  	B_{\varrho }(x_0) \cap I_{1,P} & = \Big(B_{\varrho }(x_0) \cap \Big\{ x \in I_{1,P} : \n_{P,1}(x) \cdot \xi^{z,T} (x) \geq \frac{\tilde \varepsilon/16}{1+ \tilde \varepsilon/8} \Big\}\Big) \\
  	&\quad \cup \Big(B_{\varrho }(x_0) \cap \Big\{ x \in I_{1,P} : \n_{P,1}(x) \cdot \xi^{z,T} (x) < \frac{\tilde \varepsilon/16}{1+ \tilde \varepsilon/8} \Big\}\Big)\\
  	& \subseteq M_{x_0}^{(1)} \cup M_{x_0}^{(2)}, 
  \end{align*}
where
\begin{align*}
	M_{x_0}^{(1)} & :=B_{\varrho }(x_0) \cap \Big\{ x \in \supp |\nabla \chi_1| : \frac{\nabla \chi_1}{|\nabla \chi_1|}(x) \cdot \xi^{z,T}(x) \geq \frac{\tilde \varepsilon/16}{1+ \tilde \varepsilon/8} \Big\},
	\\
	M_{x_0}^{(2)} & := B_{\varrho }(x_0) \cap \Big\{ x \in I_{1,P} : \n_{P,1}(x) \cdot \xi^{z,T}(x) < \frac{\tilde \varepsilon/16}{1+ \tilde \varepsilon/8} \Big\}.
\end{align*}
Using Lemma \ref{lemma:error} and the notation there adopted, we estimate
\begin{align*}
	\mathcal{H}^1(M_{x_0}^{(1)} )
	&= 	\mathcal{H}^1(M_{x_0}^{(1)} \cap 	\{x: \mathcal{H}^0(\partial^\ast(\Omega_1)_{P_{G_{x_0}} (x)} ) >1\})
	\\
	&\quad  + \mathcal{H}^1(M_{x_0}^{(1)} \cap 	\{x: \mathcal{H}^0(\partial^\ast(\Omega_1)_{P_{G_{x_0}} (x)} ) =1\})
\\	& \leq  C_{\mathrm{err}} E_\mathrm{int}[\chi|\bar \chi^{z,T}]  + \mathcal{H}^1(M_{x_0}^{(1)} \cap 	\{x: \mathcal{H}^0(\partial^\ast(\Omega_1)_{P_{G_{x_0}} (x)} ) =1\}).
\end{align*}
Note that by \eqref{eq:proofb0} and by the definition of $M_{x_0}^{(1)} $ it holds
$\n_{\partial \Omega_1}(x) \cdot \n_{\bar{I}^{z,T}}(P_{\bar I^{z,T}} (x_0)) \geq \frac{1}{1+ \tilde \varepsilon/8}$ for all $x \in M_{x_0}^{(1)} $. As a consequence, the coarea formula gives
\[
\mathcal{H}^1(M_{x_0}^{(1)} \cap 	\{x: \mathcal{H}^0(\partial^\ast(\Omega_1)_{P_{G_{x_0}} (x)} ) =1\})
\leq (1 {+} \tfrac{\tilde\varepsilon}8 ) \mathcal{H}^1(B_{\varrho } (P_{\bar I^{z,T}} (x_0))) \leq (1 {+} \tfrac{\tilde\varepsilon}8 ) \mathrm{Vol}_1 {\varrho } . 
\]
Therefore, we obtain
\begin{align*}
	\mathcal{H}^1(M_{x_0}^{(1)} ) \leq  C_{\mathrm{err}} E_\mathrm{int}[\chi|\bar \chi^{z,T}] +  (1+ \tfrac{\tilde\varepsilon}8 ) \mathrm{Vol}_1 {\varrho } . 
\end{align*}
Furthermore, one may estimate
\begin{align*}
		\mathcal{H}^1(M_{x_0}^{(2)} ) \leq \frac{16}{\tilde \varepsilon } \int_{I_{1,P}} 1- \n_{P,1} \cdot \xi^{z,T}\, d \mathcal{H}^1 \leq \frac{16}{\tilde \varepsilon } E_\mathrm{int}[\chi|\bar \chi^{z,T}].
\end{align*}
Collecting the estimates above, we have
	\begin{align*}
	\mu(B_{\varrho}(x_0)) &\leq  \Big(2 + C_{\mathrm{err}} + \frac{16}{\tilde \varepsilon } \Big)E_\mathrm{int}[\mathcal{V}, \chi| \bar \chi^{z,T}] +   \Big(1+ \frac{\tilde\varepsilon}8 \Big) \mathrm{Vol}_1 {\varrho }, 
\end{align*}
whence we can conclude that there esists $\delta \ll_{\tilde \varepsilon, \Lambda } 1 $ such that the assumption \eqref{hp:regularitygoodtimes} implies 
	\begin{align*}
	\frac{\mu(B_{\varrho}(x_0)) }{\mathrm{Vol}_1 {\varrho }} \leq  1+ {\tilde\varepsilon}  .
\end{align*}

It remains to prove the estimate for the tilt excess $E_{\operatorname{tilt}}[x_0, {\varrho }, \Tan_{P_{\bar I^{z,T}} (x_0)} \bar I^{z,T} ]$. First, recalling \eqref{eq:multiplicity}, we notice that
\begin{align*}
	&E_{\operatorname{tilt}}[x_0, {\varrho }, \Tan_{P_{\bar I^{z,T}} (x_0)} \bar I^{z,T} ] \\
	&\lesssim  {\varrho }^{-1} \mu(B_{\varrho }(x_0)\cap \{\omega \leq 1/2\}) + 
	{\varrho }^{-1} \sum_{\{i,j\}\neq \{1,P\}} \mathcal{H}^{1} ( I_{i,j})\\
&\quad +  {\varrho }^{-1} \int_{B_{\varrho }(x_0) \cap I_{1,P}} | \n_{P,1}(x) 
- \n_{\bar I^{z,T}} (P_{\bar I^{z,T}} (x_0)) |^2\, d \mathcal{H}^1
	\\
&	\lesssim   {\varrho }^{-1}  E_\mathrm{int}[\mathcal{V}, \chi| \bar \chi^{z,T}]  +  
{\varrho }^{-1} R_{tilt},
\end{align*}
where
\[
R_{tilt} :=  \int_{B_{\varrho }(x_0) \cap I_{1,P}} 1- 
\n_{P,1} (x)\cdot \n_{\bar I^{z,T}} (P_{\bar I^{z,T}} (x_0)) \, d \mathcal{H}^1.
\]
In order to estimate $R_{tilt}$, we decompose
\begin{align*}
B_{\varrho }(x_0) \cap I_{1,P} \subseteq 	N_{x_0}^{(1)} \cup N_{x_0}^{(2)} \cup  N_{x_0}^{(3)},
\end{align*}
where
\begin{align*}
	N_{x_0}^{(1)} &:= B_{\varrho }(x_0) \cap \Big\{ x \in \partial^\ast \Omega_1 : \frac{\nabla \chi_1}{|\nabla \chi_1|} (x) \cdot \xi^{z,T}(x) \geq \frac12 , \; \mathcal{H}^0(\partial^\ast(\Omega_1)_{P_{G_{x_0}} (x)} ) =1\Big\}, 
	\\
	N_{x_0}^{(2)} &:= B_{\varrho }(x_0) \cap \Big\{ x \in \partial^\ast \Omega_1 : \frac{\nabla \chi_1}{|\nabla \chi_1|} (x) \cdot \xi^{z,T}(x) \geq \frac12 , \; \mathcal{H}^0(\partial^\ast(\Omega_1)_{P_{G_{x_0}} (x)} ) > 1\Big\}, 
		\\
	N_{x_0}^{(3)} &:=  B_{\varrho }(x_0) \cap \{ x \in I_{1,P} : \n_{P,1} (x) \cdot \xi^{z,T}(x) < 1/2\} . 
\end{align*}
By previous arguments (in particular the one using Lemma \ref{lemma:error}), one may infer
\begin{align*}
	R_{tilt} &\leq \int_{N_{x_0}^{(1)} } 
	1 - \frac{\nabla \chi_1}{|\nabla \chi_1|} (x) \cdot \n_{\bar I^{z,T}} (P_{\bar I^{z,T}} (x_0)) \, d \mathcal{H}^1
	+ 2 \mathcal{H}^1 (N_{x_0}^{(2)} \cup N_{x_0}^{(3)}) \\
	&\leq R_{tilt}' + C^{\prime \prime} E_\mathrm{int}[\mathcal{V}, \chi| \bar\chi^{z,T}],
\end{align*}
for some constant $C^{\prime \prime}>0$, where
\begin{align*}
R_{tilt}' &:= \int_{N_{x_0}^{(1)} } 1 - \frac{\nabla \chi_1}{|\nabla \chi_1|} (x) \cdot \n_{\bar I^{z,T}} (P_{\bar I^{z,T}} (x_0)) \, d \mathcal{H}^1 \\
&\leq \int_{N_{x_0}^{(1)} } 1 - \frac{\nabla \chi_1}{|\nabla \chi_1|}(x)  \cdot \xi^{z,T} (x) \, d \mathcal{H}^1
+ \frac{\tilde \varepsilon}{16} \mathcal{H}^1(N_{x_0}^{(1)} ) \\
& \leq C^{\prime \prime \prime } E_\mathrm{int}[\chi| \chi^{z,T}]
+ \frac{\tilde \varepsilon}{16} \mathcal{H}^1(N_{x_0}^{(1)} )
\end{align*}
for some constant $C^{\prime \prime \prime }>0$, where the last inequality can be justified arguing as above.
Furthermore, note that for any $x \in N_{x_0}^{(1)} $ we have
\begin{align*}
	\frac{\nabla \chi_1}{|\nabla \chi_1|} (x) \cdot \n_{\bar I^{z,T}} (P_{\bar I^{z,T}} (x_0)) &\geq 
	\frac{\nabla \chi_1}{|\nabla \chi_1|} (x) \cdot \xi^{z,T}(x) - |\xi^{z,T}(x) - \n_{\bar I^{z,T}} (P_{\bar I^{z,T}} (x_0)) |\\
&	\geq \frac12 - \frac{\tilde \varepsilon/16}{1+ \tilde \varepsilon/8} \geq \frac14 
\end{align*}
whence we can deduce $\mathcal{H}^1(N_{x_0}^{(1)} ) \leq 8 \pi \varrho$, due to the coarea formula. 
Collecting the estimates above, we can conclude that there exists $\delta \ll_{\tilde \varepsilon, \Lambda } 1 $ such that the assumption \eqref{hp:regularitygoodtimes} implies 
\begin{align*}
E_{\operatorname{tilt}}[x_0, \varrho, \Tan_{P_{\bar I^{z,T}} x_0} \bar I^{z,T} ]	\leq  {\tilde\varepsilon}  .
\end{align*}
This shows Lemma~\ref{lemma:hpAllard}.
\end{proof}

\begin{proof}[Proof of Lemma~\ref{lemma:noothers}
(No other phases, hidden boundaries, and higher-multiplicity interfaces).]
	In order to prove \eqref{eq:noothers}, we argue by contradiction. Assume there exists $i \in \{2,...,P-1\}$ such that $\mathcal{H}^1(\supp |\nabla \chi_i|) >0$. Fix $x_0 \in \supp |\nabla \chi_i|$. By Lemma~\ref{lemma:hpAllard} and Theorem~\ref{theo:Allard}, $\supp \mu$ is a graph within $B_{\gamma \varrho} (x_0)$ for $\varrho 
	= \frac{1}{\tilde{C}}\frac{\tilde \varepsilon^2}{2 \pi (\Lambda +1)} r_T$. In addition, by Theorem~\ref{theo:Jordan}, there exists a Jordan-Lipschitz curve $J \subseteq \supp |\nabla \chi_i|$ such that $x_0 \in J$. By choosing $\delta \ll_{\tilde \varepsilon, \Lambda } 1$, one may ensure that $\mathrm{int} (J) \subseteq B_{\gamma \varrho/2}(x_0)$ (see Figure \ref{Fig:noothers} for an example): 
	\begin{figure}
		\includegraphics{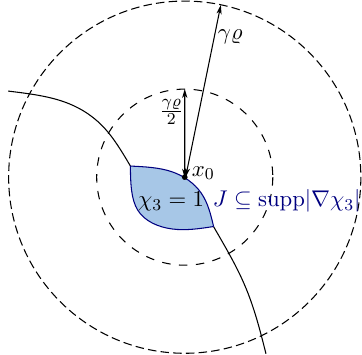}
		\caption{If $\mathcal{H}^1(\supp |\nabla \chi_3|) >0$, then $\mathrm{int} (J) \subseteq B_{\gamma \varrho/2}(x_0)$ for the Jordan-Lipschitz curve $J \subseteq \supp |\nabla \chi_3|$, contradicting the graph property of $\supp \mu $.}
		\label{Fig:noothers}
	\end{figure}
	indeed, we recall that 
	\[
	\mathcal{H}^1(\supp |\nabla\chi_i|) = \sum_{j \neq i} \int_{{I}_{i,j}} 1 \, d \mathcal{H}^1 \leq 
	2 \sum_{j \neq i} \int_{{I}_{i,j}} 1- \n_{i,j} \cdot \xi^{z,T}_{i,j} \, d \mathcal{H}^1 
	\leq 2 E_\mathrm{int}[\chi| \bar \chi^{z,T}], 
	\]
	where we used that $|\xi_{i,j}^{z,T}| \leq 1/2$ for $\{i,j\} \neq \{1,P\}$. However, $\mathrm{int} (J) \subseteq B_{\gamma \varrho/2}(x_0)$ is a contradiction to the graph property of $\supp \mu $ within $B_{\gamma \varrho} (x_0)$.
	
	It remains to prove \eqref{eq:unitdensity}. By Lemma~\ref{lemma:hpAllard} and \cite[Remark 23.2(2)]{SimonLectures}, every $x_0 \in \supp \mu$ is a point of unit density. Hence, $x_0 \in \frac12 \sum_{i=1}^{P}\supp |\nabla \chi_i| = \supp |\nabla \chi_1|$ due to \eqref{eq:noothers}. It follows that $\mu = \mathcal{H}^1 \llcorner \supp |\nabla \chi_1|$, whence we deduce \eqref{eq:unitdensity} by the rectifiability of $\mathcal{V}$.
\end{proof}

\begin{proof}[Proof of Lemma~\ref{lemma:noaway} (No interface far away from $\bar I^{z,T}$).]
	We assume by contradiction that there exists $x_0 \in \supp |\nabla \chi_1|$ such that $x_0 \in \{|\xi^{z,T} | \leq 1/2\}$. 
	From Theorem~\ref{theo:Allard}, Lemma~\ref{lemma:hpAllard}, and Lemma~\ref{lemma:noothers}, it follows that $\supp \mu = \supp |\nabla \chi_1|$ is locally a graph around $x_0$
	on a scale $\gamma \varrho$ over, say, $G_{x_0} = x_0 + (\mathbb{R} \times \{0\})$,
	where $\varrho 
	= \frac{1}{\tilde{C}}\frac{\tilde \varepsilon^2}{2 \pi (\Lambda +1)} r_T$. On the other hand, by Theorem~\ref{theo:Jordan}, there exists a Jordan-Lipschitz curve $J \subseteq \supp |\nabla \chi_1|$ such that $x_0 \in J$. 
	Without loss of generality, we may assume that $\tilde \varepsilon \ll 1$ such that $x_0 \in \{|\xi^{z,T}| < 1/2\}$ implies $B_{\gamma \varrho}(x_0) \subseteq \{|\xi^{z,T}| < 3/4\}$. In particular, we have \[\mathcal{H}^1(J \cap B_{\gamma \varrho}(x_0) ) \leq 4 \int_{J \cap B_{\gamma \varrho}(x_0)} 1- \frac{\nabla \chi_1}{|\nabla \chi_1|} \cdot \xi^{z,T}\, d \mathcal{H}^1 \lesssim E_\mathrm{int}[\chi| \bar \chi^{z,T}] \lesssim \delta r_T.
	\]
	As a consequence, for suitably small $\delta \ll_{\tilde \varepsilon, \Lambda} 1$, it follows from the continuity of~$J$ that $\mathrm{int}(J) \subseteq B_{\gamma\tilde  \rho/2}(x_0)$, which is a contradiction to the graph property of $\supp |\nabla \chi_1|$ within $B_{\gamma \varrho}(x_0)$.	
	The last inclusion in \eqref{eq:xiinclusion} follows from Construction~\ref{DefinitionCalibration}.
\end{proof}

\begin{proof}[Proof of Lemma~\ref{lemma:candidategraph} (Construction of a graph candidate).]
We proceed in two steps.

\textit{Step 1:} We claim that there exist $\delta_{\mathrm{asymp}} \ll 1$ and $\delta \ll 1$ such that 
\begin{align} \label{eq:step1}
	\mathcal{H}^1(\supp |\nabla \chi_1|) \geq \tfrac12  \mathcal{H}^1(\bar{I}^{z,T}).
\end{align}
In order to prove \eqref{eq:step1}, we argue by contradiction, namely we assume that $$	\mathcal{H}^1(\supp |\nabla \chi_1|) < \tfrac12  \mathcal{H}^1(\bar{I}^{z,T}).$$
By the isoperimetric inequality, we have 
\begin{align*}
	\int_{\mathbb{R}^2} \chi_1 \, dx \leq \tfrac{1}{4 \pi} \Big(\tfrac12  \mathcal{H}^1(\bar{I}^{z,T})\Big)^2 \leq  \big( \tfrac{9}{16} \big)^2 \pi r_T^2,
\end{align*}
where we used the fact that $ \mathcal{H}^1(\bar{I}^{z,T}) \leq \frac98 2 \pi r_T$ for a suitably small $\delta_{\mathrm{asymp}} \ll 1$.
Furthermore, for a suitably small $\delta_{\mathrm{asymp}} \ll 1$, we have 
\begin{align*}
		\int_{\mathbb{R}^2} \bar\chi^{z,T} \, dx \geq \big( \tfrac{3}{4} \big)^2 \pi r_T^2.
\end{align*}
By the triangle inequality, it follows that
\begin{align*}
	\| \chi_1 - \bar\chi^{z,T} \|_{L^1} \geq 	\int_{\mathbb{R}^2} \bar\chi^{z,T} \, dx - \int_{\mathbb{R}^2} \chi_1 \, dx  \geq \Big(\big( \tfrac{3}{4} \big)^2 - \big( \tfrac{9}{16} \big)^2 \Big)\pi r_T^2. 
\end{align*}
On the other side, we claim that
\begin{align} \label{eq:contradict}
		\frac{1}{r_T^3}\| \chi_1 - \bar\chi^{z,T} \|_{L^1}^2\lesssim E_{\mathrm{bulk}}[\chi | \bar \chi^{z,T}].
\end{align}
Indeed, by change of variables and Lemma~\ref{lemma:noaway}, we have
\begin{align*}
	&\frac{1}{r_T^3}\| \chi_1 - \bar\chi^{z,T} \|_{L^1}^2 \\
	&= 	\frac{1}{r_T^3} \Big( \int_{\bar{I}^{z,T}} \int_{-\frac{r_T}{4}}^{\frac{r_T}{4}}
	\frac{1}{1 - H_{\bar{I}^{z,T}}(x) s } |\chi_1 - \bar\chi^{z,T} |(x + s \n_{\bar{I}^{z,T}}(x)) \, d s dx	\Big)^2 \\
	&\lesssim 
	\frac{1}{r_T}
	\fint_{\bar{I}^{z,T}}\Big( \int_{-\frac{r_T}{4}}^{\frac{r_T}{4}}
	\frac{1}{1 - H_{\bar{I}^{z,T}}(x) s } |\chi_1 - \bar\chi^{z,T} |(x + s \n_{\bar{I}^{z,T}}(x)) \, d s \Big)^2 dx	\\
	&\lesssim 
	\int_{\bar{I}^{z,T}} \int_{-\frac{r_T}{4}}^{\frac{r_T}{4}}
	\frac{1}{1 - H_{\bar{I}^{z,T}}(x) s } (|\chi_1 - \bar\chi^{z,T} ||\vartheta|)(x + s \n_{\bar{I}^{z,T}}(x)) \, d s dx
	\lesssim E_{\mathrm{bulk}}[\chi | \bar \chi^{z,T}].
\end{align*}
where in the last line we used Fubini's theorem by bisecting $[-\frac{r_T}{4}, \frac{r_T}{4}]^2$ into two triangles
(cf.\ the argument in~\cite[Proof of Theorem~1]{FischerLauxSimon}). 
Hence, by choosing $\delta \ll 1$ in assumption \eqref{hp:regularitygoodtimes}, we obtain 
\begin{align*}
	\| \chi_1 - \bar\chi^{z,T} \|_{L^1} < \Big(\big( \tfrac{3}{4} \big)^2 - \big( \tfrac{9}{16} \big)^2 \Big)\pi r_T^2,
\end{align*}
whence the contradiction follows. 

\textit{Step 2:} From the previous step, there exists $x_0 \in \supp |\nabla \chi_1|$ and, by Theorem~\ref{theo:Jordan}, there exists a Jordan-Lipschitz curve $J \subseteq \supp |\nabla \chi_1|$ such that $x_0 \in J$.
By Theorem~\ref{theo:Allard}, Lemma~\ref{lemma:hpAllard}, and Lemma~\ref{lemma:noothers}, we know that $\supp \mu = \supp |\nabla \chi_1|$ can be represented within $B_{\gamma \varrho}(x_0)$ as a graph over $x_0 + \Tan_{P_{\bar I^{z,T}}(x_0)} \bar{I}^{z,T}$ with height function $u$ such that $\supp |\nabla u| \leq 2 \tilde C \tilde \varepsilon =: \tan \tilde  \alpha$, $\tilde \alpha \in (0, \pi/2)$. 
Hence, for $x_1^\pm \in x_0 + \Tan_{P_{\bar I^{z,T}} (x_0)}\bar I^{z,T}$ such that $\sign((x_1^\pm - x_0 )\cdot \tau_{\bar I^{z,T}}(P_{\bar I^{z,T}} (x_0)) = \pm 1$
and $x_1^\pm  + u(x_1^\pm ) \n_{\bar I^{z,T}}(P_{\bar I^{z,T}} (x_0))  \in \partial B_{\gamma\varrho/2}(x_0)$, it follows  $|(x_1^\pm- x_0)\cdot \tau_{\bar I^{z,T}}(P_{\bar I^{z,T}} (x_0))| \geq  \frac{\gamma \tilde  \rho}{2} \cos \tilde \alpha $. Consequently, from Lemma~\ref{lemma:hpAllard} we deduce
\begin{align} \label{eq:boundP1}
	|P_{\bar I^{z,T}} (x_1^\pm) - P_{\bar I^{z,T}} (x_0)| \geq \frac12  \frac{\gamma \varrho}{2} \cos \tilde \alpha.
\end{align}
Since $\supp |\nabla \chi_1| \cap B_{\gamma \varrho}(x_0) = J \cap B_{\gamma \varrho}(x_0) $, one may continue with $x_1 = x_1^+$ and iterate the above reasoning to conclude that
$J$ wriggles around~$\bar{I}^{z,T}$ in the sense of
\begin{align*}
	\left(\{\bar{\chi}^{z,T}=1\} \setminus \{\dist(\cdot, \bar{I}^{z,T}) \leq {\delta r_T}/4 \}\right)&\subseteq \mathrm{int} (J) ,\\
	\left(\{\bar{\chi}^{z,T}=0\} \setminus \{\dist(\cdot, \bar{I}^{z,T}) \leq \delta r_T/{4}\}\right) &\subseteq \mathrm{ext} (J).
\end{align*}
We notice that the iteration stops after finitely many steps due to \eqref{eq:boundP1}. 
At last, we argue that there exists a height function $h: \bar{I}^{z,T} \rightarrow [-r_T/4, r_T/4]$ such that \eqref{eq:Jgraph} holds. This directly follows from the compactness of ${\bar{I}^{z,T}}$ together with the fact that, for any $x_0 \in J$, there exists an open neighborhood $\mathcal{W}_{x_0} \ni x_0$ such that
\begin{align}\label{eq:upsilon}
	 J \cap \mathcal{W}_{x_0} &\rightarrow  \bar{I}^{z,T} \notag \\
	x &\mapsto P_{\bar{I}^{z,T}} (x) = x - s(x) \n_{\bar{I}^{z,T}}(P_{\bar{I}^{z,T}} (x))
\end{align}
	is a manifold diffeomorphism onto its image, due to an application of the inverse function theorem.
\end{proof}

\begin{proof}[Proof of Lemma~\ref{lemma:weakgraph} (Interface is a graph over $\bar I^{z,T}$).]
	Recall from Lemma~\ref{lemma:noaway} that $\supp |\nabla \chi_1| \subseteq \{|\xi^{z,T}| > 1/2\}
	\subseteq \{\dist(\cdot, \bar I^{z,T}) \leq r_T/4\}$.
	Assume by contradiction that there exists a nontrivial Jordan-Lipschitz curve $J'$ from the decomposition of Theorem \ref{theo:Jordan} such that $J' \subseteq \supp |\nabla \chi_1|$ and $J' \neq J$, i.e., either $\mathrm{int}(J') \cap \mathrm{int}(J) = \emptyset $, or $\mathrm{int}(J') \subset \mathrm{int}(J) $ , or $\mathrm{int}(J) \subset\mathrm{int}(J') $. 
	 Fix $x_0 \in J'$. Since $J$ is a graph over $\bar{I}^{z,T}$
	by Lemma~\ref{lemma:candidategraph}, 
	using the notation from Lemma~\ref{lemma:error}, it follows that
	 $$B_{\varrho}(x_0) \cap J' \subseteq 	\{x: \mathcal{H}^0(\partial^\ast(\Omega_1)_{P_{G_{x_0}} (x)} ) >1\}. $$
	 Recall from Lemma~\ref{lemma:hpAllard} that $\varrho$ is chosen such that the hypothesis of Lemma~\ref{lemma:error} applies, hence 
	 $\mathcal{H}^1(B_{\varrho}(x_0) \cap J' ) \lesssim E_\mathrm{int}[\chi| \bar \chi^{z,T}]$. 
	 By the continuity of $J'$, for suitably small $\delta \ll_{\tilde \varepsilon, \Lambda} 1$, one may infer that $\mathrm{int}(J') \subseteq B_{\gamma \varrho/2}(x_0)$ (see Figure \ref{Fig:interfacegraph}).
	 \begin{figure}
	 	\includegraphics{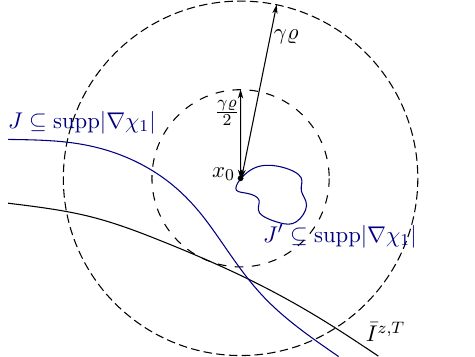}
	 	\caption{Interface is a graph over $\bar I^{z,T}$}
	 	\label{Fig:interfacegraph}
	 \end{figure}
	 On the other side, by Theorem~\ref{theo:Allard}, Lemma~\ref{lemma:hpAllard}, and Lemma~\ref{lemma:noaway}, we have that $\supp \mu = \supp |\nabla \chi_1|$ is a graph within $B_{\gamma \varrho}(x_0)$, which provides the contradiction.
\end{proof}

\begin{proof}[Proof of Lemma~\ref{lemma:heightestimates} (Height function estimates).]
	\textit{Step 1: Regularity 	\eqref{eq:graphReg}.}
	Fix $x_0 \in \supp |\nabla \chi_1|$. By Theorem~\ref{theo:Allard}, Lemma~\ref{lemma:hpAllard}, Lemma~\ref{lemma:noothers}, and Lemma~\ref{lemma:weakgraph} we know that $\supp \mu = \supp |\nabla \chi_1|=J$ can be represented within $B_{\gamma \varrho}(x_0)$ as a graph over $G_{x_0} := x_0 + \Tan_{P_{\bar I^{z,T}}(x_0)} \bar{I}^{z,T}$ with height function $u$. 
	Furthermore, we know by Lemma~\ref{lemma:candidategraph} that there exists a height function $h\colon \bar{I}^{z,T} \rightarrow [-r_T/4, r_T/4]$ such that \eqref{eq:Jgraph} holds. 

Define $\mathcal{U}_{x_0} := P_{x_0}(G_{x_0} \cap B_{\gamma \varrho}(x_0))$, and we first claim that
	\begin{align} \label{eq:stepu}
	u \in H^2(\mathcal{U}_{x_0}  ).
	\end{align}
	In order to prove \eqref{eq:stepu}, we fix $g \in C^\infty_{cpt}(\mathcal{U}_{x_0})$ and then test \eqref{eq:Hvarifold} with $B(x)= g(P_{\bar{I}^{z,T}} (x) ) \n_{\bar{I}^{z,T}}(P_{\bar{I}^{z,T}} (x_0)) $. Using $H_\mu= (H_\mu\cdot \n_{\supp \mu })\n_{\supp \mu }$, the coarea formula, and the coordinates induced by the height function $u$, we obtain
		\begin{align*}
		&- \int_{\mathcal{U}_{x_0}  } (H_\mu\cdot \n_{\supp \mu } )(x+ u(x) \n_{\bar{I}^{z,T}}(P_{\bar{I}^{z,T}} (x_0)))  g( x ) \, d \mathcal{H}^1 \\
		&= \int_{\mathcal{U}_{x_0} }\frac{u'}{\sqrt{1+(u')^2}} g' \, d \mathcal{H}^1,
	\end{align*}
whence we deduce $\frac{u'}{\sqrt{1+(u')^2}} \in H^1(\mathcal{U}_{x_0})$ due to
the assumption \eqref{boundgoodtimes} of controlled dissipation. 
The regularity \eqref{eq:stepu} then follows from Lemma~\ref{lemma:hpAllard} and Theorem~\ref{theo:Allard}, in particular from the estimate \eqref{eq:Allard}
(recall in this context also the regularity $u \in C^{1,\frac{1}{2}}$).

In a second step, we argue that one may capitalize on~\eqref{eq:stepu}
    to show that there is an open 
		neighborhood $\bar{\mathcal{U}}_{x_0}  $ of $P_{\bar{I}^{z,T}} (x_0)$ in ${\bar{I}^{z,T}}$ such that 
    \begin{align}\label{eq:steph}
    h \in H^2(\bar{\mathcal{U}}_{x_0}   ),
    	\end{align}
whence one may deduce \eqref{eq:graphReg} by compactness of ${\bar{I}^{z,T}}$.
    Indeed, define $\bar{\mathcal{U}}_{x_0} :=  \iota(\mathcal{U}_{x_0} )$, where $\iota: \mathcal{U}_{x_0} \rightarrow \bar{I}^{z,T}$ is given by $x \mapsto P_{\bar{I}^{z,T}}(x+ u(x)\n_{\bar{I}^{z,T}}(P_{\bar{I}^{z,T}} (x_0)))$.
		Since the map~$\iota$ is a chart for~$\bar{I}^{z,T}$, we conclude~\eqref{eq:steph}
		from the formula $h (\iota(x)) = s(x +  u(x) \n_{\bar{I}^{z,T}}(P_{\bar{I}^{z,T}} (x)))$ 
		for any $x \in \mathcal{U}_{x_0} $, the regularity of the signed distance function,
		and~\eqref{eq:stepu}.

\textit{Step 2: Estimate \eqref{eq:smallC0norm2} for $\sup |h|$.}
The idea here is to exploit that if~\eqref{eq:smallC0norm2}
would not be satisfied, then one accumulates too much $L^1$-error
between the two phases $\chi_1$ and $\bar\chi_1$
in contradiction with the smallness of the overall error~\eqref{hp:regularitygoodtimes}.

Hence, we assume by contradiction that there exists $x_0 \in \supp |\nabla \chi_1|$ such that 
$$
	\|h\|_{L^\infty(\bar{I}^{z,T})} > \frac{r_T}{16 \max\{C,C_\zeta\}}. 
$$
	Recall that by Theorem~\ref{theo:Allard}, Lemma~\ref{lemma:hpAllard}, and Lemma~\ref{lemma:weakgraph}, we know that $\supp \mu \cap B_{\gamma \varrho}(x_0)$ can be represented as a graph over $(x_0 + \Tan_{P_{\bar I^{z,T}}(x_0)} \bar{I}^{z,T} ) \cap  B_{\gamma \varrho}(x_0)   $ with height function $u$,
	where $\varrho 
	= \frac{1}{\tilde{C}}\frac{\tilde \varepsilon^2}{2 \pi (\Lambda +1)} r_T$. 
	From Theorem~\ref{theo:Allard} and Lemma~\ref{lemma:hpAllard}, more precisely from \eqref{eq:hpAllard}, \eqref{eq:Allard} and \eqref{eq:boundeps}, it follows that
		\begin{align} \label{eq:estimatesu}
	\sup |u| \leq 2 \varrho \tilde C \tilde \varepsilon \leq \frac14 \frac{r_T}{16 \max\{C,C_\zeta\}}, \quad 
	\sup |u'|  \leq 2  \tilde C \tilde \varepsilon =: \tan \tilde \alpha, \; \tilde \alpha \in (0, \pi/2).
\end{align}
	In particular, we have (see Figure \ref{Fig:estimatesh})
	\begin{align*} 
		\partial B_{\gamma \varrho}(x_0) \cap 
		\{
		y + u(y) &\n_{\bar{I}^{z,T}} (P_{\bar{I}^{z,T}} (x_0)) : y \in (x_0 + \Tan_{P_{\bar I^{z,T}}(x_0)} \bar{I}^{z,T} ) \cap  B_{\gamma \varrho}(x_0)
		\}\\
		\subseteq \bigg\{
		x \in \mathbb{R}^2 :\, 
		&|(x- x_0)\cdot \n_{\bar I^{z,T}}(P_{\bar I^{z,T}} (x_0))| \leq  \frac14 \frac{r_T}{16 \max\{C,C_\zeta\}}, \\
		&  |(x- x_0)\cdot \tau_{\bar I^{z,T}}(P_{\bar I^{z,T}} (x_0))| \geq  {\gamma \varrho} \cos  \tilde \alpha
		\bigg\}.
	\end{align*}
	Furthermore, by \eqref{eq:boundsp} from Lemma~\ref{lemma:hpAllard}, we deduce that 
    for all $x\in \bar I^{z,T}$ such that $|(x- P_{\bar I^{z,T}} (x_0))\cdot \tau_{\bar I^{z,T}}(P_{\bar I^{z,T}} (x_0))| \leq \gamma \varrho $, it holds
	\begin{align*} 
		{|(x- P_{\bar I^{z,T}} (x_0)) \cdot \n_{\bar I^{z,T}}(P_{\bar I^{z,T}} (x_0) )|} \leq \frac{1}4 \frac{r_T}{16 \max\{C, C_\zeta\}}.
	\end{align*}
\begin{figure}
	\includegraphics{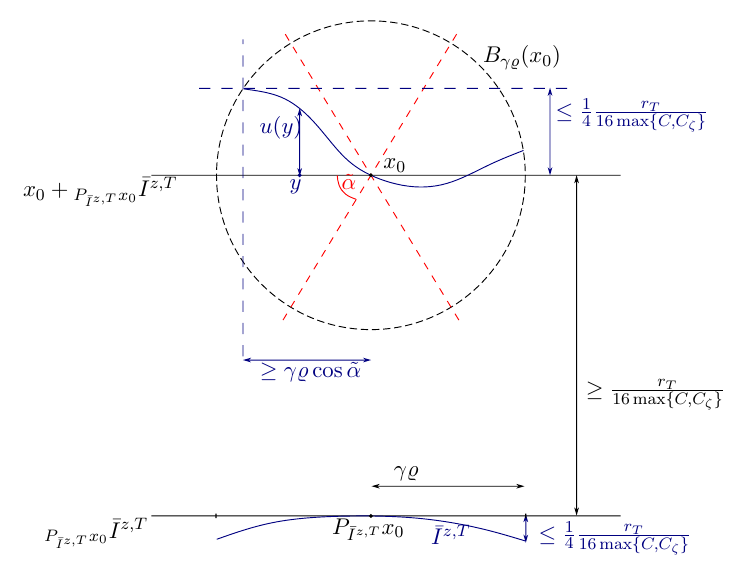}
	\caption{Height function estimates}
	\label{Fig:estimatesh}
\end{figure}
    Hence, we have 
    $$
    \| \chi_1 - \bar \chi^{z,T}\|_{L^1} \geq {\gamma \varrho} \frac1{2} \cos \tilde \alpha \frac{r_T}{16 \max\{C, C_\zeta\}}.
    $$
    The contradiction follows from \eqref{eq:contradict} and by choosing $\delta \ll_{\tilde \varepsilon, \Lambda, C, C_\zeta} 1$ suitably small in assumption \eqref{hp:regularitygoodtimes}. 

\textit{Step 3: Estimate \eqref{eq:smallC1norm2} for $\sup |h'|$.}
Let $\bar \gamma $ be an arc-lenght parametrization of~$\bar I^{z,T}$ so that $\supp |\nabla \chi_1|$ admits the parametrization $\gamma_h := (\Id + h \n_{\bar{I}^{z,T}}) \circ \bar \gamma $. Fix $x_0 \in \supp |\nabla \chi_1| $ and denote by $u$ the associated height function given by Theorem~\ref{theo:Allard} and Lemma~\ref{lemma:hpAllard} on scale $\gamma \varrho$. In other words, locally around $x_0$, we have
a second parametrization $\gamma_u := (\Id + u \n_{\bar{I}^{z,T}}(P_{\bar I^{z,T}} (x_0) )) \circ \bar \gamma_{x_0}$, where $\bar \gamma_{x_0}$ is an arc-length parametrization of $(x_0 + \Tan_{P_{\bar I^{z,T}}(x_0)} \bar{I}^{z,T} ) \cap  B_{\gamma \varrho}(x_0)   $.
One may compute
$$
\frac{1}{\sqrt{1+ \Big(\frac{h'(P_{\bar I^{z,T}}(x_0))}{1- (H_{\bar{I}^{z,T}}h)(P_{\bar I^{z,T}}(x_0))}\Big)^2}} = \frac{1}{\sqrt{1+(u'(x_0))^2}}
$$
as both terms equals $\frac{\nabla \chi_1 }{|\nabla \chi_1|}(x_0)\cdot \n_{\bar{I}^{z,T}}(P_{\bar I^{z,T}} (x_0) )$.
Note that this gives us a relation expressing $|h'(P_{\bar I^{z,T}}(x_0))|$
in terms of $|u'(x_0)|$.
In particular, by choosing $\delta_{\mathrm{asymp}} \ll 1$ suitably small and using $|h(P_{\bar I^{z,T}}(x_0))| \leq \frac14 r_T$, we obtain $|h'(P_{\bar I^{z,T}}(x_0))| \leq \frac32 |u'(x_0)|$ for all $x_0 \in  \supp |\nabla \chi_1|$. 
From \eqref{eq:estimatesu} by means of \eqref{eq:boundeps} one may infer that $\sup |h' | \leq \frac32 	\sup |u' | \leq 1/C'$.
\end{proof}

\begin{proof}[Proof of Lemma~\ref{lemma:error} (Error control).]
	Define $\n_{x_0}:=\n_{\bar I^{z,T}}(P_{\bar I^{z,T}} (x_0)) $ and the set 
	\[
	A_{x_0} := B_{\varrho}(x_0) \cap
	\partial^\ast \Omega_1
	\cap 
	\{x: \mathcal{H}^0(\partial^\ast(\Omega_1)_{P_{G_{x_0}} (x)} ) >1\},
	\]
	which we then decompose as
	\[
	A_{x_0} = \left( A_{x_0} \cap \{ \n_{\partial^\ast \Omega_1} \cdot \n_{x_0} \geq 1/2\}\right) \cup \left( A_{x_0} \cap \{ \n_{\partial^\ast \Omega_1} \cdot \n_{x_0} < 1/2 \}\right) =: A^{(1)}_{x_0}  \cup A^{(2)}_{x_0} .
	\]
	Since $\n_{\partial^\ast \Omega_1} (x)\cdot \xi^{z,T}(x) \leq |\xi^{z,T}(x) - \n_{x_0}(x) | +  \n_{\partial^\ast \Omega_1} (x)\cdot \n_{x_0} \leq 3/4$ for any $x \in A^{(2)}_{x_0} $, then
	we obtain $1- \n_{\partial^\ast \Omega_1} (x)\cdot \xi^{z,T}(x)  \geq 1/4$ for any $x \in A^{(2)}_{x_0} $.
	Hence, we have 
	\begin{align*}
		\mathcal{H}^1( A^{(2)}_{x_0} ) &\leq 4 \int_{\partial^\ast \Omega_1}  1- \n_{\partial^\ast \Omega_1} \cdot \xi^{z,T} \, {d}	\mathcal{H}^1
		\\
		&\leq 4 \int_{I_{1,P}} 1- \n_{P,1} \cdot \xi^{z,T} \, {d}	\mathcal{H}^1 + 8 \sum_{j \notin \{1,P\} } \int_{I_{1,j}} 1  \, {d}	\mathcal{H}^1
		\\
		& \leq 4 \int_{I_{1,P}} 1- \n_{P,1} \cdot \xi^{z,T} \, {d}	\mathcal{H}^1 + 16 \sum_{i,j \notin \{1,P\} } \int_{I_{i,j}} 1  - \n_{i,j} \cdot \xi_{i,j}^{z,T}\, {d}	\mathcal{H}^1\\
		& \leq 16  E_\mathrm{int}[\chi|\bar \chi^{z,T}],
	\end{align*}
	where we use the fact that $|\xi_{i,j}^{z,T}| \leq 1/2$ for $\{i,j\} \neq \{1,P\}$.
	Moreover, using the coarea formula, we obtain
	\begin{align*}
		\mathcal{H}^1( A^{(1)}_{x_0} ) &\leq 2 \int_{B_{\varrho }(P_{\bar I^{z,T}} x_0 ) \cap G_{x_0}} \Big( \sum_{\substack{y: \,  x + y \n_{x_0} \in A^{(1)}_{x_0}  }} 1 \Big) \, {d}	\mathcal{H}^1\\
		&\leq 4 \int_{B_{\varrho } (P_{\bar I^{z,T}} x_0 ) \cap G_{x_0}} \Big( \sum_{\substack{y: \,  x + y \n_{x_0} \in A^{(2)}_{x_0}  }} 1 \Big) \, {d}	\mathcal{H}^1\\
		&\leq 4	\mathcal{H}^1( A^{(2)}_{x_0} ) ,
	\end{align*}
    where the second inequality follows from the fact that, for any $x^+_0 \in A^{(1)}_{x_0}$, one may associate $x_0^-$ such that $\n_{\partial^\ast \Omega _1 }(x_0^-) \cdot \n_{x_0} \leq 0$ (see Figure \ref{Fig:Erelcontrol}). 
    \begin{figure}
    	\includegraphics{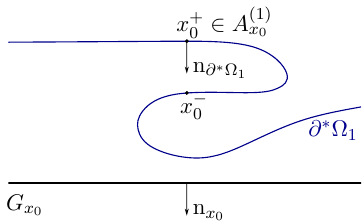}
    	\caption{Error control by $E_\mathrm{int}[\chi|\bar \chi^{z,T}]$}
    	\label{Fig:Erelcontrol}
    \end{figure}
    The claim follows by combining the two estimates above. 
\end{proof}

\subsection{Proof of Lemma~\ref{lemma:relentfreaze}: Error functionals in perturbative regime}\label{proof:errorfuncperturbative}
The estimates~\eqref{eq:bulkPerturbativeRegime} follow directly from the change of variables~\eqref{eq:aux31}
to tubular neighborhood coordinates,
the definition \eqref{def:bulkError} of $E_{\mathrm{bulk}}$, the identity~\eqref{eq:aux17}
for the weight~$\vartheta$ in the vicinity of the interface as well as
the height function estimate~\eqref{eq:smallC0norm},
whereas the estimates~\eqref{eq:relEntropyPerturbativeRegime} are
immediate consequences of the coarea formula~\eqref{eq:aux30}, the definition~\eqref{def:interfaceerror}
of $E_{\mathrm{int}}$, the representation \eqref{eq:aux15} for the vector field~$\xi$
near the interface, the linearization \eqref{eq:aux25} for the normal of the weak solution as well as
the height function estimates~\eqref{eq:smallC0norm}--\eqref{eq:smallC1norm}.
\qed


\section{Auxiliary computations in perturbative regime}
\label{sec:appendix}
Let $(\xi,\vartheta,B)$
be the maps from Construction~\ref{DefinitionCalibration},
and let $(\tchi,z,T)$ be the space-time shifts from Lemma~\ref{lemma:existzT}.
Fix $t \in (0,\tchi)$ and assume the existence of an height function $h(\cdot, t)$ satisfying the properties as in the conclusions of Proposition~\ref{theo:graph}. For ease of notation, we will drop in the following any
dependence on the time~$t$. Furthermore, we will abbreviate
in the tubular neighborhood $\{\dist(\cdot,\bar{I}^{z,T}) < r_T/2\}$
\begin{align*}
	s_{\bar{I}^{z,T}} &:= \sdist_{\bar{I}^{z,T}},
	\\
	\bar{\n}_{\bar{I}^{z,T}} &:= \n_{\bar{I}^{z,T}}\circ P_{\bar{I}^{z,T}},
	\quad \bar{\tau}_{\bar{I}^{z,T}} := \tau_{\bar{I}^{z,T}}\circ P_{\bar{I}^{z,T}},
	\\
	\bar{H}_{\bar{I}^{z,T}} &:= H_{\bar{I}^{z,T}}\circ P_{\bar{I}^{z,T}}.
\end{align*}
Finally, define $I:=I_{1,P}$, which is by assumption subject to the graph 
representation~\eqref{eq:graphReg}--\eqref{eq:smallC1norm}, and denote
$V:=V_{1}$, $\n := \n_{P,1}$ as well as, by slight abuse of notation, $\chi:=\chi_1$ and $\bar\chi := \bar\chi_1$.

\begin{lemma}
\label{lem:perturbativeComputations}
For given $\delta_{\mathrm{err}} \in (0,1)$, one may choose the constants 
$C,C' \gg_{\delta_{\mathrm{err}}} 1$ from~\emph{\eqref{eq:smallC0norm}--\eqref{eq:smallC1norm}} such that
the individual terms of the stability estimates~\eqref{eq:prelimStabilityRelEntropy} 
and~\eqref{eq:prelimStabilityBulk} are estimated as follows:
\begin{align}
	\nonumber
	&-\int_{I} \big(\partial_t\xi^{z,T} {+} (B^{z,T}\cdot\nabla)\xi^{z,T} {+} (\nabla B^{z,T})^\mathsf{T}\xi^{z,T}\big)
	\cdot (\n{-}\xi^{z,T}) \,d\mathcal{H}^1
	\\&~~~~~~~~~~~~~~~
	\label{eq:aux1}
	\leq - \int_{\bar{I}^{z,T}} H_{\bar{I}^{z,T}}^2 h (\n_{\bar{I}^{z,T}}\cdot\dot{z}) \,d\mathcal{H}^1
	\\&~~~~~~~~~~~~~~~~~~ \nonumber
	- \int_{\bar{I}^{z,T}} H_{\bar{I}^{z,T}}' (\tau_{\bar{I}^{z,T}}\cdot\dot{z}) h \,d\mathcal{H}^1 
	- \int_{\bar{I}^{z,T}} H_{\bar{I}^{z,T}}' \dot{\mathfrak{T}} h' \,d\mathcal{H}^1 
	\\&~~~~~~~~~~~~~~~~~~ \nonumber
	+ \int_{\bar{I}^{z,T}} \delta_{\mathrm{err}} \Big(\frac{1}{r_T}\big|(\tau_{\bar{I}^{z,T}}\cdot\dot{z}) h'\big|
	{+} \big|H_{\bar{I}^{z,T}}' \dot{\mathfrak{T}} h'\big|\Big) \,d\mathcal{H}^1,
	\\
	\label{eq:aux2}
	&-\int_{I} \big(\partial_t\xi^{z,T} {+} (B^{z,T}\cdot\nabla)\xi^{z,T}\big) \cdot \xi^{z,T} \,d\mathcal{H}^1 = 0,
	\\
	\nonumber
	&~~~\int_{I}  \frac{1}{2} \big|\nabla\cdot\xi^{z,T} {+} B^{z,T}\cdot\xi^{z,T}\big|^2 \,d\mathcal{H}^1
	\\&~~~~~~~~~~~~~~~
	\label{eq:aux3}
	\leq \int_{\bar{I}^{z,T}} \frac{1}{2} H_{\bar{I}^{z,T}}^4 h^2 + \delta_{\mathrm{err}}\frac{1}{r_T^4} h^2 \,d\mathcal{H}^1,
	\\
	\label{eq:aux4}
	&-\int_{I} \frac{1}{2} \big|B^{z,T}\cdot\xi^{z,T}\big|
	\big(1 - |\xi^{z,T}|^2\big) \,d\mathcal{H}^1 = 0,
	\\
	\nonumber
	&-\int_{I} (1 - \n\cdot\xi^{z,T}) \nabla\cdot\xi^{z,T}
	(B^{z,T}\cdot\xi^{z,T}) \,d\mathcal{H}^1
	\\&~~~~~~~~~~~~~~~
	\label{eq:aux5}
	\leq \int_{\bar{I}^{z,T}} \frac{1}{2} H_{\bar{I}^{z,T}}^2 (h')^2 + \delta_{\mathrm{err}}\frac{1}{r_T^2} (h')^2 \,d\mathcal{H}^1,
	\\
	&~~~\int_{I} \big((\mathrm{Id} {-} \xi^{z,T}\otimes \xi^{z,T})B^{z,T}\big) \cdot 
	(V {+} \nabla\cdot\xi^{z,T})\n \,d\mathcal{H}^1
	\label{eq:aux6}
	= 0,
	\\
	\nonumber
	&~~~\int_{I} (1 - \n\cdot\xi^{z,T}) \nabla\cdot B^{z,T} \,d\mathcal{H}^1
	\\&~~~~~~~~~~~~~~~
	\label{eq:aux7}
	\leq - \int_{\bar{I}^{z,T}} \frac{1}{2} H_{\bar{I}^{z,T}}^2 (h')^2 \,d\mathcal{H}^1
	+ \int_{\bar{I}^{z,T}} \delta_{\mathrm{err}} \frac{1}{r_T^2} (h')^2 \,d\mathcal{H}^1,
	\\
	\nonumber
	&-\int_{I} (\n {-} \xi^{z,T}) \otimes (\n {-} \xi^{z,T}) : \nabla B^{z,T} \,d\mathcal{H}^1
	\\&~~~~~~~~~~~~~~~
	\label{eq:aux8}
	\leq \int_{\bar{I}^{z,T}} H_{\bar{I}^{z,T}}^2 (h')^2 \,d\mathcal{H}^1
	+ \int_{\bar{I}^{z,T}} \delta_{\mathrm{err}} \Big(\frac{1}{r_T^2} {+} \big|H_{\bar{I}^{z,T}}'\big| \Big) (h')^2 \,d\mathcal{H}^1,
	\\
	\nonumber
	&-\int_{I} \frac{1}{2} \big|V{+}\nabla\cdot\xi^{z,T}\big|^2 \,d\mathcal{H}^1
	\\&~~~~~~~~~~~~~~~
	\label{eq:aux9}
	\leq - \int_{\bar{I}^{z,T}} \frac{1}{2}(h'')^2 \,d\mathcal{H}^1
	\\&~~~~~~~~~~~~~~~~~~ \nonumber
	+ \int_{\bar{I}^{z,T}} \delta_{\mathrm{err}} \Big((h'')^2 + \frac{1}{r_T^2}(h')^2
	+ \Big((H_{\bar{I}^{z,T}}')^2 {+} \frac{1}{r_T^4}\Big) h^2\Big) \,d\mathcal{H}^1
	\\
	\nonumber
	&-\int_{I} \frac{1}{2} \big|V\n{-}(B^{z,T}\cdot\xi^{z,T})\xi^{z,T}\big|^2 \,d\mathcal{H}^1
	\\&~~~~~~~~~~~~~~~
	\label{eq:aux10}
	\leq - \int_{\bar{I}^{z,T}}  \frac{1}{2}
	\Big((h'')^2 + H_{\bar{I}^{z,T}}^4 h^2 - H_{\bar{I}^{z,T}}^2 (h')^2\Big) \,d\mathcal{H}^1
	\\&~~~~~~~~~~~~~~~~~~ \nonumber
	+ \int_{\bar{I}^{z,T}} 2 H_{\bar{I}^{z,T}} H_{\bar{I}^{z,T}}' h h' \,d\mathcal{H}^1
	\\&~~~~~~~~~~~~~~~~~~ \nonumber
	+ \int_{\bar{I}^{z,T}} \delta_{\mathrm{err}} \Big((h'')^2 + \frac{1}{r_T^2}(h')^2
	+ \Big((H_{\bar{I}^{z,T}}')^2 {+} \frac{1}{r_T^4}\Big) h^2\Big) \,d\mathcal{H}^1,
	\\
	\nonumber
	&~~~\int_{I} \vartheta^{z,T} (B^{z,T}\cdot\xi^{z,T} {-} V) \,d\mathcal{H}^1
	\\&~~~~~~~~~~~~~~~
	\label{eq:aux11}
	\leq \int_{\bar{I}^{z,T}} \frac{H_{\bar{I}^{z,T}}^2}{r_T^2} h^2 \,d\mathcal{H}^1
	- \int_{\bar{I}^{z,T}} \frac{1}{r_T^2} (h')^2 \,d\mathcal{H}^1
	\\&~~~~~~~~~~~~~~~~~~ \nonumber
	+ \int_{\bar{I}^{z,T}} \delta_{\mathrm{err}} \Big((h'')^2 + \frac{1}{r_T^2}(h')^2
	+ \Big((H_{\bar{I}^{z,T}}')^2 {+} \frac{1}{r_T^4}\Big) h^2\Big) \,d\mathcal{H}^1,
	\\
	\nonumber
	&~~~\int_{I} \vartheta^{z,T} B^{z,T}\cdot (\n - \xi^{z,T}) \,d\mathcal{H}^1
	\\&~~~~~~~~~~~~~~~
	\label{eq:aux12}
	\leq \int_{\bar{I}^{z,T}} \delta_{\mathrm{err}} \Big(\frac{1}{r_T^4}h^2 + \frac{1}{r_T^2}(h')^2\Big) \,d\mathcal{H}^1,
	\\
	\nonumber
	&~~~\int_{\mathbb{R}^2} (\chi {-} \bar\chi^{z,T}) \vartheta^{z,T} \nabla\cdot B^{z,T} \,dx
	\\&~~~~~~~~~~~~~~~
	\label{eq:aux13}
	\leq - \int_{\bar{I}^{z,T}} \frac{1}{2}\frac{H_{\bar{I}^{z,T}}^2}{r_T^2} h^2 \,d\mathcal{H}^1
	+ \int_{\bar{I}^{z,T}} \delta_{\mathrm{err}} \frac{1}{r_T^4} h^2 \,d\mathcal{H}^1,
	\\
	\nonumber
	&~~~\int_{\mathbb{R}^2} (\chi {-} \bar\chi^{z,T})\big(\partial_t\vartheta^{z,T} {+} (B^{z,T}\cdot\nabla)\vartheta^{z,T}\big) \,dx
	\\&~~~~~~~~~~~~~~~
	\label{eq:aux14}
	\leq \int_{\bar{I}^{z,T}} \frac{1}{r_T^4} h^2 \,d\mathcal{H}^1
	\\&~~~~~~~~~~~~~~~~~~ \nonumber
	- \int_{\bar{I}^{z,T}} \frac{1}{r_T^2}H_{\bar{I}^{z,T}} h \dot{\mathfrak{T}} \,d\mathcal{H}^1
	- \int_{\bar{I}^{z,T}} \frac{1}{r_T^2} h (\n_{\bar{I}^{z,T}} \cdot \dot{z}) \,d\mathcal{H}^1
	\\&~~~~~~~~~~~~~~~~~~ \nonumber
	+ \int_{\bar{I}^{z,T}} \delta_{\mathrm{err}} \bigg( \frac{1}{r_T^3}\big|\dot{\mathfrak{T}}h\big| {+} 
	\frac{1}{r_T^2}\big|(\n_{\bar{I}^{z,T}} \cdot \dot{z})h\big|
	{+} \frac{1}{r_T^4} h^2 \bigg)  \,d\mathcal{H}^1.
\end{align}
\end{lemma}

\begin{proof}
	We proceed in several steps.
	
	\textit{Step~1: Properties of gradient flow calibration.} 
	Thanks to the height function estimate~\eqref{eq:smallC0norm}, the definitions~\eqref{eq:defxi}--\eqref{eq:defWeight}
	of the gradient flow calibration,
	and the identities~\eqref{eq:shiftedSignedDistance}--\eqref{eq:shiftedProjection} 
	for the shifted geometry,
	it holds on~$I \subset \{\dist(\cdot,\bar{I}^{z,T})<r_T/8\}$ 
	\begin{align}
		\label{eq:aux15}
		\xi^{z,T} &= \bar{\n}_{\bar{I}^{z,T}} = \nabla s_{\bar{I}^{z,T}},
		\\
		\label{eq:aux16}
		B^{z,T} &= \bar{H}_{\bar{I}^{z,T}}\bar{\n}_{\bar{I}^{z,T}},
		\\
		\label{eq:aux17}
		\vartheta^{z,T} &= -\frac{s_{\bar{I}^{z,T}}}{r_T^2}.
	\end{align}
	In particular, because of
	\begin{align*}
		\nabla P_{\bar{I}^{z,T}} = \Id - \bar{\n}_{\bar{I}^{z,T}} \otimes \bar{\n}_{\bar{I}^{z,T}} - s_{\bar{I}^{z,T}}\nabla \bar{\n}_{\bar{I}^{z,T}},
	\end{align*}
	we obtain by direct computation throughout $\{\dist(\cdot,\bar{I}^{z,T})<r_T/8\}$
	\begin{align}
		\label{eq:aux18}
		\nabla \xi^{z, T} &= - \frac{\bar{H}_{\bar{I}^{z,T}}}{1 - \bar{H}_{\bar{I}^{z,T}} s_{\bar{I}^{z,T}}}
		\bar{\tau}_{\bar{I}^{z,T}} \otimes \bar{\tau}_{\bar{I}^{z,T}},
		\\
		\label{eq:aux19}
		\nabla \cdot \xi^{z, T} &= - \frac{\bar{H}_{\bar{I}^{z,T}}}{1 - \bar{H}_{\bar{I}^{z,T}} s_{\bar{I}^{z,T}}},
		\\
		\label{eq:aux20}
		\nabla B^{z, T}  &=  - \frac{\bar{H}_{\bar{I}^{z,T}}^2}{1 - \bar{H}_{\bar{I}^{z,T}} s_{\bar{I}^{z,T}}} 
		\bar{\tau}_{\bar{I}^{z,T}} \otimes \bar{\tau}_{\bar{I}^{z,T}}
		+ \frac{\bar{H}'_{\bar{I}^{z,T}}}{1 - \bar{H}_{\bar{I}^{z,T}} s_{\bar{I}^{z,T}}} 
		\bar{\n}_{\bar{I}^{z,T}} \otimes \bar{\tau}_{\bar{I}^{z,T}},
		\\
		\label{eq:aux21}
		\nabla \cdot B^{z, T}  &=  - \frac{\bar{H}_{\bar{I}^{z,T}}^2}{1 - \bar{H}_{\bar{I}^{z,T}} s_{\bar{I}^{z,T}}},
		\\
		\label{eq:aux22}
		\nabla \vartheta^{z, T}  &=  - \frac{1}{r_T^2}\bar{\n}_{\bar{I}^{z,T}}.
	\end{align}
	Note that these computations are justified thanks to
	$|1{-}\bar{H}_{\bar{I}^{z,T}} s_{\bar{I}^{z,T}}| \geq 1/2$ being valid throughout $\{\dist(\cdot,\bar{I}^{z,T})<r_T/8\}$, which in turn follows from 
	$\bar{H}_{\bar{I}^{z,T}} \leq 2/r_T$ since, by assumption, $r_T/2$ is an admissible tubular neighborhood width for~$\bar{I}^{z,T}$
	(cf.\ Definition~\ref{DefinitionShrinkingCircle}). 
	Within~$\{\dist(\cdot,\bar{I}^{z,T})<r_T/8\}$, we also record the following
	simplifications of~\eqref{eq:PDEshiftedXi} and~\eqref{eq:PDEWeightShifted}:
	\begin{align}
		\label{eq:aux23}
		\partial_t\xi^{z,T} &= - \big(1{+}\dot{\mathfrak{T}}\big) \frac{H'_{\bar{I}^{z,T}} \circ P_{\bar{I}^{z,T}}}
		{1 - \bar{H}_{\bar{I}^{z,T}} s_{\bar{I}^{z,T}}} \bar{\tau}_{\bar{I}^{z,T}}
		+ \big(\bar{\tau}_{\bar{I}^{z,T}} \cdot \dot{z}\big)
		\frac{\bar{H}_{\bar{I}^{z,T}}}{1 - \bar{H}_{\bar{I}^{z,T}} s_{\bar{I}^{z,T}}} \bar{\tau}_{\bar{I}^{z,T}},
		\\
		\label{eq:aux24}
		\partial_t\vartheta^{z,T} &= -\frac{1{+}\dot{\mathfrak{T}}}{r_T^2}
		\Big(2\frac{s_{\bar{I}^{z,T}}}{r_T^2} - \bar{H}_{\bar{I}^{z,T}}\Big)
		+ \frac{\bar{\n}_{\bar{I}^{z,T}} \cdot \dot{z}}{r_T^2}. 
	\end{align}
	
	\textit{Step~2: Identities for geometric quantities of perturbed interface.}
	First, we define $(\mathfrak{h},\mathfrak{h}',\mathfrak{h}''):=(h,h',h'')\circ P_{\bar{I}^{z,T}}$. 
	Then, denoting by $o(1)$ any quantity 
	$f(\bar{H}_{\bar{I}^{z,T}} \mathfrak{h}, \mathfrak{h}')$ such that
	$f\colon\mathbb{R}\times\mathbb{R}\to\mathbb{R}$ is a continuous function satisfying
	$f(x_1,x_2) \rightarrow 0$  whenever $|(x_1,x_2)| \rightarrow 0$,
	we claim that along~$I$
	\begin{align}
		\label{eq:aux25}
		\n &= \bigg( 1 + \Big( - \frac12 + o(1) \Big) 
		(\mathfrak{h}')^2\bigg) \bar{\n}_{\bar{I}^{z,T}}
		- \big(1 {+} o(1)\big) \mathfrak{h}'\bar{\tau}_{\bar{I}^{z,T}},
		\\
		\label{eq:aux26}
		V &= \frac{\bar{H}_{\bar{I}^{z,T}}}{1 -\bar{H}_{\bar{I}^{z,T}}\mathfrak{h}} + \mathfrak{h}'' + 
		o(1) \mathfrak{h}'' + o(1) \mathfrak{h}'\bar{H}_{\bar{I}^{z,T}} + o(1) \mathfrak{h}\bar{H}_{\bar{I}^{z,T}}'.
	\end{align}

	For a proof of~\eqref{eq:aux25}--\eqref{eq:aux26}, it is computationally convenient to represent the interface~$I$
	as the image of the curve $\gamma_h := (\mathrm{id}+h\bar{\n}_{\bar{I}^{z,T}})\circ\bar{\gamma}^{z,T}$,
	where $\bar{\gamma}$ is an arc-length parametrization of $\bar{I}(T^{-1}(\cdot))$ such that
	$\bar{\tau}_{\bar{I}^{z,T}}\circ\bar{\gamma}^{z,T} = (\bar{\gamma}^{z,T})'$. Then
	\begin{align}
		\label{eq:aux27}
		\gamma_h' = \bigg(\big(1 {-} \bar{H}_{\bar{I}^{z,T}}h\big)
		\bar{\tau}_{\bar{I}^{z,T}} + h'\bar{\n}_{\bar{I}^{z,T}}\bigg)\circ\bar{\gamma}^{z,T},
	\end{align}
	hence (recall that $J \in \mathbb{R}^{2\times 2}$ denotes counter-clockwise rotation by $90^\circ$)
	\begin{align}
		\label{eq:aux28}
		\n \circ {\gamma_h} = J\frac{\gamma_h'}{|\gamma_h'|} = 
		\bigg(\frac{\big(1 {-} \bar{H}_{\bar{I}^{z,T}}h\big)
			\bar{\n}_{\bar{I}^{z,T}} - h'\bar{\tau}_{\bar{I}^{z,T}}}{\sqrt{\big(1 {-} \bar{H}_{\bar{I}^{z,T}}h\big)^2+(h')^2}}\bigg) \circ \bar{\gamma}^{z,T},
	\end{align}
	so that~\eqref{eq:aux25} follows from Taylor expansion
	with respect to the variables $\bar{H}_{\bar{I}^{z,T}}h$ and $h'$.
	By virtue of $H^2$~regularity of the height function~$h$
	and~$V$ being the distributional curvature of~$I$ due to~\cite[Definition~13, item~iii)]{FischerHenselLauxSimon}
	and~$\chi_i\equiv 0$ for all $i \notin\{1,P\}$, we deduce $V=H_{\gamma_h}$. In other words,
	\begin{equation}
		\label{eq:aux29}
		\begin{aligned}
			V \circ \gamma_h &= \frac{\gamma_h''\cdot J\gamma_h'}{|\gamma_h'|^3}
			\\&
			= \bigg(\frac{h'\big(\bar{H}_{\bar{I}^{z,T}}'h+2\bar{H}_{\bar{I}^{z,T}}h'\big)+ 
				(1 {-} \bar{H}_{\bar{I}^{z,T}}h)\big(h''+\bar{H}_{\bar{I}^{z,T}}(1{-}\bar{H}_{\bar{I}^{z,T}}h)\big)}
			{\sqrt{\big(1 {-} \bar{H}_{\bar{I}^{z,T}}h\big)^2+(h')^2}^3}\bigg) \circ \bar{\gamma}^{z,T},
		\end{aligned}
	\end{equation}
	so that~\eqref{eq:aux26} again follows from Taylor expansion
	with respect to the variables $\bar{H}_{\bar{I}^{z,T}}h$, $h'$ and $h''$.
	
	\textit{Step~3: Change of variables formula.}
	Let $g\colon I\to\mathbb{R}$ be integrable. Then, by the coarea formula
	\begin{align}
		\label{eq:aux30}
		\int_{I} g \,d\mathcal{H}^1 = \int_{\bar{I}^{z,T}} g \circ (\mathrm{id} {+}h\n_{\bar{I}^{z,T}})
		\sqrt{\big(1 {-} H_{\bar{I}^{z,T}}h\big)^2+(h')^2} \,d\mathcal{H}^1.
	\end{align}
	Furthermore, since the Jacobian of the tubular neighborhood diffeomorphism
	$x \mapsto (P_{\bar{I}^{z,T}}(x),s_{\bar{I}^{z,T}})$ is given by $1/(1 {-} \bar{H}_{\bar{I}^{z,T}} s_{\bar{I}^{z,T}})$,
	we also obtain for any integrable $G\colon\mathbb{R}^2\to\mathbb{R}$ with $\supp G\subset \{\dist(\cdot,\bar{I}^{z,T})<r_T/4\}$ 
	by the area formula
 and the assumed conclusions of Proposition~\ref{theo:graph}
	\begin{align}
		\label{eq:aux31}
		\int_{\mathbb{R}^2} (\chi - \bar\chi^{z,T}) G \,dx = - \int_{\bar{I}^{z,T}} \int_{0}^{h} 
		\frac{G\big(x{+}s\bar{\n}_{\bar{I}^{z,T}}(x)\big)}{1 {-} \bar{H}_{\bar{I}^{z,T}}(x)s} \,dsd\mathcal{H}^1(x).
	\end{align}
	
	\textit{Step~4: Collecting further auxiliary identities.}
	We obtain from~\eqref{eq:aux25}
	and~\eqref{eq:aux15} that, along~$I$,
	\begin{align}
		\label{eq:aux32}
		\n - \xi^{z,T}
		&= 
		o(1) \mathfrak{h}' \bar{\n}_{\bar{I}^{z,T}}
		- \big(1 {+} o(1)\big) \mathfrak{h}'\bar{\tau}_{\bar{I}^{z,T}},
		\\
		\label{eq:aux33}
		1 - \n\cdot\xi^{z,T}
		&= \Big(\frac12 + o(1) \Big) (\mathfrak{h}')^2.
	\end{align}
	In addition, we infer from~\eqref{eq:aux26} and~\eqref{eq:aux19}
	that, along~$I$,
	\begin{align}
		\label{eq:aux33a}
		V + \nabla\cdot\xi^{z,T} = \mathfrak{h}'' + 
		o(1) \mathfrak{h}'' + o(1) \mathfrak{h}'\bar{H}_{\bar{I}^{z,T}} + o(1) \mathfrak{h}\bar{H}_{\bar{I}^{z,T}}',
	\end{align}
	as well as from~\eqref{eq:aux25}, \eqref{eq:aux26} and~\eqref{eq:aux15}--\eqref{eq:aux16} 
	\begin{equation}
		\begin{aligned}
			\label{eq:aux33b} 
			&V\n -(B^{z,T}\cdot\xi^{z,T})\xi^{z,T} 
			\\&~~~
			= V\n - B^{z,T}		
			\\&~~~
			= \big(\bar{H}^2_{\bar{I}^{z,T}}\mathfrak{h}{+}\mathfrak{h}''\big)\bar{\n}_{\bar{I}^{z,T}}
			- 	 \bar{H}_{\bar{I}^{z,T}}\mathfrak{h}'\bar{\tau}_{\bar{I}^{z,T}}
			\\&~~~~~~
			+ \big(o(1)\mathfrak{h}'' + o(1)\bar{H}_{\bar{I}^{z,T}}\mathfrak{h}'
			+o(1)(\bar{H}'_{\bar{I}^{z,T}}\circ P_{\bar{I}^{z,T}})\mathfrak{h}
			+o(1) \bar{H}_{\bar{I}^{z,T}}^2\mathfrak{h}\big) \bar{\n}_{\bar{I}^{z,T}}
			\\&~~~~~~
			+ \big(o(1)\mathfrak{h}'' + o(1)\bar{H}_{\bar{I}^{z,T}}\mathfrak{h}'
			+o(1)(\bar{H}'_{\bar{I}^{z,T}}\circ P_{\bar{I}^{z,T}})\mathfrak{h}
			+o(1) \bar{H}_{\bar{I}^{z,T}}^2\mathfrak{h}\big) \bar{\tau}_{\bar{I}^{z,T}}
		\end{aligned}
	\end{equation}
	and
	\begin{align}
		\label{eq:aux33c}
		B^{z,T}\cdot\xi^{z,T} - V &= - \bar{H}_{\bar{I}^{z,T}}^2 \mathfrak{h} - \mathfrak{h}'' +
		o(1) \mathfrak{h}'' + o(1) \mathfrak{h}'\bar{H}_{\bar{I}^{z,T}} + o(1) \mathfrak{h}\bar{H}_{\bar{I}^{z,T}}'
		+ o(1) \mathfrak{h} \bar{H}_{\bar{I}^{z,T}}^2.
	\end{align}
	
	Next, we exploit the information gathered so far to express
	the terms originating from the stability estimate of~$E_{\mathrm{int}}$
	in terms of the height function~$h$ (and its derivatives).
	First, we get from combining~\eqref{eq:aux30}, \eqref{eq:aux23}, \eqref{eq:aux15}--\eqref{eq:aux16},
	\eqref{eq:aux18}, \eqref{eq:aux20} and \eqref{eq:aux32},
	\begin{equation}
		\label{eq:aux34}
		\begin{aligned}
			&- \int_{I} \big(\partial_t\xi^{z,T} {+} (B^{z,T}\cdot\nabla)\xi^{z,T} {+} (\nabla B^{z,T})^\mathsf{T}\xi^{z,T}\big)
			\cdot (\n{-}\xi^{z,T}) \,d\mathcal{H}^1
			\\&~~~
			\leq \int_{\bar{I}^{z,T}} H_{\bar{I}^{z,T}} h' (\tau_{\bar{I}^{z,T}}\cdot\dot{z})  \,d\mathcal{H}^1
			- \int_{\bar{I}^{z,T}} H_{\bar{I}^{z,T}}' \dot{\mathfrak{T}} h' \,d\mathcal{H}^1 
			\\&~~~~~~
			+ \int_{\bar{I}^{z,T}} |o(1)| \Big(\big|H_{\bar{I}^{z,T}} (\tau_{\bar{I}^{z,T}}\cdot\dot{z}) h'\big|
			{+} \big|H_{\bar{I}^{z,T}}' \dot{\mathfrak{T}} h'\big|\Big) \,d\mathcal{H}^1
			\\&~~~
			= - \int_{\bar{I}^{z,T}} H_{\bar{I}^{z,T}}^2 h (\n_{\bar{I}^{z,T}}\cdot\dot{z}) \,d\mathcal{H}^1
			\\&~~~~~~
			- \int_{\bar{I}^{z,T}} H_{\bar{I}^{z,T}}' (\tau_{\bar{I}^{z,T}}\cdot\dot{z}) h \,d\mathcal{H}^1 
			- \int_{\bar{I}^{z,T}} H_{\bar{I}^{z,T}}' \dot{\mathfrak{T}} h' \,d\mathcal{H}^1 
			\\&~~~~~~
			+ \int_{\bar{I}^{z,T}} |o(1)| \Big(\big|H_{\bar{I}^{z,T}} (\tau_{\bar{I}^{z,T}}\cdot\dot{z}) h'\big|
			{+} \big|H_{\bar{I}^{z,T}}' \dot{\mathfrak{T}} h'\big|\Big) \,d\mathcal{H}^1,
		\end{aligned}
	\end{equation}
	where in the last step we also integrated by parts.
	Next, it directly follows from~\eqref{eq:aux30}, \eqref{eq:aux15}--\eqref{eq:aux16} and~\eqref{eq:aux19}
	\begin{align}
		\label{eq:aux35}
		\int_{I} \frac{1}{2} \big|\nabla\cdot\xi^{z,T} {+} B^{z,T}\cdot\xi^{z,T}\big|^2 \,d\mathcal{H}^1
		\leq \int_{\bar{I}^{z,T}} \big(1 {+} |o(1)|\big) \frac{1}{2} H_{\bar{I}^{z,T}}^4 h^2 \,d\mathcal{H}^1,
	\end{align}
	and exploiting in addition~\eqref{eq:aux33}
	\begin{equation}
		\label{eq:aux36}
		\begin{aligned}
			&-\int_{I} (1 - \n\cdot\xi^{z,T}) \nabla\cdot\xi^{z,T}
			(B^{z,T}\cdot\xi^{z,T}) \,d\mathcal{H}^1
			\\&~~~
			\leq \int_{\bar{I}^{z,T}} \big(1 {+} |o(1)|\big) \frac{1}{2} H_{\bar{I}^{z,T}}^2 (h')^2 \,d\mathcal{H}^1.
		\end{aligned}
	\end{equation}
	Analogously, recalling also~\eqref{eq:aux20} and~\eqref{eq:aux21},
	\begin{align}
		\label{eq:aux37}
		\int_{I} (1 - \n\cdot\xi^{z,T}) \nabla\cdot B^{z,T} \,d\mathcal{H}^1
		\leq - \int_{\bar{I}^{z,T}} \big(1 {-} |o(1)|\big) \frac{1}{2} H_{\bar{I}^{z,T}}^2 (h')^2 \,d\mathcal{H}^1
	\end{align}
	as well as
	\begin{equation}
		\begin{aligned}
			\label{eq:aux38}
			&-\int_{I} (\n {-} \xi^{z,T}) \otimes (\n {-} \xi^{z,T}) : \nabla B^{z,T} \,d\mathcal{H}^1
			\\&~~~
			\leq \int_{\bar{I}^{z,T}} \big(1 {+} |o(1)|\big) H_{\bar{I}^{z,T}}^2 (h')^2 \,d\mathcal{H}^1
			+ \int_{\bar{I}^{z,T}} |o(1)| \big|H_{\bar{I}^{z,T}}'\big| (h')^2 \,d\mathcal{H}^1.
		\end{aligned}
	\end{equation}
	Just plugging in~\eqref{eq:aux33a} and estimating by Young's inequality yields
	\begin{equation}
		\label{eq:aux39}
		\begin{aligned}
			&-\int_{I} \frac{1}{2} \big|V{+}\nabla\cdot\xi^{z,T}\big|^2 \,d\mathcal{H}^1
			\\&~~~
			\leq - \int_{\bar{I}^{z,T}} \big(1 {-} |o(1)|\big) \frac{1}{2}(h'')^2 \,d\mathcal{H}^1
			\\&~~~~~~
			+ \int_{\bar{I}^{z,T}} |o(1)| \Big(H_{\bar{I}^{z,T}}^2(h')^2
			+ \big((H_{\bar{I}^{z,T}}')^2 {+} H_{\bar{I}^{z,T}}^4\big) h^2\Big) \,d\mathcal{H}^1,
		\end{aligned}
	\end{equation}
	and analogously based on~\eqref{eq:aux33b}
	\begin{equation}
		\label{eq:aux40}
		\begin{aligned}
			&-\int_{I} \frac{1}{2} \big|V{-}(B^{z,T}\cdot\xi^{z,T})\xi^{z,T}\big|^2 \,d\mathcal{H}^1
			\\&~~~
			\leq - \int_{\bar{I}^{z,T}} \big(1 {-} |o(1)|\big) \frac{1}{2}
			\Big((h'')^2 + H_{\bar{I}^{z,T}}^2 (h')^2 + H_{\bar{I}^{z,T}}^4 h^2\Big) \,d\mathcal{H}^1
			\\&~~~~~~
			- \int_{\bar{I}^{z,T}} H_{\bar{I}^{z,T}}^2 h h'' \,d\mathcal{H}^1
			\\&~~~~~~
			+ \int_{\bar{I}^{z,T}} |o(1)| (H_{\bar{I}^{z,T}}')^2h^2\,d\mathcal{H}^1
			\\&~~~
			\leq - \int_{\bar{I}^{z,T}} \big(1 {-} |o(1)|\big) \frac{1}{2}
			\Big((h'')^2 + H_{\bar{I}^{z,T}}^4 h^2\Big) \,d\mathcal{H}^1
			\\&~~~~~~
			+ \int_{\bar{I}^{z,T}} \frac{1}{2} H_{\bar{I}^{z,T}}^2 (h')^2 \,d\mathcal{H}^1
			\\&~~~~~~
			+ \int_{\bar{I}^{z,T}} 2 H_{\bar{I}^{z,T}} H_{\bar{I}^{z,T}}' h h' \,d\mathcal{H}^1
			\\&~~~~~~
			+ \int_{\bar{I}^{z,T}} |o(1)| \Big(H_{\bar{I}^{z,T}}^2 (h')^2 + (H_{\bar{I}^{z,T}}')^2 h^2 \Big)\,d\mathcal{H}^1,
		\end{aligned}
	\end{equation}
	where in the last step we also performed an integration by parts
	and estimated by Young's inequality.
	
	We continue with the terms originating from the stability estimate of $E_{\mathrm{bulk}}$.
	First, by means of~\eqref{eq:aux30}, \eqref{eq:aux17}, \eqref{eq:aux33c}
	and integration by parts we obtain
	\begin{equation}
		\label{eq:aux41}
		\begin{aligned}
			&\int_{I} \vartheta^{z,T} (B^{z,T}\cdot\xi^{z,T} {-} V) \,d\mathcal{H}^1
			\\&~~~
			\leq \int_{\bar{I}^{z,T}} \frac{H_{\bar{I}^{z,T}}^2}{r_T^2} h^2 \,d\mathcal{H}^1
			- \int_{\bar{I}^{z,T}} \frac{1}{r_T^2} (h')^2 \,d\mathcal{H}^1
			\\&~~~~~~
			+ \int_{\bar{I}^{z,T}} |o(1)| \bigg( \Big(\frac{H_{\bar{I}^{z,T}}^2}{r_T^2}
			{+} \frac{1}{r_T^4} {+} (H_{\bar{I}^{z,T}}')^2\Big) h^2
			+  H_{\bar{I}^{z,T}}^2 (h')^2 + (h'')^2 \bigg) \,d\mathcal{H}^1.
		\end{aligned}
	\end{equation}
	Next, just plugging in~\eqref{eq:aux16}--\eqref{eq:aux17} and~\eqref{eq:aux32}
	into~\eqref{eq:aux30} and applying Young's inequality entails
	\begin{align}
		\label{eq:aux42}
		&\int_{I} \vartheta^{z,T} B^{z,T}\cdot (\n - \xi^{z,T}) \,d\mathcal{H}^1
		\leq \int_{\bar{I}^{z,T}} |o(1)| \Big(\frac{1}{r_T^4}h^2 + H_{\bar{I}^{z,T}}^2(h')^2\Big) \,d\mathcal{H}^1.
	\end{align}
	In addition, based on~\eqref{eq:aux31}, \eqref{eq:aux17}
	and~\eqref{eq:aux21}, we may infer
	\begin{align}
		\label{eq:aux43}
		\int_{\mathbb{R}^2} (\chi {-} \bar\chi^{z,T}) \vartheta^{z,T} \nabla\cdot B^{z,T} \,dx
		\leq - \int_{\bar{I}^{z,T}} \big(1 {-} |o(1)|\big) \frac{1}{2}\frac{H_{\bar{I}^{z,T}}^2}{r_T^2} h^2 \,d\mathcal{H}^1,
	\end{align}
	whereas it finally follows from~\eqref{eq:aux31}, \eqref{eq:aux24}, \eqref{eq:aux16}
	and~\eqref{eq:aux22}
	\begin{equation}
		\label{eq:aux44}
		\begin{aligned}
			&\int_{\mathbb{R}^2} (\chi {-} \bar\chi^{z,T})\big(\partial_t\vartheta^{z,T} {+} (B^{z,T}\cdot\nabla)\vartheta^{z,T}\big) \,dx
			\\&~~~
			\leq \int_{\bar{I}^{z,T}} \frac{1}{r_T^4} h^2 \,d\mathcal{H}^1
			- \int_{\bar{I}^{z,T}} \frac{H_{\bar{I}^{z,T}}}{r_T^2} h \dot{\mathfrak{T}} \,d\mathcal{H}^1
			- \int_{\bar{I}^{z,T}} \frac{1}{r_T^2} h (\n_{\bar{I}^{z,T}} \cdot \dot{z}) \,d\mathcal{H}^1
			\\&~~~~~~
			+ \int_{\bar{I}^{z,T}} |o(1)| \bigg( \frac{1}{r_T^3}\big|\dot{\mathfrak{T}}h\big| {+} 
			\frac{1}{r_T^2}\big|H_{\bar{I}^{z,T}}\dot{\mathfrak{T}}h\big| {+} \frac{1}{r_T^2}\big|(\n_{\bar{I}^{z,T}} \cdot \dot{z})h\big|
			{+} \frac{1}{r_T^4} h^2 \bigg)  \,d\mathcal{H}^1.
		\end{aligned}
	\end{equation}
	
	\textit{Step 5: Conclusion.} Due to~\eqref{eq:aux15} and~\eqref{eq:aux16},
	the identities~\eqref{eq:aux2}, \eqref{eq:aux4} and~\eqref{eq:aux6}
	hold true for trivial reasons. The remaining estimates
	follow from~\eqref{eq:aux34}--\eqref{eq:aux44} and $|H_{\bar{I}^{z,T}}| \leq 2/r_T$.
\end{proof}

\appendix

\section{Auxiliary results from geometric measure theory} \label{appendix}
In this appendix, we recall the two main ingredients from geometric measure theory which we use to prove Proposition~\ref{theo:graph}, namely the Allard's regularity theory for integer rectifiable varifolds \cite[Chapter 5, Theorem 23.1 and Remark 23.2(a)]{SimonLectures} and the decoposition of a reduced boundary of a set of finite
perimeter in $\mathbb{R}^2$ into rectifiable Jordan-Lipschitz curves \cite[Section 6, Theorem 4]{Ambrosio}.

Using the notation from Definition \ref{DefinitionVarSolution} and omitting the dependence in time, we recall the usual definition of tilt-excess $E_{\operatorname{tilt}}$ for an integer rectifiable 1-varifold $\mathcal{V}$ with associated mass measure $\mu$, namely
\[
E_{\operatorname{tilt}}[x_0, \varrho, G] = \varrho^{-1} \int_{B_{\varrho}(x_o)} |P_{\operatorname{Tan}_{x_0} (\supp \mu )} - P_{G}|^2 \,  \mathrm{d} \mu \,, 
\]
where $x_0 \in \supp \mu $, $\varrho >0$, and $G$ is a one dimensional subspace of $\mathbb{R}^2$, whreas $P$ denotes the projection onto the approximate tangent space ${\operatorname{Tan}_{x_0} (\supp \mu )}$ and onto the given subspace $G$, respectively.
In our setting, Allard's regularity theorem \cite[Chapter 5, Theorem 23.1 and Remark 23.2(a)]{SimonLectures} 
reads as follows.

\begin{theorem}[Allard's regularity theory]\label{theo:Allard}
	Fix $\varrho>0$, $p>1$, $x_0 \in \supp \mu $, and 
	a one-dimensional subspace $G$ of $\mathbb{R}^2$. 
	There exist $\varepsilon = \varepsilon (p), \, \gamma=\gamma(p) \in (0,1)$ 
	and $C_{Allard} = C_{Allard}(p)>0$ such that:
	if 
	\begin{align}
		\frac{\mu(B_{\varrho}(x_0))}{\mathrm{Vol}_1  \varrho}\leq 1+\varepsilon,
	\end{align} 
where $\mathrm{Vol}_1$ denotes the one-dimensional volume of the unit ball,
     and
     	\begin{align}\label{eq:hpAllard}
     	\max\bigg\{ E_{\operatorname{tilt}}[x_0, \varrho, G], \varepsilon^{-1} \Big(\int_{B_{\varrho}(x_0)} |H_\mu|^p \, \dx\Big)^{\frac2p} \varrho^{2(1-\frac1p)}\bigg\} \leq \varepsilon , 
     \end{align}
      then there exists a $C^{1, 1-\frac1p}$ function $u\colon (x_0 + G) \cap B_{\gamma \varrho}(x_0) \rightarrow \mathbb{R}$ satisfying: 
      \begin{itemize}
      	\item[i)] $u(x_0)=0$;
      	      	\item[ii)] 
      	      	$	\supp \mu \cap B_{\gamma \varrho }(x_0) = B_{\gamma \varrho }(x_0) \cap \big\{ y + u(y) \n_{G}(y) : y \in (x_0 + G ) \cap B_{\gamma \varrho}(x_0)\big\},
      	      $ 
            	where $\n_{G}$ is the normal vector field to the affine space $x_0 + G$;
      	      	      	\item[iii)] it holds 
      	      	      	\begin{align}\notag 
      	      	      	&	\varrho^{-1} \sup_{x  }|u(x)| + \sup_x |\nabla u(x)| + \varrho^{1-\frac1p} \sup_{x \neq y} \big\{|x-y|^{-(1-\frac1p)} |\nabla u(x)- \nabla u (y)|\big\}\\ 
      	      	      		&\leq C_{Allard} \biggl[ E^\frac12 _{\operatorname{tilt}}[x_0, \varrho, G]
      	      	      		+ \Big(\int_{B_{\varrho}(x_0)} |H_\mu|^p \, \dx\Big)^{\frac1p} \varrho^{1-\frac1p} \biggr]. \label{eq:Allard}
      	      	      	\end{align}
      \end{itemize}
\end{theorem}

\begin{figure}
	\includegraphics{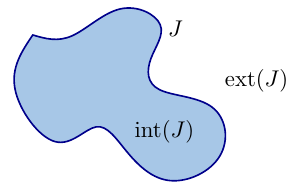}
	\caption{Jordan curve $J$}
		\label{Fig:Jordan}
\end{figure} 
The second result makes use of the notion of Jordan-Lipschitz curves, for which
we refer the reader to~\cite[Section 8]{Ambrosio}, and states a decomposition
for the reduced boundary of a planar set of finite perimeter
(see Figures~\ref{Fig:Jordan} and~\ref{Fig:decomposition} for an example) as given
by~\cite[Section 8, Corollary~1]{Ambrosio}.
\begin{figure}
	\includegraphics{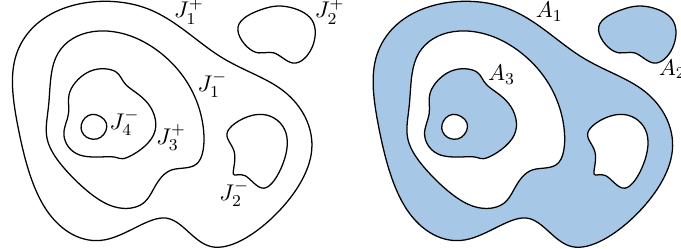}
	\caption{Decomposition of the reduced boundary of a planar set of finite perimeter into Jordan-Lipschitz curves $\{J^+_i, J^-_k: i,k \in \mathbb{N}\}$.}
	\label{Fig:decomposition}
\end{figure}

\begin{theorem}[Decomposition result for planar sets of finite perimeters]\label{theo:Jordan}
	Let $\chi \in BV(\mathbb{R}^2; \{0,1\})$ be the indicator function associated to a set of finite perimeter.
	Then, $\supp |\nabla \chi|$ can be uniquely decomposed into a countable family of Jordan-Lipschitz curves $\{J^+_i, J^-_k: i,k \in \mathbb{N}\}$ such that the following properties hold:
	\begin{itemize}
		\item[i)]Let $i\neq k$, then either $\operatorname{int}(J_i^+)\cap\operatorname{int}(J_k^+)= \emptyset$ or $\operatorname{int}(J_i^+) \subseteq \operatorname{int}(J^+_k)$ and analogously for $\{J_i^-: i \in \mathbb{N}\}$. Furthermore, for each $i \in \mathbb{N}$ there exists $k= k(i)\in \mathbb{N}$ such that $\operatorname{int}(J_i^-) \subseteq \operatorname{int}(J^+_k)$.
		\item[ii)] $\mathcal{H}^1(\supp |\nabla \chi|) = \sum_i  \mathcal{H}^1 (J_i^+)+ \sum_k \mathcal{H}^1 (J_k^-)$.
		\item[iii)] If $\operatorname{int}(J_i^+) \subseteq \operatorname{int}(J_k^+)$, there exists $j \in \mathbb{N}$ such that $\operatorname{int}(J_i^+) \subseteq \operatorname{int}(J_j^-) \subseteq \operatorname{int}(J_k^+)$, and analogously for the roles of $\{J_i^+: i \in \mathbb{N}\}$ and $\{J_k^-: i \in \mathbb{N}\}$ switched.
		\item[iv)] Setting $L_k := \{i \in \mathbb{N} : \operatorname{int}(J_i^-)\subseteq \operatorname{int}(J_k^+) \}$, then the sets $A_k := \operatorname{int}(J_k^+) \setminus \cup_{i \in L_k} \operatorname{int} (J_i^-)$ are piecewise disjoint, indecomposable, and $\chi = \sum_k \chi_{A_k }$.
	\end{itemize}
\end{theorem}

\section*{Acknowledgements}
This project has received funding from the
European Research Council (ERC) under the European Union's Horizon 2020
research and innovation programme (grant agreement No 948819)
\smash{
\begin{tabular}{@{}c@{}}\includegraphics[width=6ex]{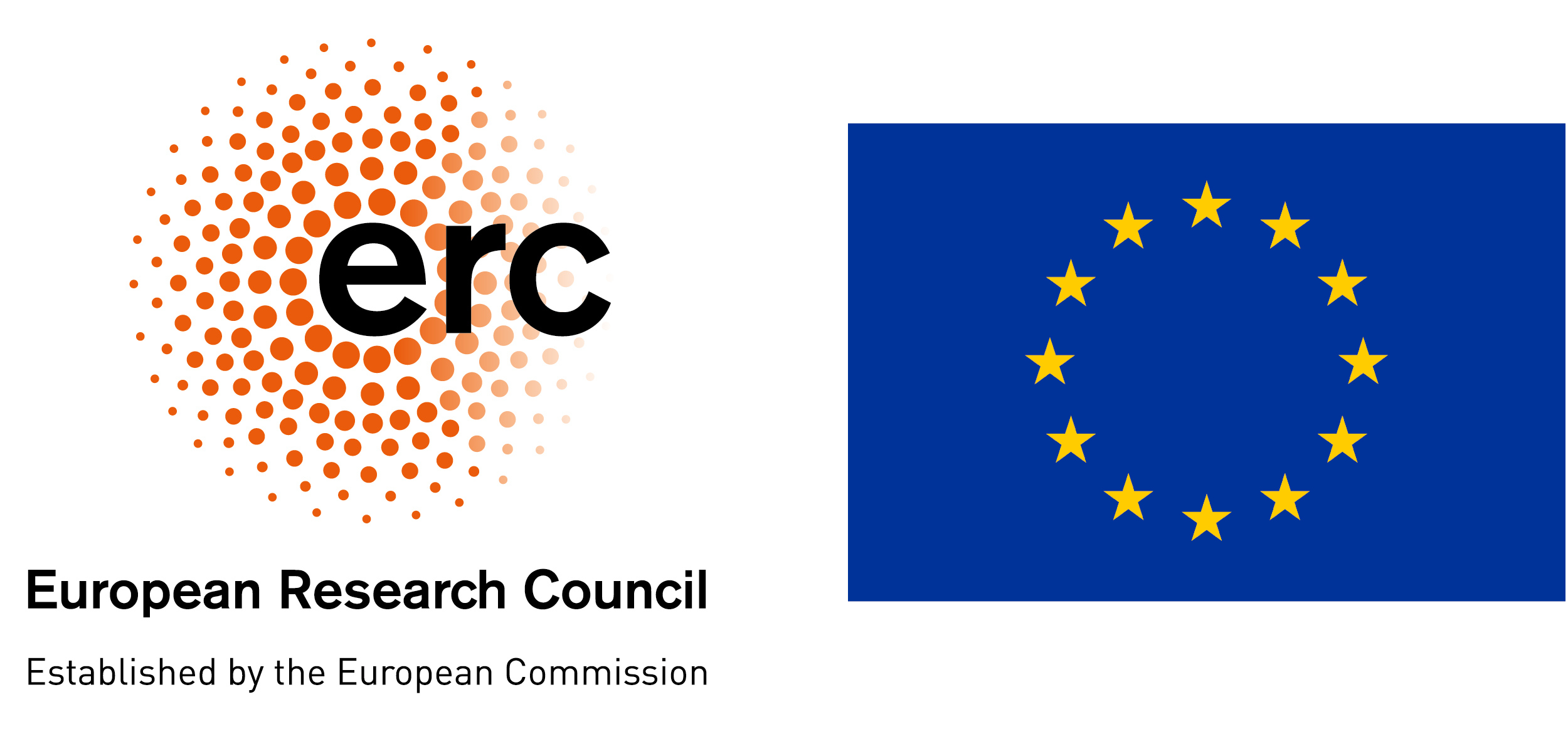}\end{tabular}
},
and from the Deutsche Forschungsgemeinschaft (DFG, German Research Foundation) 
under Germany's Excellence Strategy -- EXC-2047/1 -- 390685813.

Parts of this work were written during the visit of A.M. to the Hausdorff Research Institute for Mathematics (HIM), University of Bonn, in the framework of the trimester program ``Mathematics for complex materials''. The support and the hospitality of HIM are gratefully acknowledged.

\bibliographystyle{abbrv}
\bibliography{MCF_stability_circle}

\begin{thebibliography}{10}

\bibitem{AbelsFischerMoser}
H.~Abels, J.~Fischer, and M.~Moser.
\newblock Approximation of classical two-phase flows of viscous incompressible
  fluids by a {Navier-Stokes/Allen-Cahn} system.
\newblock {\em Preprint}, 2024.
\newblock \href{https://arxiv.org/abs/2311.02997}{arXiv:2311.02997}.

\bibitem{Ambrosio}
L.~Ambrosio, V.~Caselles, S.~Masnou, and J.-M. Morel.
\newblock Connected components of sets of finite perimeter and applications to
  image processing.
\newblock {\em J. Eur. Math. Soc. (JEMS)}, 3(1):39--92, 2001.

\bibitem{BarlesSonerSouganidis}
G.~Barles, H.~M. Soner, and P.~E. Souganidis.
\newblock Front propagation and phase field theory.
\newblock {\em SIAM J. Control Optim.}, 31(2):439--469, 1993.

\bibitem{Brakke}
K.~A. Brakke.
\newblock {\em The motion of a surface by its mean curvature}, volume~20 of
  {\em Mathematical Notes}.
\newblock Princeton University Press, Princeton, N.J., 1978.

\bibitem{ChenGigaGoto}
Y.~G. Chen, Y.~Giga, and S.~Goto.
\newblock Uniqueness and existence of viscosity solutions of generalized mean
  curvature flow equations.
\newblock {\em J. Differential Geom.}, 33(3):749--786, 1991.

\bibitem{EvansSpruck}
L.~C. Evans and J.~Spruck.
\newblock Motion of level sets by mean curvature. {I}.
\newblock {\em J. Differential Geom.}, 33(3):635--681, 1991.

\bibitem{FischerHensel}
J.~Fischer and S.~Hensel.
\newblock Weak-strong uniqueness for the {N}avier-{S}tokes equation for two
  fluids with surface tension.
\newblock {\em Arch. Ration. Mech. Anal.}, 236(2):967--1087, 2020.

\bibitem{FischerHenselLauxSimonMS}
J.~Fischer, S.~Hensel, T.~Laux, and T.~Simon.
\newblock A weak-strong uniqueness principle for the {Mullins--Sekerka}
  equation.
\newblock {\em Preprint}, 2024.

\bibitem{FischerHenselLauxSimon}
J.~Fischer, S.~Hensel, T.~Laux, and T.~M. Simon.
\newblock The local structure of the energy landscape in multiphase mean
  curvature flow: {W}eak-strong uniqueness and stability of evolutions.
\newblock {\em first part accepted for publication at J. Eur. Math. Soc.
  (JEMS)}, 2024.
\newblock \href{https://arxiv.org/abs/2003.05478}{arXiv:2003.05478v2}.

\bibitem{FischerLauxSimon}
J.~Fischer, T.~Laux, and T.~M. Simon.
\newblock Convergence rates of the {A}llen--{C}ahn equation to mean curvature
  flow: {A} short proof based on relative entropies.
\newblock {\em SIAM J. Math. Anal.}, 52(6):6222--6233, 2020.

\bibitem{FischerMarveggio}
J.~Fischer and A.~Marveggio.
\newblock Quantitative convergence of the vectorial {A}llen-{C}ahn equation
  towards multiphase mean curvature flow.
\newblock {\em to appear in Ann. Inst. H. Poincar{\'e} Anal. Non Lin{\'e}aire},
  2022.
\newblock arXiv:2203.17143.

\bibitem{GageHamilton}
M.~Gage and R.~S. Hamilton.
\newblock The heat equation shrinking convex plane curves.
\newblock {\em J. Differential Geom.}, 23(1):69--96, 1986.

\bibitem{Grayson1987}
M.~A. Grayson.
\newblock The heat equation shrinks embedded plane curves to round points.
\newblock {\em J. Differential Geom.}, 26(2):285--314, 1987.

\bibitem{HenselLaux}
S.~Hensel and T.~Laux.
\newblock Weak-strong uniqueness for the mean curvature flow of double bubbles.
\newblock {\em Interfaces Free Bound.}, 25:37--107, 2023.

\bibitem{HenselMoser}
S.~Hensel and M.~Moser.
\newblock Convergence rates for the {A}llen-{C}ahn equation with boundary
  contact energy: the non-perturbative regime.
\newblock {\em Calc. Var.}, 61(201):61 pp., 2022.

\bibitem{KimTonegawa}
L.~Kim and Y.~Tonegawa.
\newblock On the mean curvature flow of grain boundaries.
\newblock {\em Ann. Inst. Fourier (Grenoble)}, 67(1):43--142, 2017.

\bibitem{LauxOtto}
T.~Laux and F.~Otto.
\newblock Convergence of the thresholding scheme for multi-phase mean-curvature
  flow.
\newblock {\em Calc. Var. Partial Differential Equations}, 55(5):1--74, 2016.

\bibitem{Laux2020}
T.~Laux and F.~Otto.
\newblock Brakke's inequality for the thresholding scheme.
\newblock {\em Calculus of Variations and Partial Differential Equations},
  59(1), 2020.

\bibitem{LauxSimon}
T.~Laux and T.~M. Simon.
\newblock Convergence of the {A}llen--{C}ahn equation to multiphase mean
  curvature flow.
\newblock {\em Comm. Pure Appl. Math.}, 71(8):1597--1647, 2018.

\bibitem{OhtaJasnowKawasaki}
T.~Ohta, D.~Jasnow, and K.~Kawasaki.
\newblock Universal scaling in the motion of random interfaces.
\newblock {\em Physical review letters}, 49(17):1223, 1982.

\bibitem{OsherSethian}
S.~Osher and J.~A. Sethian.
\newblock Fronts propagating with curvature-dependent speed: algorithms based
  on {H}amilton-{J}acobi formulations.
\newblock {\em J. Comput. Phys.}, 79(1):12--49, 1988.

\bibitem{SimonLectures}
L.~Simon.
\newblock {\em Lectures on geometric measure theory}, volume~3 of {\em
  Proceedings of the Centre for Mathematical Analysis, Australian National
  University}.
\newblock Australian National University, Centre for Mathematical Analysis,
  Canberra, 1983.

\bibitem{StuvardTonegawa}
S.~Stuvard and Y.~Tonegawa.
\newblock On the existence of canonical multi-phase {B}rakke flows.
\newblock {\em Advances in Calculus of Variations}, 17(1):33--78, 2022.

\end{thebibliography}

\end{document}